\documentclass{article}
\usepackage{graphicx}
\usepackage{amssymb, amsmath, amsthm,mathtools}
\usepackage{epstopdf}

\usepackage{bbm}

\usepackage[marginratio=1:1,height=600pt,width=460pt,tmargin=100pt]{geometry}

\usepackage{layout, color}
\usepackage{bbm}
\usepackage{enumerate}

\usepackage{algorithm}
\usepackage{algpseudocode}

\usepackage{setspace}
\usepackage{enumitem}

\usepackage{array}
\usepackage{multirow}

\usepackage{hyperref}

\usepackage{url}
\usepackage{mdwlist}
\usepackage{multirow}
\usepackage{amsthm}
\usepackage{color}

\usepackage[font=scriptsize,labelfont=bf]{caption}

\usepackage[labelfont=scriptsize,labelfont=scriptsize]{subcaption}

\newcommand{\calM}{{\mathcal{M}}}

\newcommand{\calH}{{ H}}

\newcommand{\calL}{{\mathcal{L}}}

\newcommand{\R}{\mathbb{R}} 

\newcommand{\E}{\mathbb{E}}

\newcommand{\Vol}{\mathrm{Vol}}

\newtheorem{theorem}{Theorem}[section]
\newtheorem{proposition}[theorem]{Proposition}
\newtheorem{assumption}{Assumption}

\newtheorem{lemma}[theorem]{Lemma}

\theoremstyle{remark}
\newtheorem{remark}{Remark}

\newtheorem{conjecture*}{Conjecture}

\theoremstyle{plain}

\usepackage{xcolor}

\newcommand{\rev}[1]{\textcolor{cyan}{#1}}

\begin{document}


\title{Eigen-convergence of Gaussian kernelized graph Laplacian by manifold heat interpolation}

\author{
Xiuyuan Cheng\thanks{Department of Mathematics, Duke University. Email: xiuyuan.cheng@duke.edu}
~~~~~~~~~~~~~~~~
Nan Wu\thanks{Department of Mathematics and Department of Statistical Science, Duke University. Email: nan.wu@duke.edu }
}

\date{
}

\maketitle

\begin{abstract}

\noindent
We study the spectral convergence of graph Laplacians to the Laplace-Beltrami operator when the kernelized graph affinity matrix is constructed from $N$ random samples on a $d$-dimensional manifold in an ambient Euclidean space. By analyzing Dirichlet form convergence and constructing candidate approximate eigenfunctions via convolution with manifold heat kernel, we prove eigen-convergence with rates as $N$ increases. The best eigenvalue convergence rate is $N^{-1/(d/2+2)}$ (when the kernel bandwidth parameter $\epsilon \sim (\log N/ N)^{1/(d/2+2)}$) and the best eigenvector 2-norm convergence rate is $N^{-1/(d/2+3)}$ (when $\epsilon \sim (\log N/N)^{1/(d/2+3)}$). These rates hold up to a $\log N$-factor for finitely many low-lying eigenvalues of both un-normalized and normalized graph Laplacians. When data density is non-uniform, we prove the same rates for the density-corrected graph Laplacian, and we also establish new operator point-wise convergence rate and Dirichlet form convergence rate as intermediate results. Numerical results are provided to support the theory.

\end{abstract}


\

\noindent
\rev{
This updated arXiv version is to correct a typo in the condition of Theorem \ref{thm:refined-rates-rw-density-correct} in the published version: 
X. Cheng and N. Wu. ``Eigen-convergence of Gaussian kernelized graph Laplacian by manifold heat interpolation". Applied and Computational Harmonic Analysis, 61, 132-190 (2022).}

\rev{
Specifically, the assumption of density $p$ is as in  Assumption \ref{assump:M-p}(A2), instead of assuming $p$ uniform. Section \ref{sec:density-corrected} is to handle non-uniform density $p$, and the proved rates are same as in the density uniform case, see Table \ref{tab:theory-summary}.}

\section{Introduction}

 \begin{table}[t]
 \small
   \caption{
\label{tab:notations}
List of default notations
} \vspace{-4pt}
 \begin{minipage}[t]{0.462\linewidth}
 \begin{tabular}{  p{0.5cm}  p{5.9cm}   }
 \hline
 ${\calM}$ 	& $d$-dimensional manifold in $\R^D$  	  \\ 
 $p$			& data sampling density on ${\calM}$  \\
 $\Delta_{\calM}$ & Laplace-Beltrami operator, also as $\Delta$ 		\\	
$\mu_k$      &  population eigenvalue of $-\Delta$ \\
$\psi_k$       & population eigenfunctions of $-\Delta$ \\
$\lambda_k$       & empirical eigenvalue of graph Laplacian\\
$v_k$       	          & empirical eigenvector of graph Laplacian  \\
$\nabla_{\calM}$ & manifold gradient, also as $\nabla$ 		\\	
 $\calH_t$		& manifold heat kernel \\
 $Q_t$		& semi-group operator of manifold diffusion, $Q_t = e^{t \Delta}$ \\
 $X$      		& dataset points used for computing  $W$ 		\\
$N$ 			& number of samples in $X$		\\
$\epsilon$ 	&  kernel bandwidth parameter 		\\
$K_\epsilon$ 	& graph affinity kernel, 
				$W_{ij} = K_\epsilon(x_i,x_j)$,
				$K_\epsilon(x,y)=\epsilon^{-d/2}h( \frac{\| x-y\|^2}{\epsilon})$ \\
$h$       		& a function  $[0,\infty) \to \R$			\\
 $m_0$		& $m_0[h]:=\int_{\R^d} h(|u|^2) du$ 	\\
 $m_2$		& $m_2[h]:= \frac{1}{d}\int_{\R^d} |u|^2 h(|u|^2) du$ 	\\
 $W$ 		 & kernelized graph affinity matrix 			\\ 
  \hline
\end{tabular}
\end{minipage}
 \begin{minipage}{0.45\linewidth}
 \begin{tabular}{  p{0.5cm}  p{7.2cm}   }
 \hline
 $D$			& degree matrix of $W$, $D_{ii} = \sum_{j=1}^N W_{ij}$\\
 $L_{un}$ 				& un-normalized graph Laplacian \\
 $L_{rw}$				& random-walk graph Laplacian \\
  $ E_N$				& graph Dirichlet form  \\
  $ \rho_X$			& function evaluation operator, $\rho_X f = \{ f(x_i) \}_{i=1}^N$ \\
   $\tilde{W}$ 			& density-corrected affinity matrix, $\tilde{W} = D^{-1}W D^{-1}$ 			\\
$\tilde{D}$				& degree matrix of $\tilde{W}$\\
  \hline
\end{tabular}
 \begin{tabular}{  p{0.5cm}  p{7.2cm}   }
  \hline
 \multicolumn{2}{ c }{Asymptotic Notations} \\
 \hline
$O(\cdot)$ & 		$f = O(g)$: $|f| \le C |g|$ in the limit, $C> 0$, 
				$O_a(\cdot)$ declaring the constant dependence on $a$	\\ 
$\Theta(\cdot)$ 	& $f = \Theta(g)$: for $f$, $g \ge 0$, $C_1 g \le f \le C_2 g$ in the limit, $C_1, C_2 >0$ \\
$\sim$			& $f \sim g$ same as $f = \Theta(g)$								\\
$o(\cdot)$ 		& $f = o(g)$: for $g > 0$, $|f|/g \to 0$ in the limit \\
$\Omega(\cdot)$ 	& $f = \Omega(g)$: for $f, g > 0$, $f/g \to \infty$ in the limit \\ 
$\tilde{O}(\cdot)$ 	& $O(\cdot)$ multiplied another factor involving a log, defined every time used in text 	\\			
 \multicolumn{2}{p{7.6cm}}{ 
~ When the superscript $_{a}$ is omitted,  it declares that the constants are absolute ones.
 
~ $f = O( g_1, g_2) $  means that $f =O(|g_1|+|g_2| )$. 
				 }\\
\hline
\end{tabular}
\end{minipage}
\end{table}

Graph Laplacian matrices built from data samples are widely used in data analysis and machine learning.
The earlier works include Isomap \cite{balasubramanian2002isomap}, 
Laplacian Eigenmap \cite{belkin2003laplacian},
Diffusion Map \cite{coifman2006diffusion,talmon2013diffusion},
among others.
Apart from being a widely-used unsupervised learning method for clustering analysis and dimension reduction (see, e.g., the review papers \cite{van2009dimensionality, talmon2013diffusion}),
 graph Laplacian methods also drew attention via the application in semi-supervised learning \cite{nadler2009semi,el2016asymptotic,slepcev2019analysis,flores2019algorithms}.
Under the manifold setting, 
data samples are assumed to lie on low-dimensional manifolds embedded in a possibly high-dimensional ambient space.
A fundamental problem is the convergence of the graph Laplacian matrix to the manifold Laplacian operator in the large sample limit.
The operator point-wise convergence has been intensively studied and established in a series of works \cite{hein2005graphs,hein2006uniform,belkin2007convergence,coifman2006diffusion,singer2006graph},
and extended to variant settings,
such as different kernel normalizations \cite{marshall2019manifold,wormell2021spectral} 
and general class of kernels \cite{ting2011analysis,berry2016variable,cheng2020convergence}.
The eigen-convergence,
namely how the empirical eigenvalues and eigenvectors converge to the population eigenvalues and eigenfunctions of the manifold Laplacian,
is a more subtle issue and has been studied in \cite{belkin2007convergence,von2008consistency,burago2014graph,wang2015spectral,singer2016spectral,eldridge2017unperturbed} (among others) and recently in \cite{trillos2020error,calder2019improved,dunson2021spectral,calder2020lipschitz}.

The current work proves the eigen-convergence, 
specifically the consistency of eigenvalues and eigenvectors in 2-norm, 
for finitely many low-lying eigenvalues of the graph Laplacian constructed using Gaussian kernel from i.i.d. sampled manifold data.
The result covers the un-normalized and random-walk graph Laplacian when data density is uniform,
and the density-corrected graph Laplacian (defined  below) with non-uniformly sampled  data.
For the latter, we also prove new point-wise and Dirichlet form convergence rates as an intermediate result.
We overview the main results in Section \ref{subsec:overview} in the context of literature,
which are also  summarized in Table \ref{tab:theory-summary}.

The framework of our work follows the variational principle formulation of eigenvalues using the graph and manifold Dirichlet forms.
Dirichlet form-based approach to prove graph Laplacian eigen-convergence was firstly carried out in \cite{burago2014graph}
under a non-probabilistic setting. 
\cite{trillos2020error,calder2019improved} extended the approach under  the probabilistic setting, where $x_i$ are i.i.d. samples,
using optimal transport techniques. 
Our analysis follows the same form-based approach and differs from previous works in the following aspects:
Let  $\epsilon$ be the (squared) kernel bandwidth parameter corresponding to diffusion time,
$N$ the number of samples, 
and $d$ the manifold intrinsic dimensionality,
\vspace{5pt}

$\bullet$ Leveraging the observation in \cite{coifman2006diffusion,singer2006graph} 
that the bias error in the point-wise rate of graph Laplacian can be improved from $O(\sqrt{\epsilon})$ to $O(\epsilon)$ using a $C^2$ kernel function, 
we show that the improved point-wise rate 
{${\rm Err}_{ pt} = O\left( \epsilon, \sqrt{\frac{\log N}{N \epsilon^{d/2+1}}} \right)$}
of Gaussian kernelized graph Laplacian translates into an improved eigen-convergence rate
than using compactly supported kernels.
{Specifically, the eigenvector (2-norm) convergence rate is ${O}( (\log N /N)^{1/(d/2+3)})$, achieved at the optimal choice of $\epsilon \sim (\log N /N)^{1/(d/2+3)}$.}
\vspace{5pt}

$\bullet$ We show that the eigenvalue convergence rate
matches that of the Dirichlet form convergence rate {${\rm Err}_{ form} = O \left( \epsilon, \sqrt{\frac{\log N}{N \epsilon^{d/2}}} \right)$} in \cite{cheng2020convergence}, 
which is better than the point-wise rate ${\rm Err}_{ pt}$.
{This leads to an eigenvalue convergence rate of $ {O}( (\log N /N)^{1/(d/2+2)})$,
achieved at the optimal choice of $\epsilon \sim (\log N /N)^{1/(d/2+2)}$.
The optimal $\epsilon$ for eigenvalue and eigenvector estimation differs in order of $N$.}
\vspace{5pt}

$\bullet$ In obtaining the initial crude eigenvalue lower bound (LB), called Step 1 in below,
we develop a short proof using manifold  heat kernel to define the ``interpolation mapping'',
which constructs from a vector $v$ a smooth function $f$ on $\calM$.
The manifold variational form of $f$, defined via the heat kernel,
naturally relates to the graph Dirichlet form of $v$ when the graph affinity matrix is constructed using a Gaussian kernel. 
The analysis makes use of special properties of manifold heat kernel
and only holds when the graph affinity kernel locally approximates the heat kernel, like the Gaussian. 
This specialty of heat kernel has not been exploited in previous graph Laplacian analysis to obtain eigen-convergence rates.
\vspace{5pt}

Towards the eigen-convergence, our work also recaps and develops several intermediate results under weaker assumptions of the kernel function (i.e., non-Gaussian),
including an improved point-wise convergence rate of density-corrected graph Laplacian.
The density-corrected graph Laplacian, originally proposed in  \cite{coifman2006diffusion},
is an important variant of the kernelized graph Laplacian
where the affinity matrix is $\tilde{W}=D^{-1}WD^{-1}$. 
In applications, the data distribution $p$ is often not uniform on the manifold,
and then the standard graph Laplacian with $W$ recovers the Fokker-Planck operator (weighted Laplacian) 
with measure $p^2$,
which involves a drift term depending on $\nabla_\calM \log p$.
The density-corrected graph Laplacian, in contrast,
 recovers the Laplace-Beltrami operator consistently when $p$ satisfies certain regularity condition,
 and thus is useful in many applications. 
In this work, we first prove the point-wise convergence and Dirichlet form convergence of the density-corrected graph Laplacian with $\tilde{W}$, 
both matching those of the standard graph Laplacian, 
and this can be of independent interest.
Then the eigen-consistency result extends to such graph Laplacians (with Gaussian kernel function),
also achieving  the same rate as the standard graph Laplacian when $p$ is uniform.

In below, we give an overview of the theoretical results starting from assumptions,
and end the introduction section with some further literature review. 
In the rest of the paper, 
Section \ref{sec:prelim} gives preliminaries needed in the analysis.
Sections \ref{sec:step0}-\ref{sec:step23} develop the eigen-convergence of standard graph Laplacians,
both the un-normalized and the normalized (random-walk) ones.
Section \ref{sec:density-corrected} extends to density-corrected graph Laplacian,
and Section \ref{sec:exp} gives numerical results. 
We discuss possible extensions in the last section.

{\bf Notations}.
Default  and asymptotic notations like $O(\cdot)$, $\Omega(\cdot)$, $\Theta(\cdot)$, are listed in Table \ref{tab:notations}. 
In this paper, we treat constants which are determined by $h$, $\calM$, $p$ as absolute ones, including the intrinsic dimension $d$. 
We mainly track the number of samples $N$ and the kernel diffusion time parameter $\epsilon$,
and we may emphasize the constant dependence on $p$ or $\calM$ in certain circumstances, using the subscript notation like $O_{\calM}(\cdot)$.
All constant dependence can be tracked in the proof.

 \begin{table}[t]
\hspace{-20pt}
 \begin{centering}
\scriptsize
   \caption{
\label{tab:theory-summary}
Summary of theoretical results.
} 
\vspace{3pt}
 \begin{tabular}{  m{1.7cm} | m{1.5cm} m{1.5cm}  | c | m{1.4cm} | m{2.6cm} | m{2.7cm} }
 \hline
				&	\multicolumn{2}{c |}{$p$ uniform} 				      & $p$ non-uniform				& 	\multicolumn{2}{c|}{Needed assumptions } & \multirow{2}{*}{Error bound} \\ 			
\cline{5-6}
 				& 	$L_{un}$ with $W$         & 	$L_{rw}$ with $W$	      &	$\tilde{L}_{rw}$ with $\tilde{W}$	  	& on $h$ 	&	on $\epsilon$ ($\epsilon \to 0+$) &	\\
 \hline
\hspace{-5pt}\scriptsize{Eigenvalue UB}  	& 	Prop. \ref{prop:eigvalue-UB} 	& 	Prop. \ref{prop:eigvalue-UB-rw}		& 	Prop. \ref{prop:eigvalue-UB-rw-density-correct}			& Assump. \ref{assump:h-C2-nonnegative}  &   $\epsilon^{d/2} = \Omega( \frac{\log N}{N})$ &  form rate \\
\hspace{-5pt}\scriptsize{Crude eigenvalue LB}  & 	Prop. \ref{prop:eigvalue-LB-crude}& 	Prop. \ref{prop:eigvalue-LB-crude-rw} & 	Prop. \ref{prop:eigvalue-LB-crude-rw-density-correct}&  	Gaussian	&  $\epsilon^{d/2+2} > c_K \frac{\log N}{N}$ 	&	 $O(1)$\\
\hline
\hspace{-5pt}\scriptsize{Eigenvector convergence} 	& Prop. \ref{prop:step2}	& 	\centering{-}  & 	\centering{-} 	&	\multirow{3}{*}{Gaussian}	& \multirow{3}{*}{$\epsilon^{d/2+2} > c_K \frac{\log N}{N}$} & point-wise rate\\
\hspace{-5pt}\scriptsize{Eigenvalue convergence}  	& Prop. \ref{prop:step3}	& 	\centering{-}   & \centering{-} 	&	 					&								 & form rate\\
\hline
\vspace{2pt}
\multirow{4}{1.6cm}{\hspace{-5pt}\scriptsize{{Eigenvalue/vector combined} convergence }} \vspace{2pt}
													& \multirow{4}{*}{ Thm. \ref{thm:refined-rates}} & \multirow{4}{*}{ Thm. \ref{thm:refined-rates-rw}} & \multirow{4}{*}{ Thm. \ref{thm:refined-rates-rw-density-correct} }
													&   \multirow{4}{*}{Gaussian}	&\vspace{2pt} $\epsilon^{d/2+3} \sim \frac{\log N}{N}$ {(optimal order of $\epsilon$ to minimize ${\rm Err}_{pt}$)} & \vspace{2pt} Both $\lambda_k$ and $v_k$: $\tilde{O}(N^{-{1}/(d/2+3)})$	 \vspace{1pt}  \\
\cline{6-7}													
					&						 &							& 				& 	&  \vspace{2pt} $\epsilon^{d/2+2} \sim \frac{\log N}{N}$ 
																							{(optimal order of $\epsilon$ to minimize ${\rm Err}_{form}$)} \vspace{0pt} &  $\lambda_k: \tilde{O}(N^{-{1}/{(d/2+2)}})$, $v_k: \tilde{O}(N^{- {1}/{(d+4)}})$ \\
  \hline
\hspace{-5pt}Point-wise  convergence & 	\multicolumn{2}{c|}{Thm. \ref{thm:pointwise-rate-C2h} \cite{singer2006graph,cheng2020convergence}$^{*}$} &Thm. \ref{thm:pointwise-rate-dencity-correct}& Assump. \ref{assump:h-C2-nonnegative}
			& $\epsilon^{d/2+1} = \Omega( \frac{\log N}{N})$ & point-wise rate\\
\hspace{-5pt}Dirichlet form convergence & 	\multicolumn{2}{c|}{Thm. \ref{thm:form-rate} \cite{cheng2020convergence}$^{*}$} 				& Thm. \ref{thm:form-rate-density-correction}	&  Assump. \ref{assump:h-C2-nonnegative}
& $\epsilon^{d/2} = \Omega( \frac{\log N}{N})$	& form rate\\
\hline
\end{tabular}
\vspace{3pt}
\end{centering}

\begin{center}
\small
``form rate'' is ${\rm Err}_{ form} = O \left( \epsilon, \sqrt{\frac{\log N}{N \epsilon^{d/2}}} \right)$, ~~~
``point-wise rate'' is ${\rm Err}_{ pt} = O\left( \epsilon, \sqrt{\frac{\log N}{N \epsilon^{d/2+1}}} \right)$.
\end{center}

\begin{flushleft}
{\small
In the table, convergence of first $k_{max}$ eigenvalues and eigenvectors are concerned, where $k_{max}$ is fixed.
In the most right column, ``$\lambda_k$'' means the error of eigenvalue convergence, and ``$v_k$'' means the error of eigenvector convergence (in 2-norm). $\tilde{O}( \cdot)$ stands for the possible involvement of a factor of $(\log N)^\alpha$ for some $\alpha >0$. 
In the 2nd (3rd) column, the eigenvector and eigenvalue convergences are proved in Thm. \ref{thm:refined-rates-rw} (Thm. \ref{thm:refined-rates-rw-density-correct})
and are not written as separated propositions. 
$^{*}$The point-wise convergence and Dirichlet form convergence results of graph Laplacian with $W$
hold when $p$ satisfies Assump. \ref{assump:M-p}(A2), i.e., when $p$ is not uniform.
The Dirichlet form convergence with rate may hold when $h$ is not differentiable, e.g., when $h = {\bf 1}_{[0,1)}$, 
cf. Remark \ref{rk:indicator-h-form-rate}.
}\end{flushleft}
\end{table}

\subsection{Overview of main results}\label{subsec:overview}

We first introduce needed assumptions,
and then provide a technical overview of our analysis in Section \ref{subsec:interpolatoin} (Steps 0-1) and Section \ref{subsec:roadmap} (Steps 2-3), summarized as a roadmap at the end of the section.

\subsubsection{{Set-up and assumptions}}

The current paper inherits the probabilistic manifold data setting, namely,
the dataset $\{x_i\}_{i=1}^N$  consists of i.i.d. samples drawn from a distribution on $\calM$  with density $p$ satisfying  the following assumption:
\begin{assumption}[Smooth ${\calM}$ and  $p$]\label{assump:M-p}

(A1) ${\calM}$ is a $d$-dimensional compact connected  
 $C^{\infty}$ manifold (without boundary) 
isometrically embedded in $\mathbb{R}^{D}$.

(A2)
$p\in C^{\infty}(\calM)$ and uniformly bounded both from below and above, that is, $\exists p_{min}, \, p_{max} > 0$ s.t.
\[ 0< p_{min}  \le p(x) \le p_{max} < \infty, \quad\forall x\in{\calM}.
\]
\end{assumption}
\noindent
Suppose $\calM$ is embedded via $\iota$, and  when there is no danger of confusion, we use the same notation $x$ to denote $x\in {\calM}$ and $\iota(x)\in \mathbb{R}^D$.
We have the measure space $(\calM, dV)$:
when $\calM$ is orientable, $dV$ is the Riemann volume form; 
otherwise, $dV$ is the measure associated with the local volume form.
The smoothness of $p$ and $\calM$ fulfills many application scenarios, and possible extensions to less regular $\calM$  or $p$ are postponed.
Our analysis first addresses the basic case where  $p$ is uniform on $\calM$, i.e., $p = \frac{1}{\Vol(\calM)}$ and is a positive constant.
For non-uniform $p$ as in (A2), 
we adopt and analyze the density correction graph Laplacian
in Section \ref{sec:density-corrected}.
In both cases, the graph Laplacian recovers the Laplace-Beltrami operator $\Delta_\calM$. 
In below, we write $\Delta_{\calM}$ as $\Delta$,  $\nabla_{\calM}$ as $\nabla$.

Given $N$ data samples, 
the {\it graph affinity} matrix $W$ and the {\it degree matrix} $D$ are defined as
\[
W_{ij} = K_{\epsilon}(x_i, x_j),  
\quad
D_{ii} = \sum_{j = 1}^N W_{ij}.
\]
$W$ is real symmetric, typically $W_{ij} \ge 0$, and for the kernelized affinity matrix, $W_{ij} = K_\epsilon(x_i,x_j)$ where
\begin{equation}\label{eq:def-K-eps}
K_\epsilon(x,y) : = \epsilon^{-d/2} h \left(\frac{ \| x- y\|^2 }{\epsilon} \right),
\end{equation}
for a function $h: [0,\infty) \to \R$. 
The parameter $\epsilon > 0$ can be viewed as the ``time'' of  the diffusion process.
Some results in literature are written in terms of the parameter $\sqrt{\epsilon} > 0$,
which corresponds to the scale of the local distance  $\| x - y \|$ such that $h(\frac{ \| x- y\|^2 }{\epsilon})$  is of $O(1)$ magnitude.
Our results are written with respect to the time parameter $\epsilon$,
which  corresponds to the {\it squared}  local distance length scale. 

Our main result of graph Laplacian eigen-convergence 
 considers when the kernelized graph affinity is computed with 
\begin{equation}\label{eq:def-h-gaussian}
h( \xi) = \frac{1}{( 4 \pi )^{d/2}} e^{- \xi/ 4}, \quad \xi \in [0, \infty),
\end{equation}
we call such $h$ the {\it Gaussian} kernel function.
(The constant factor $( 4 \pi )^{-d/2}$ is included in the definition of $h$ for theoretical convenience, and may not be needed in algorithm, e.g., in the normalized graph Laplacian the constant factor is cancelled.)

The Gaussian $h$ belongs to a larger family of differentiable functions:

\begin{assumption}[Differentiable $h$]\label{assump:h-C2-nonnegative}
(C1) 
Regularity. $h$ is continuous on $[0,\infty)$,  $C^2$ on $(0, \infty)$. 
\\
(C2) Decay condition.  $\exists a, a_k >0$, s.t., $ |h^{(k)}(\xi )| \leq a_k e^{-a \xi}$ for all $\xi > 0$, 
$k=0, 1,2 $.
\\
(C3) {Non-negativity}. $h \ge 0$ on $[0, \infty)$. To exclude the case that $h \equiv 0$, assume $ \| h \|_\infty > 0 $. 
 \end{assumption}
 
 \noindent
 A summary of results with needed assumptions is provided  in Table \ref{tab:theory-summary},
 from which we can see that several important intermediate results, which can be of independent interest,
only require $h$ to satisfy Assumption \ref{assump:h-C2-nonnegative}  or weaker, including 

\vspace{5pt}
\noindent~
-  Point-wise convergence of graph Laplacians.

\vspace{2pt}
\noindent~
-  Convergence of the graph Dirichlet  form.

\vspace{2pt}
\noindent~
-  The eigenvalue upper bound (UB), which matches to the Dirichlet form convergence rate.
\vspace{5pt}

\noindent
The point-wise convergence and Dirichlet form convergence of standard graph Laplacian
only require a differentiable and decay condition of $h$ as originally taken in \cite{coifman2006diffusion},
and even without Assumption \ref{assump:h-C2-nonnegative}(C3) non-negativity.
Our analysis of density-corrected graph Laplacian assumes $W_{ij} \ge 0$,
and our main result of eigen-convergence needs $h$ to be Gaussian, 
 thus we include (C3) in Assumption \ref{assump:h-C2-nonnegative} to simplify exposition.
{The need of Gaussian $h$ shows up in proving the (initial crude) eigenvalue lower bound (LB), to be explained in below,
and it is  due to the fundamental connection between Gaussian kernel and the manifold heat kernel.}

\subsubsection{{Eigenvalue UB/LB and the interpolation mapping}}\label{subsec:interpolatoin}

{To explain these results and the difference in proving eigenvalue UB and LB, 
we start by introducing the notion of  {\it point-wise rate} and  {\it form rate}.
In the current paper, }

\vspace{2pt}
$\bullet$
Point-wise convergence of graph Laplacians
{is shown to have the rate of $O\left( \epsilon, \sqrt{\frac{\log N}{N \epsilon^{d/2+1}}} \right)$.
We call this rate the ``point-wise rate'', and denote by ${\rm Err}_{ pt}$.}
\vspace{2pt}

$\bullet$
Convergence of the graph Dirichlet  form $\frac{1}{\epsilon N^2}u^T (D-W) u$
applied to smooth manifold functions, i.e., $u = \{ f(x_i) \}_{i=1}^N$ for $f$ smooth on $\calM$,
{is shown to have the rate of $O\left( \epsilon, \sqrt{\frac{\log N}{N \epsilon^{d/2}}} \right)$.
We call this rate the ``form rate'', and denote by ${\rm Err}_{form}$.}
\vspace{5pt}

In literature,
the point-wise convergence of random-walk graph Laplacian $(I-D^{-1}W)$ with differentiable and decay $h$ was firstly shown to have rate
 $O(\epsilon, \sqrt{\frac{\log N}{N \epsilon^{d/2+1}}})$
in \cite{singer2006graph}.
The exposition in \cite{singer2006graph} was for Gaussian $h$ but the analysis therein extends directly to general $h$.
The Dirichlet form convergence with differentiable $h$ was shown to have rate $O(\epsilon, \sqrt{\frac{\log N}{N \epsilon^{d/2}}})$
in \cite{cheng2020convergence} via a V-statistic analysis.
\cite{cheng2020convergence} also derived point-wise rate for both the random-walk and the un-normalized graph Laplacian $(D-W)$.
The analysis in \cite{cheng2020convergence} 
was mainly developed for kernel with adaptive bandwidth, 
and higher order regularity of $h$ ($C^4$ instead of $C^2$) 
was assumed to handle the complication due to variable kernel bandwidth.
For the fixed-bandwidth kernel as in \eqref{eq:def-K-eps},
the analysis in \cite{cheng2020convergence} can be simplified to proceed under less restrictive conditions of $h$. 
We include more details in below when quoting these previous results, {which pave the way towards proving eigen-convergence.}

{Table \ref{tab:theory-summary} illustrates a difference between eigenvalue UB and LB analysis.}
Specifically, the eigenvalue UB holds for general differentiable $h$,
while the initial crude eigenvalue LB,
and consequently the final eigenvalue and eigenvector convergence rate, 
need $h $ to be Gaussian. 
This difference between eigenvalue UB and LB analysis
 is due to the subtlety of the variational principle approach in analyzing empirical eigenvalues.  
To be more specific,  by ``projecting'' the population eigenfunctions  to vectors in $\R^N$ and use as ``candidate'' eigenvectors in the variational form,
the  {Dirichlet form convergence rate}
 directly translates into a rate of eigenvalue UB (for fixed finitely many low-lying eigenvalues). 
{This is why the eigenvalue UB matches the form rate before any LB is derived, and we call this the ``Step 0'' of our analysis.}

The eigenvalue LB, however, is more difficult, as has been pointed out in \cite{burago2014graph}. 
In \cite{burago2014graph} and following works taking the variational principle approach, 
the LB analysis is by ``interpolating'' the empirical eigenvectors to be functions on $\calM$.
Unlike with the population eigenfunctions  which are known to be smooth,
there is less property of the empirical eigenvectors that one can use, 
and any regularity property of these discrete objects is usually non-trivial to obtain \cite{calder2020lipschitz}.
The interpolation mapping in \cite{burago2014graph} first assigns 
a point $x_i$ to a Voronoi cell $V_i$, 
assuming that $\{x_i\}_i$ forms an $\varepsilon$-net of $\calM$ to begin with (a non-probabilistic setting),
and this maps a vector $u$ to a piece-wise constant function $P^* u $ on $\calM$;
next, $P^* u $ is convolved with a kernel function which is compacted supported on a small geodesic ball, 
and this produces ``candidate'' eigenfunctions,
whose manifold differential Dirichlet form is upper bounded by the graph Dirichlet form of $u$, up to an error,
 through differential geometry calculations. 
Under the probabilistic setting of i.i.d. samples,
\cite{trillos2020error} constructed the mapping $P^*$ using a Wasserstein-$\infty$ optimal transport (OT) map, 
where the $\infty$-OT distance between the empirical measure $\frac{1}{N}\sum_i \delta_{x_i}$ and the population measure $p dV$ is bounded 
by constructing a Voronoi tessellation of $\calM$ when $d \ge 2$. 
This led to an overall eigen-convergence rate of $\tilde{O}( N^{-1/2d})$ in \cite{trillos2020error}
when $h$ is compactly supported and satisfies certain regularity conditions and $d \ge 2$,
the $\tilde{O}(\cdot)$ indicating a possible a factor of certain power of $\log N$.
A typical example is when $h$ is an indicator function $h = {\bf1}_{[0,1)}$, 
which is called ``$\varepsilon$-graph''  in computer science literature ($\varepsilon$ corresponds to $\sqrt{\epsilon}$ in our notation).
The approach was extended to $k$NN graphs in \cite{calder2019improved}, 
where the rate of eigenvalue and $2$-norm eigenvector convergence 
was also improved to match the point-wise rate of the epsilon-graph or $k$NN graph Laplacians,
leading to a rate of  $\tilde{O}(N^{- 1/(d + 4)})$ when 
${\epsilon}^{d/2+2} = \Omega  (\frac{\log N }{N})$.
The same rate was shown for $\infty$-norm consistency of eigenvectors in \cite{calder2020lipschitz},
combined with Lipschitz regularity analysis of empirical eigenvectors using advanced PDE tools.
{Eigenvalue consistency with degraded rate was obtained under the regime $\epsilon^{d/2} =\Omega( \frac{\log N}{N})$, which is very sparse graph just beyond graph connectivity threshold \cite{calder2019improved}.
}

In the current work,
we take a different approach for the interpolation mapping in the eigenvalue LB analysis.
Our method is based on manifold heat kernels,
and the analysis makes use of the fact that at short time and on small local neighborhoods, the heat kernel $\calH_t(x,y)$ can be approximated by
\begin{equation}\label{eq:def-Gt}
 G_{t} (x,y) :=
\frac{1}{( 4 \pi  t )^{d/2}} e^{- \frac{ d_\calM( x,y)^2 } { 4 t }},
\end{equation}
and consequently by $K_t(x,y)$ when $h$ is Gaussian as in \eqref{eq:def-h-gaussian}.
The first approximation $H_t \approx G_t$ is by classical results of elliptical operators on Riemannian manifolds, cf. Theorem \ref{thm:Heat-short-time}.
Next, we show that $G_t \approx K_t$ because $K_t$ replaces geodesic distance $d_\calM( x,y)$
with Euclidean distance $\| x-y\|$ in $G_t$,
and the two locally match by $d_\calM( x,y) = \| x- y\| + O(\| x- y\|^3)  $.
{(The constant in the big-O here depends on the second fundamental form, and by compactness of $\calM$ is universal for $x$. Similar universal constant in big-O holds throughout the paper.)}
These estimates allow us to construct interpolated $C^\infty(\calM)$ functions $I_r [v]$ from  discrete vector $v \in \R^N$ by convolving with the heat kernel at time $r = \frac{\epsilon \delta }{2}$, where $ 0 < \delta < 1$ is a fixed constant determined by the first $K=k_{max}+1$ low-lying population eigenvalues $\mu_k$ of $-\Delta$.
Specifically, $\delta$ is inversely proportional to the smallest eigen-gap between $\mu_k$ for $k \le K$
($\mu_k$  assumed to have single multiplicity in the first place, and then the result generalizes to greater than one multiplicity),
which is an $O(1)$ constant determined by $-\Delta$ and $K$.
Applying the variational principle to the operator $I-Q_t$, 
where $Q_t$ is the diffusion semi-group operator and $Q_t$'s spectrum is determined by that of $-\Delta$, 
allows to prove an initial eigenvalue LB smaller than half of the minimum first-$K$ eigen-gap.

{The step to derive $O(1)$ initial crude eigenvalue LB using manifold heat kernel interpolation mapping is called ``Step 1'' in our analysis.
While the interpolation mapping by convolving with a smooth kernel has been used in previous works \cite{burago2014graph,trillos2020error,calder2019improved}, using the manifold heat kernel plays a special role in the eigenvalue LB analysis, and this cannot be equivalently achieved by other choices of kernels (unless the kernel locally approximates the heat kernel, like the Gaussian kernel here). 
Specifically, Lemma \ref{lemma:qs2-UB} is proved using heat kernel properties (without using concentration of i.i.d. data samples),
and the lemma connects the continuous integral form of interpolated candidate eigenfunctions with the graph Dirichlet form.
}

\subsubsection{Road-map of analysis}\label{subsec:roadmap}

The previous subsection has explained Step 0 and 1 of our analysis. Here we summarize the rest of the analysis and provide a road-map.

After an $O(1)$ initial crude eigenvalue LB is obtained in Step 1,
we adopt the  ``bootstrap strategy'' from \cite{calder2019improved}, named as therein,
to obtain a refined (2-norm) eigenvector consistency rate to match to the  graph Laplacian point-wise convergence rate.
We call this ``Step 2''.
Note that the use of smooth kernel (like Gaussian) has an improved bias error in the point-wise rate than compactly supported kernel function,
and then consequently improves the eigen-convergence rate, see more in Remark \ref{rk:eigen-rate-indicator-h}.

Next, leveraging the eigenvector consistency proved in Step 2, we further improve the eigenvalue convergence to match the form rate, {which is better than the point-wise rate. 
We call this ``Step 3''.}
Then the refined eigenvalue LB matches the eigenvalue UB in rate.
In the process, the first $K$ many empirical eigenvalues are upper bounded to be $O(1)$,
which follows by the eigenvalue UB proved in the beginning.

In summary, our eigen-convergence analysis consists of the following four steps,
\begin{itemize}[leftmargin=10pt]
\item[-]
Step 0. 
Eigenvalue UB by the Dirichlet form convergence, matching to the form rate.

\item[-]
Step 1. 
Initial crude eigenvalue LB,
providing eigenvalue error up to the smallest first $K$ eigen-gap.

\item[-]
Step 2. $2$-norm consistency of eigenvectors, up to the point-wise rate.

\item[-]
Step 3. Refined eigenvalue consistency, up to the form rate. 

\end{itemize}

\noindent
Step 1 requires $h $  to be non-negative and currently only covers the Gaussian case. 
This  may be relaxed, 
since the proof only uses the approximation property of $h$, namely that $K_\epsilon \approx \calH_\epsilon$.
In this work, we restrict to the Gaussian case for simplicity and the wide use of Gaussian kernels in applications.

\subsection{More related works}

As we adopt a Dirichlet form-based analysis, 
the eigen-convergence result in the current paper is of the same type as in previous works using variational principle \cite{burago2014graph, trillos2020error, calder2019improved}.
In particular,  the rate concerns the convergence of the first $k_{max}$  many low-lying eigenvalues of the Laplacian,
where $k_{max}$ is a {\it fixed}  finite integer. 
The constants in the big-$O$ notations in the bounds are treated as $O(1)$,
and they depend on $k_{max}$ and these leading eigenvalues and eigenfunctions of the manifold Laplacian.
Such results are useful for applications where leading eigenvectors are the primary focus,
e.g., spectral clustering and dimension-reduced spectral embedding.
An alternative approach is to analyze functional operator consistency
\cite{belkin2007convergence,von2008consistency,singer2016spectral,shi2015convergence},
which may provide different eigen-consistency bounds, e.g.,
$\infty$-norm consistency of eigenvectors using compact embedding of Glivenko-Cantelli function classes \cite{dunson2021spectral}.

The current work considers noise-less data on $\calM$,
while the robustness of graph Laplacian against noise in data is important for applications. 
When manifold data vectors are perturbed by noise in the ambient space, \cite{el2016graph} showed that Gaussian kernel function $h$
has special property to make  kernelized graph Laplacian robust to noise (by a modification of diagonal entries). 
More recently, \cite{landa2021doubly} showed that bi-stochastic normalization can make 
the Gaussian kernelized graph affinity matrix robust to high dimensional heteroskedastic noise in data. 
These results suggest that Gaussian $h$ is a special and useful choice of kernel function for graph Laplacian methods. 

Meanwhile, bi-stochastically normalized graph Laplacian has been studied in \cite{marshall2019manifold},
where the point-wise convergence of the kernel integral operator to the manifold operator was proved.
The spectral convergence of bi-stochastically normalized graph Laplacian for data on hyper-torus was recently proved to be $O( N^{-1/(d/2+4) +o(1)})$ in  \cite{wormell2021spectral}.
 The density-corrected affinity kernel matrix $\tilde{W}= D^{-1}WD^{-1}$, which is analyzed in the current work, provides another normalization of the graph Laplacian which recovers the Laplace-Beltrami operator. 
 It would be interesting to explore the connections to these works and extend our analysis to bi-stochastically normalized graph Laplacians,
 which may have better properties of spectral convergence and noise-robustness.

\section{Preliminaries}\label{sec:prelim}

\subsection{Graph and manifold Laplacians}

We define the following moment constants of function $h$ satisfying Assumption \ref{assump:h-C2-nonnegative},
 \[
m_0 [ h ] := \int_{\R^d} h( \| u \|^2) du,
 \quad
  m_2 [ h ] :=  \frac{1}{d} \int_{\R^d} \| u \|^2 h( \| u \|^2) du,
  \quad
  \tilde{m}[h] :=\frac{m_2 [h]}{ 2 m_0[h]}.
 \]
By (C3), $h \ge 0$ and the case $h \equiv 0$ is excluded, thus $m_0[h], m_2[h] > 0$.
With Gaussian $h$ as in \eqref{eq:def-h-gaussian},
$ m_0 = 1$,
$ m_2 = 2$,
and $\tilde{m} = 1$.
Denote $m_2[h]$ and $m_0[h]$ by $m_2$ and $m_0$ for a shorthand notation, and
\begin{itemize}
\item
 The un-normalized graph Laplacian $L_{un}$ is defined as
\begin{equation}\label{eq:def-L-un}
L_{un} : = \frac{1}{ \frac{m_2}{2} p \epsilon N} (D-W).
\end{equation}
Note that the standard un-normalized graph Laplacian is usually $D-W$, and we divide by the constant $\frac{m_2}{2} p \epsilon N$ 
for the convergence of $L_{un}$ to $-\Delta$.

\item
The random-walk graph Laplacian $L_{rw}$ is defined as 
\begin{equation}\label{eq:def-L-rw}
L_{rw} : = \frac{1}{ \frac{m_2}{ 2 m_0} \epsilon } (I - D^{-1}W),
\end{equation}
with the constant normalization to ensure convergence to $-\Delta$.
\end{itemize}

\noindent
The matrix $L_{un}$ is real-symmetric, 
positive semi-definite (PSD), and the smallest eigenvalue is zero. 
Suppose eigenvalues of $L_{un}$ are $\lambda_k$, $k=1, 2, \cdots$, and sorted in ascending order, that is,
\[
0 = \lambda_1 (L_{un}) \le \lambda_2 (L_{un}) \le \cdots \le \lambda_N (L_{un}).
\] 
The $L_{rw} $ matrix is well-define when $D_i > 0$ for all $i$, 
which holds w.h.p. under the regime that $\epsilon^{d/2} =\Omega( \frac{\log N}{N}) $,
cf. Lemma \ref{lemma:Di-concen}.
We always work under the $\epsilon^{d/2} =\Omega( \frac{\log N}{N}) $ regime,
namely the connectivity regime.
Due to that $D^{-1}W$ is similar to $D^{-1/2}WD^{-1/2}$ which is PSD, 
$L_{rw} $ is also real-diagonalized and has $N$ non-negative real eigenvalues, sorted and denoted as 
$
0 = \lambda_1(L_{rw}) \le \lambda_2 (L_{rw}) \le \cdots \le \lambda_N(L_{rw})$.
We also have that, by the min-max variational formula for real-symmetric matrix, 
\begin{equation*}
\lambda_k(L_{un}) = \min_{ L \subset \R^N, \, dim(L) = k} \sup_{ v \in L, v \neq 0} 
 \frac{ v^T L_{un} v}{v^T v},
 \quad k=1,\cdots, N.
\end{equation*}
We define the {\it graph Dirichlet form} $E_N( u)$ for $u \in \R^N$ as 
\begin{equation}\label{eq:def-ENu}
E_N(u) = 
\frac{1}{\frac{m_2}{2}} 
\frac{1}{\epsilon N^2} u^T( D-W) u 
=
\frac{1}{\frac{m_2}{2}} 
 \frac{1}{2 \epsilon N^2}\sum_{i,j = 1}^N W_{i,j} (u_i - u_j)^2.
\end{equation}
By  \eqref{eq:def-L-un},
$E_N(u) = p 
\frac{1}{N} u^T L_{un} u$,
and thus
\begin{equation}\label{eq:lambdak-min-max}
\lambda_k (L_{un})
= \min_{ L \subset \R^N, \, dim(L) = k} \sup_{ v \in L, v \neq 0 } 
\frac{ 
E_N(v)}{  p  \frac{1}{N}   \| v \|^2 }, 
\quad k =1, \cdots, N.
\end{equation}
Similarly, we have 
\begin{equation}\label{eq:lambdak-rw-min-max}
\lambda_k (L_{rw})
= \min_{ L \subset \R^N, \, dim(L) = k} \sup_{ v \in L,  v \neq 0} 
\frac{ 
E_N(v)}{  \frac{1}{m_0}  \frac{1}{N^2}   v^T D v }, 
\quad k =1, \cdots, N.
\end{equation}

To introduce notations of manifold Laplacian,
we define inner-product in $H: =L^2(\calM ,dV) $ as
$\langle f, g \rangle : = \int_\calM f(x) g(x) dV(x)$, for $f ,g \in L^2(\calM ,dV)$.
We also use $\langle \cdot, \cdot \rangle_q$ to denote inner-product in 
$L^2( \calM, q dV)$, $qdV$ being a general measure on $\calM$ (not necessarily probability measure), 
that is 
$
\langle f, g \rangle_q : = \int_\calM f(x) g(x) q(x) dV(x)$, 
for $f ,g \in L^2(\calM , qdV)$.
For smooth connected compact manifold $\calM$, 
the (minus) manifold Laplacian-Beltrami operator $-\Delta$ has eigen-pairs $\{ \mu_k, \psi_k\}_{k=1}^\infty$, 
\[
0 = \mu_1 < \mu_2 \le  \cdots \le \mu_k \le \cdots,
\]
\[
-\Delta \psi_k = \mu_k \psi_k, 
\quad
\langle \psi_k, \psi_l \rangle = \delta_{k,l}, 
\quad
\psi_k \in C^\infty(\calM),
\quad k, l =1, 2, \cdots.
\]
The second eigenvalue $\mu_2 > 0 $ due to connectivity of $\calM$.
When $\mu_i = \cdots = \mu_{i+l-1} = \mu$ for some eigenvalue $\mu$ of $-\Delta$ having multiplicity $l$, 
the eigenfunctions $\psi_i , \cdots, \psi_{i+l-1}$ can be set to be an orthonormal basis of the $l$-dimensional eigenspace associated with $\mu$. 
Note that $\psi_k \in C^\infty(\calM)$ for generic smooth $\calM$.

 \subsection{Heat kernel on $\calM$}
 
We leverage the special property of Gaussian kernel in the ambient space $\R^D$ that it  locally approximates the manifold heat kernel on ${\calM}$. 
We start from the notations of manifold heat kernel.
Since $\calM$ is smooth compact (no-boundary), 
the Green's function of the heat equation on ${\calM}$ exists, namely the heat kernel $\calH_t(x,y)$ of $\calM$.  
We denote the heat diffusion semi-group operator as $Q_t$ which can be formally written as $Q_t = e^{ t \Delta}$, and
\[
Q_t f(x) = \int_{\calM} \calH_t(x,y) f(y)  dV(y), \quad \forall f \in L^2(\calM, dV).
\]
By that $Q_t$ is semi-group, we have the reproduce property
\[
\int_{\calM} \calH_t(x,y)  \calH_t( y,z) dV(y) = H_{2t} (x,z), \quad \forall x, z \in \calM, \quad \forall t > 0.
\]
Meanwhile, by the probability interpretation, 
\[
\int_{\calM} \calH_t(x,y)  dV(y) = 1, \quad \forall x\in \calM, \quad \forall t > 0.
\]
Using the eigenvalue and eigenfunctions $\{ \mu_k, \psi_k \}_k$ of $-\Delta$,
the heat kernel has the expansion representation 
$\calH_t(x,y) = \sum_{k=1}^\infty e^{-t \mu_k} \psi_k(x) \psi_k(y)$.
We will not use the spectral expansion of $\calH_t$  in our analysis, 
but only that $\psi_k$ are also eigenfunctions of $Q_t$, that is,
\begin{equation}\label{eq:Qt-eigen}
Q_t \psi_k = e^{-t \mu_k } \psi_k, \quad k=1,2,\cdots
\end{equation}

Next, we derive Lemma \ref{lemma:heat},
which characterizes two properties of the heat kernel $\calH_t$ at sufficiently short time:
First, on a local neighborhood 
on $\calM$, $H_t(x,y)$ can be approximated by $K_t(x,y)$ in the leading order,
where $K_t$ is defined as in \eqref{eq:def-K-eps} with Gaussian $h$;
Second, globally on the manifold  the heat kernel $H_t(x,y)$ has a sub-Gaussian decay.
These are based on classical results about heat kernel on Riemannian manifolds \cite{li1986parabolic,grigor1997gaussian,rosenberg1997laplacian,grigor2009heat},
summarized in the following theorem.

\begin{theorem}[Heat kernel parametrix and decay \cite{rosenberg1997laplacian,grigor1997gaussian}]
\label{thm:Heat-short-time}
Suppose $\calM$ is as in Assumption \ref{assump:M-p} (A1),
 and $m > d/2+2 $ is a positive integer.
Then there are  positive constants $t_0 < 1$, $\delta_0 < inj(\calM)$ i.e. the injective radius of $\calM$, 
and both $t_0$ and $\delta_0$ depend on $\calM$, and 

1) Local approximation: 
There are positive constants $C_1$, $C_2$  which depending on $\calM$, 
 and  $u_0, \cdots, u_m$ $\in C^{\infty}(\calM)$, where $u_0$ satisfies that 
 \[
 |u_0(x,y) -1| \le C_1 d_\calM(x,y)^2, \quad \forall y \in \calM, \, d_\calM( y,x) < \delta_0, 
 \]
 and $G_t$ is defined as in \eqref{eq:def-Gt},
 such that,  when $ t < t_0$, for any $x \in \calM$, 
\begin{equation}\label{eq:parametrix-m}
\left| \calH_t( x,y) - G_t(x,y) \left( \sum_{l=0}^m t^l u_l(x,y) \right) \right| \le C_2 t^{m-d/2+1}, 
\quad
\forall y \in \calM, \, d_\calM( y,x) < \delta_0.
\end{equation}

2) Global decay:
There is positive constant $C_3$ depending on $\calM$ such that, when $ t < t_0$, 
\begin{equation}\label{eq:H-decay}
\calH_t( x,y ) \le  C_3 t^{-d/2} e^{- \frac{ d_\calM( x,y)^2}{ 5 t}},
\quad 
\forall x, y \in \calM.
\end{equation}
\end{theorem} 

\noindent
Part 1) is by the classical parametrix construction of heat kernel on $\calM$, see e.g. Chapter 3 of \cite{rosenberg1997laplacian},
and Part 2) follows the classical upper bound of heat kernel by Gaussian estimate dating back to 60s \cite{aronson1967bounds,grigor2009heat}.
We include a proof of the theorem in Appendix \ref{app:proofs-prelim} for completeness. 

The theorem directly gives to the following lemma (proof in Appendix \ref{app:proofs-prelim}),
which is useful for our construction of interpolation mapping using heat kernel. 
We denote by $B_\delta(x)$ the Euclidean ball in $\R^D$ centered at point $x$ of radius $\delta$.

\begin{lemma}\label{lemma:heat}
Suppose $\calM$ is as in Assumption \ref{assump:M-p} (A1),
and $t \to 0+$. Let $\delta_t := \sqrt{ 6(10 + \frac{d}{2}) t \log{\frac{1}{t}}}$,
and  $K_t(x,y)$ be with Gaussian kernel $h$, i.e.,
$K_t(x,y) = (4 \pi t)^{-d/2} e^{ - \|x - y\|^2/4t}$.
Then there is positive constant $\epsilon_0$ depending on $\calM$
such that, when $t < \epsilon_0$, for any $ x \in \calM$,
\begin{eqnarray}
& \calH_t ( x,y)  = K_t( x,y) (1 + {O}( t  (\log  t^{-1})^2) ) + O(t^3), 
\quad
\forall y \in B_{\delta_t}(x) \cap \calM,
\label{eq:H-eps-local}\\
& \calH_t(x,y)  =  O( t^{10}), 
\quad \forall y \notin B_{\delta_t}(x) \cap \calM,
\label{eq:H-eps-truncate} \\
& \calH_t(x,y)  = O( t^{-d/2}), 
\quad \forall  x, y \in \calM. \label{eq:H-global-boundedness}
\end{eqnarray}
The constants in big-$O$ in all the equations
 only depend on $\calM$ and are uniform for all $x$.
\end{lemma}

\section{Eigenvalue upper bound}\label{sec:step0}

In this section, we consider uniform $p$ on $\calM$,
and standard graph Laplacians $L_{un}$ and $L_{rw}$
with the kernelized affinity matrix $W$, $W_{ij} = K_\epsilon(x_i, x_j)$ defined as in \eqref{eq:def-K-eps}.
We show the eigenvalue UB for general differentiable $h$ satisfying Assumption \ref{assump:h-C2-nonnegative}, not necessarily Gaussian.

\subsection{Un-normalized graph Laplacian {eigenvalue UB}}

{We now derive Step 0 for $L_{un}$, the result being summarized in the following proposition.}

\begin{proposition}[Eigenvalue UB of $L_{un}$]
\label{prop:eigvalue-UB}
Under Assumption \ref{assump:M-p}(A1),
$p$ being uniform on $\calM$,
 and Assumption \ref{assump:h-C2-nonnegative}.
For fixed $K \in \mathbb{N}$, 
if as $N \to \infty$, $\epsilon \to 0+ $ and $\epsilon^{d/2} = \Omega(  \frac{\log N}{N} ) $, then for sufficiently large $N$, 
w.p. $> 1-4 K^2 N^{-10}$,
\[
\lambda_k (L_{un})
\le \mu_k + O \left(\epsilon ,  \sqrt{ \frac{\log N}{N  \epsilon^{d/2} } } \right) ,
 \quad k=1,\cdots, K.
\]
\end{proposition}

\noindent
The proposition holds when the population eigenvalues $\mu_k$ have more than 1 multiplicities,
as long as they are sorted in an ascending order. 
The proof is by constructing a $k$-dimensional subspace $L$ in \eqref{eq:lambdak-min-max}
spanned by vectors in $\R^N$ which are produced by evaluating the population eigenfunctions $\psi_k$ at the $N$ data points.
{The proof is given in the end of this subsection after we introduce a few needed middle-step results.}

Given $X = \{x_i\}_{i=1}^N$, 
define the function evaluation operator $\rho_X$ applied to  $f: \calM \to \R$ as
\[
\rho_X: C(\calM) \to \R^N, \quad
\rho_X f = (f(x_1), \cdots, f(x_N)).
\]
We will use $u_k = \frac{1}{\sqrt{p}} \rho_X \psi_k$  as ``candidate'' approximate eigenvectors.  
To analyze $E_N ( \frac{1}{\sqrt{p}} \rho_X \psi_k)$, 
the following result from \cite{cheng2020convergence} shows that it converges to the differential Dirichlet form
\[
p^{-1} \langle \psi_k, (-\Delta) \psi_k \rangle_{p^2} = p \mu_k
\]
with the form rate.
The result is for general smooth $p$ and weighted Laplacian $\Delta_q$,
which is defined  as
$\Delta_q :=  \Delta +   \frac{\nabla q}{q } \cdot \nabla$
for measure $qdV$ on $\calM$.
$\Delta_q$ is reduced to $\Delta$ when $q$ is uniform.

\begin{theorem}[Theorem  {3.4} in \cite{cheng2020convergence}]
\label{thm:form-rate}
Under Assumptions \ref{assump:M-p} and \ref{assump:h-C2-nonnegative}, 
as $N \to \infty$, $ \epsilon \to 0+$, $ \epsilon^{d/2 } = \Omega( \frac{ \log N}{N})$,
then for any $f \in C^{\infty} ({\calM})$, 
 when $N$ is sufficiently large,
w.p. $> 1- 2N^{-10}$,
\[
 E_N( \rho_X f )
= \langle f, -\Delta_{p^{2}} f \rangle_{p^{2}}
+ O_{p,f} \left( \epsilon \right) 
+ O \left(  \sqrt{  \frac{  \log N    }{ N \epsilon^{d/2  }}  \int_{\calM}  |\nabla f |^4 p^{2}  }\right).
\]
The constant in $O_{p,f}(\cdot)$ depends on the $C^4$ norm of $p$ and $f$ on $\calM$,
and that in $O(\cdot)$ is an absolute one. 
\end{theorem}

\begin{proof}[Proof of Theorem \ref{thm:form-rate}]
The proof is by a going through of the proof of Theorem {3.4} of \cite{cheng2020convergence} under the simplified situation when $\beta = 0$ (no normalization of the estimated density is involved). 
Specifically,
the proof uses the concentration of the $V$-statistics $V_{ij}:= \frac{1}{\epsilon} K_\epsilon(x_i, x_j) (f(x_i) -f(x_j))^2$.
The expectation of $\E V_{ij}$, $i \neq j$, equals
$\frac{1}{\epsilon} \int_{\calM} \int_{\calM} K_\epsilon(x,y) (f(x)-f(y))^2 p(x)p(y) dV(x) dV(y)
= m_2[h] \langle f, -\Delta_{p^2} f \rangle_{p^2} + O_{p,f}(\epsilon)$.
Meanwhile, $ | V_{ij}|$ is bounded by $O(\epsilon^{-d/2})$, and the variance of the $V_{ij}$ can also be bounded by $O(\epsilon^{-d/2})$
with the constant as in the theorem,
following the calculation in the proof of Theorem {3.4} in \cite{cheng2020convergence}.
The concentration of $\frac{1}{N(N-1)}\sum_{i,j =1}^N V_{ij}$ at $\E V_{ij}$ then follows by the decoupling of the $V$-statistics,
and it gives the high probability bound in the theorem. 

Note that the results in  \cite{cheng2020convergence} are proved  under the assumption that $h$ to be $C^4$ rather than $C^2$,
that is, requiring Assumption \ref{assump:h-C2-nonnegative}(C1)(C2) to hold for up to 4-th derivative of $h$.
This is because $C^4$ regularity of $h$ is used to  handle complication of  the adaptive bandwidth in the other analysis in  \cite{cheng2020convergence}.
With the fixed bandwidth kernel $K_\epsilon(x,y)$ as defined in \eqref{eq:def-K-eps},
$C^2$ regularity suffices, as originally assumed in \cite{coifman2006diffusion}.
\end{proof}

\begin{remark}[{Relaxation of Assumption \ref{assump:h-C2-nonnegative}}]
\label{rk:non-nagativity-not-needed}
Since the proof only involves the computation of moments of the $V$-statistic, 
it is possible to relax  Assumption \ref{assump:h-C2-nonnegative}(C3) non-negativity of $h$
and replace with certain non-vanishing conditions on $m_0[h]$ and $m_2[h]$,
e.g., as in \cite{coifman2006diffusion} and Assumption {A.3} in \cite{cheng2020convergence}.
Since the non-negativity of $W_{ij}$ is used in other places in the paper,
and our eigenvalue LB needs $h$ to be Gaussian,
we adopt  the non-negativity of $h$ in Assumption \ref{assump:h-C2-nonnegative} for simplicity. 
The $C^4$ regularity of $f$ may also be relaxed, and the constant in $O_{p,f}(\cdot)$ may be improved accordingly.
These extensions are  not further pursued here. 
\end{remark}
\vspace{3pt}

\begin{remark}[{Dirichlet form convergence with compactly supported $h$}]
\label{rk:indicator-h-form-rate}
{The ``epsilon-graph'' corresponds to construct graph affinity using the indicator function kernel  $h= {\bf 1}_{[0,1)}$. Note that the ``epsilon'' stands for the scale of local distance and thus is the $\sqrt{\epsilon}$ here,
because our $\epsilon$ is ``time''.}
When $h = {\bf 1}_{[0,1)}$, 
using the same method as in the proof of Lemma 8 in \cite{coifman2006diffusion},
one can verify that (proof in Appendix \ref{app:proofs-step0}), for $i \neq j$,
\begin{equation}\label{eq:bias-error-remark-indicator-h}
\E V_{ij} = m_2 [h] \langle f, -\Delta_{p^2} f \rangle_{p^2} + O_{p,f}(\epsilon),
\quad  f \in C^{\infty}(\calM).
 \end{equation}
The boundedness and variance of $V_{ij}$ are again  bounded by $O(\epsilon^{-d/2})$,
and thus  the Dirichlet form convergence with 
$h = {\bf 1}_{[0,1)}$ has the same rate $O(\epsilon, \sqrt{\frac{\log N}{N \epsilon^{d/2}}})$ as in Theorem \ref{thm:form-rate}.
This firstly implies that the eigenvalue UB also has the same rate,
following the same proof of Proposition \ref{prop:eigvalue-UB}. 
The final eigen-convergence rate also depends on the point-wise rate of the graph Laplacian,
see more in Remark \ref{rk:eigen-rate-indicator-h}.
\end{remark}
\vspace{10pt}

In Theorem \ref{thm:form-rate} and in below, 
the $\log N$ factor in the variance error bound 
is due to the concentration argument.
Throughout the paper, the classical Bernstein inequality Lemma \ref{lemma:bern} is intensively used.

To proceed, recall the definition of $E_N(u)$ as in \eqref{eq:def-ENu},
we define the bi-linear form for $u,v \in \R^N$ as 
\[
B_N(u,v) := \frac{1}{4}( E_N(u+v) - E_N(u-v) ) = \frac{1}{m_2/2}\frac{1}{\epsilon N^2} u^T( D-W) v, 
\]
which is symmetric, i.e., $B_N(u,v)= B_N(v,u)$,
and $B_N(u,u) = E_N(u)$.
The following lemma characterizes the forms $E_N$ and $B_N$ applied to $\rho_X \psi_k$, proved in Appendix \ref{app:proofs-step0}.

\begin{lemma}\label{lemma:form-rate-psi}
Under Assumption \ref{assump:M-p} (A1),
$p$ being uniform on $\calM$,
and  Assumption \ref{assump:h-C2-nonnegative}.
As $N \to \infty$, $ \epsilon \to 0+$, $ \epsilon^{d/2 } N = \Omega( \log N)$.
For fixed $K$,
when $N$ is sufficiently large, w.p. $> 1 -  2 K^2 N^{-10}$,
\begin{equation}\label{eq:form-rate-psi}
\begin{split}
 E_N( \frac{1}{\sqrt{p}} \rho_X \psi_k) 
 & = p \mu_k + O(\epsilon ) + O \left(  \sqrt{  \frac{  \log N    }{ N \epsilon^{d/2  }}   }\right),
\quad k = 1, \cdots, K, \\
 B_N( \frac{1}{\sqrt{p}} \rho_X \psi_k, \frac{1}{\sqrt{p}} \rho_X \psi_l) 
 & = O(\epsilon ) + O \left(  \sqrt{  \frac{  \log N    }{ N \epsilon^{d/2  }}   }\right),
 \quad k \neq l, \, 1 \le k,l \le K.
\end{split}
\end{equation}
\end{lemma}

We need to show the linear independence of the vectors $\rho_X \psi_1, \cdots, \rho_X \psi_{K}$ such that they span a $K$-dimensional subspace in $\R^N$.
This holds w.h.p. at large $N$, by the following lemma showing the near-isometry of the projection mapping $\rho_X$, proved in Appendix \ref{app:proofs-step0}.

\begin{lemma}\label{lemma:rhoX-isometry-whp}
Under Assumption \ref{assump:M-p} (A1),
$p$ being uniform on $\calM$.
For fixed $K$, 
when $N$ is sufficiently large, w.p. $> 1 -  2 K^2 N^{-10}$,
\begin{equation}\label{eq:uk-near-orthonormal}
\begin{split}
\frac{1}{N } \| \frac{1}{\sqrt{p}} \rho_X \psi_k \|^2 
& = 1 + O( \sqrt{\frac{\log N}{N}}), \, 1 \le k \le K; \\
\frac{1}{N }  ( \frac{1}{\sqrt{p}} \rho_X \psi_k)^T ( \frac{1}{\sqrt{p}} \rho_X \psi_l) 
& = O( \sqrt{\frac{\log N}{N}}), \, k\neq l, \, 1 \le k,l \le K.
\end{split}
\end{equation}
\end{lemma}

{Given these estimates, we are ready to prove Proposition \ref{prop:eigvalue-UB}.}

\begin{proof}[Proof of Proposition \ref{prop:eigvalue-UB}]
For fixed $K$,
consider the intersection of both good events in Lemma \ref{lemma:form-rate-psi}
and \ref{lemma:rhoX-isometry-whp},
which happens w.p. $> 1- 4K^2 N^{-10}$ with large enough $N$.
Let $u_k = \frac{1}{\sqrt{p}} \rho_X \psi_k$, 
by \eqref{eq:uk-near-orthonormal}, 
the set  $\{ u_1, \cdots, u_K\}$ is linearly independent.

For any $1 \le k \le K$,
let $L = \text{Span}\{ u_1, \cdots, u_k\}$, then $dim(L) = k$.
By \eqref{eq:lambdak-min-max}, to show the UB of $\lambda_k$ as in the proposition, it suffices to show that
\[
\sup_{ v \in L, \|v\|^2 = N} \frac{1}{p} E_N(v) 
\le \mu_k + O(\epsilon ) + O \left(  \sqrt{  \frac{  \log N    }{ N \epsilon^{d/2  }}   }\right).
\]
For any $v \in L$,  $\|v\|^2 = N$,
there are $c_j$, $1 \le j \le k $, such that 
$v = \sum_{j=1}^k c_j u_j$.
By \eqref{eq:uk-near-orthonormal},
\[
1 = \frac{1}{N} \|v\|^2 
= \sum_{j=1}^k c_j^2 (1 + O(   \sqrt{\frac{\log N}{N}} ) ) 
+ \sum_{j\neq l, j,l =1}^k |c_j | | c_l |  O(   \sqrt{\frac{\log N}{N}} )
= \| c \|^2 ( 1+ O( K \sqrt{ \frac{\log N}{ N}}) ),
\]
thus $\| c\|^2 = 1 + O( \sqrt{\frac{\log N }{N }})$.
 Meanwhile, 
 $E_N( v) = E_N( \sum_{j=1}^k c_j u_j)
 = \sum_{j,l = 1}^k c_j c_l B_N( u_j, u_l)$,
 and by \eqref{eq:form-rate-psi},
\begin{align}
E_N( v) 
& = \sum_{j=1}^k c_j^2 \left( p \mu_j + O(\epsilon ,   \sqrt{  \frac{  \log N    }{ N \epsilon^{d/2  }}   } ) \right) 
+ \sum_{j\neq l, j,l=1}^k |c_j | | c_l | O(\epsilon ,  \sqrt{  \frac{  \log N    }{ N \epsilon^{d/2  }}   }  )   \nonumber \\
& = p \sum_{j=1}^k \mu_j c_j^2 + K \| c \|^2   O(\epsilon ,  \sqrt{  \frac{  \log N    }{ N \epsilon^{d/2  }}   } )
 \le \| c \|^2  \left\{ p \mu_k +  O(\epsilon ,  \sqrt{  \frac{  \log N    }{ N \epsilon^{d/2  }}   } )  \right\},
\label{eq:UB-ENv}
\end{align}
where since $K$ is fixed integer, we incorporate it into the big-$O$. 
Also, $\mu_k \le \mu_K = O(1)$, and then 
\[
\frac{1}{p}E_N( v) 
\le \left( 1 + O( \sqrt{\frac{\log N }{N }}) \right)
\left\{ \mu_k + O(\epsilon) + O \left(  \sqrt{  \frac{  \log N    }{ N \epsilon^{d/2  }}   }\right)
\right\}
= \mu_k + O(\epsilon) + O \left(  \sqrt{  \frac{  \log N    }{ N \epsilon^{d/2  }}   }\right),
\]
which finishes the proof.
\end{proof}

\subsection{Random-walk graph Laplacian {eigenvalue UB}}

We fist establish a concentration argument of $D_i$ in the following lemma,
which shows that $D_i >0$ w.h.p., by that  $\frac{1}{N}D_i$ concentrates at the value of  $m_0 p > 0$.
Consequently, $\frac{1}{N^2} u^T D u$ also concentrates and the deviation is uniformly bounded for all $u \in \R^N$,
which will be used in analyzing  \eqref{eq:lambdak-rw-min-max}.

\begin{lemma}\label{lemma:Di-concen}
Under Assumption \ref{assump:M-p}(A1), $p$ uniform,  and Assumption \ref{assump:h-C2-nonnegative}.
Suppose as $N \to 0$, $ \epsilon \to 0+$ and $ \epsilon^{d/2} = \Omega( \frac{\log N}{N})$.
Then, when $N$ is large enough, w.p. $> 1- 2N^{-9}$, 

1) The degree $D_i$ concentrates for all $i$, namely, 
\begin{equation}\label{eq:concen-Di}
\frac{1}{N}D_i =  m_0 p + O \left( \epsilon, \sqrt{\frac{\log N}{ N \epsilon^{d/2}}} \right), 
\quad \forall i=1, \cdots, N.
\end{equation}

2) The from $\frac{1}{N^2} u^T D u $ concentrates for all $u$, namely,
\begin{equation}\label{eq:concen-uDu}
\frac{1}{N^2} u^T D u 
=  \frac{1}{N} \|u\|^2 \left( m_0 p + O \left( \epsilon, \sqrt{\frac{\log N}{ N \epsilon^{d/2}}} \right) \right),
 \quad
 \forall u \in \R^N.
\end{equation}
The constants in big-O in \eqref{eq:concen-Di} and \eqref{eq:concen-uDu}
are determined by $(\calM, h)$
and uniform for all $i$ and $u$.
\end{lemma}

\noindent
Part 2) immediately follows from Part 1), the latter being proved by standard concentration argument of independent sum and a union bound for $N$ events. 
With Lemma \ref{lemma:Di-concen}, the proof of the following proposition is similar to that of Proposition \ref{prop:eigvalue-UB},
and the difference lies in handling the denominator of the Rayleigh quotient in  \eqref{eq:lambdak-rw-min-max}.
The proofs of Lemma \ref{lemma:Di-concen} and Proposition \ref{prop:eigvalue-UB-rw} are in  Appendix \ref{app:proofs-step0}.

\begin{proposition}[Eigenvalue UB of $L_{rw}$]
\label{prop:eigvalue-UB-rw}
Suppose $\calM$, $p$ uniform, $h$, $K$, $\mu_k$,  and $\epsilon$
are under the same condition as in Proposition \ref{prop:eigvalue-UB},
then for sufficiently large $N$, 
w.p. $> 1- 2 N^{-9} - 4 K^2 N^{-10}$,  $D_i >0$ for all $i$, and 
\[
\lambda_k  (L_{rw})
\le \mu_k + O \left(\epsilon , \sqrt{ \frac{\log N}{N  \epsilon^{d/2} } } \right),
 \quad k=1,\cdots, K.
\]
\end{proposition}

\section{Eigenvalue crude lower bound in Step 1}\label{sec:step1}

In this section, we  prove $O(1)$ eigenvalue LB  in Step 1,
 first for $L_{un}$,
and then the proof for $L_{rw}$ is similar.

We consider for $t > 0$ the operator $\calL_t $ on $H = L^2( \calM,  dV  )$ defined as
\[
\calL_t := I - Q_t,
\quad
\calL_t f(x) = f(x) - \int_{\calM} \calH_t(x,y) f(y) dV(y), 
\quad f \in H.
\]
The semi-group operator $Q_t$ is  Hilbert-Schmidt, compact, and has eigenvalues and eigenfunctions as in \eqref{eq:Qt-eigen}.
Thus, the operator $\calL_t$ is self-adjoint and PSD, and has
\[
\calL_t \psi_k = (1-e^{-t \mu_k}) \psi_k, \quad k = 1, 2, \cdots
\]
For any $t>0$, the eigenvalues  
$ \{ 1-e^{-t \mu_k} \}_k$
are ascending from 0 and have limit point 1.
We denote $\| f \|^2 = \langle f, f \rangle$ for $f \in H$.
By the variational principle,  we have that when $t>0$, for any $k$,
\begin{equation}\label{eq:eig-Lt-minmax}
1-e^{-t \mu_k}
= \inf_{L \subset H, \, dim(L) = k } \sup_{f \in L, \, \|f\|^2 \neq 0} 
\frac{ \langle f ,  \calL_t f  \rangle}{\langle f, f \rangle }.
\end{equation}
For the first result, we  assume that $\mu_k$ are all of multiplicity 1 for simplicity.
When population eigenvalues have greater than one multiplicity,
the result extends by considering eigenspace rather than eigenvectors in the standard way,
see  Remark \ref{rk:multiplicity}.

\subsection{Un-normalized graph Laplacian {eigenvalue crude LB}}

{We now derive Step 1 for $L_{un}$, the result being summarized in the following proposition.}

\begin{proposition}[Initial crude eigenvalue LB of $L_{un}$]
\label{prop:eigvalue-LB-crude}
Under Assumption \ref{assump:M-p} (A1),
suppose $p$ is uniform on $\calM$,
and  $h$ is Gaussian.
For fixed $k_{max} \in \mathbb{N}$, $K = k_{max}+1$,
suppose $0 = \mu_1  <\cdots < \mu_{K} < \infty$ are all of single multiplicity,
and define 
\begin{equation}\label{eq:def-gamma-K}
\gamma_K : = \frac{1}{2} \min_{1 \le k \le k_{max}} (\mu_{k+1} - \mu_k),
\end{equation}
$\gamma_K > 0$ and is a fixed constant. 
Then there is a absolute constant $c_K$ determined by $\calM$  and $k_{max}$
(specifically, $c_K = c (\frac{\mu_K}{\gamma_K})^{d/2} \gamma_K^{-2}$, where $c$ is a constant depending on $\calM$),
such that,
if as $N \to \infty$, $\epsilon \to 0+ $,
 and  $\epsilon^{d/2+2} >  c_K  \frac{\log N}{N} $, then for sufficiently large $N$, 
w.p. $> 1 - 4 K^2 N^{-10} -4 N^{-9}$,
\[
\lambda_k  (L_{un})
> \mu_k - \gamma_K,
 \quad k=2,\cdots, K.
\]
\end{proposition}

{We prove Proposition \ref{prop:eigvalue-LB-crude} in the end of this subsection after we introduce heat kernel interpolation and establish the needed lemmas.} 

Suppose $\{ \lambda_k, v_k\}_{k=1}^K$ are eigenvalue and eigenvectors of $L_{un}$, to construct a test function $f_k$ on $\calM$ from the vector $v_k$, we define the 
{\it interpolation  mapping} 
(the terminology ``interpolation'' is inherited from \cite{burago2014graph})
by the heat kernel with diffusion time $r$, $0 < r < \epsilon $ to be determined. 
Specifically, define
\[
I_r [ u ](x):= \frac{1}{N} \sum_{j=1}^N  u_j \calH_r( x, x_j), 
\quad I_r: \R^N \to C^\infty(\calM),
\]
and then for any $t > 0$, 
\begin{equation}
{ \langle I_r [u] ,  Q_t I_r [u] \rangle} 
 = \frac{1}{N^2} \sum_{i,j=1}^N u_i u_j \calH_{2r+t}(x_i, x_j), 
\quad
{\langle I_r [u],  I_r [u] \rangle} 
 = \frac{1}{N^2} \sum_{i,j=1}^N u_i u_j \calH_{2r}(x_i, x_j). 
\end{equation}
We define the quadratic form
\[
q_s(u) := \frac{1}{N^2} \sum_{i,j=1}^N u_i u_j \calH_{s}(x_i, x_j),
\quad s > 0, \quad u \in \R^N. 
\]
We also define $q_s^{(0)}$ and  $q_s^{(2)}$ as below, and then for any $u \in \R^N$,
$q_s(u)  = q^{(0)}_{s}(u) - q^{(2)}_{s}(u)$, where
\begin{align}
q^{(0)}_{s}(u) 
 := \frac{1}{N} \sum_{i=1}^N u_i^2 \left(  \frac{1}{N}  \sum_{j=1}^N \calH_{s}(x_i, x_j) \right), 
\quad
q^{(2)}_{s}(u)
 :=  \frac{1}{2} \frac{1}{N^2} \sum_{i,j=1}^N \calH_s(x_i,x_j) (u_i - u_j)^2
\end{align}
We will show that $q^{(0)}_{s}(u) \approx p \frac{1}{N}\|u\|^2$ by concentration of the independent sum $\frac{1}{N}  \sum_{j=1}^N \calH_{s}(x_i, x_j)$;
$ q^{(2)}_{s}(u) \ge 0$ by definition, and will be $  O(s)$ when $u$ is an eigenvector with $\|u \|^2 = N$.

\begin{lemma}\label{lemma:qs0-concen}
Under Assumption \ref{assump:M-p} (A1),
$p$ being uniform on $\calM$.
Suppose as $N \to 0$, $s \to 0+$ and $s^{d/2} = \Omega( \frac{\log N}{N})$.
Then, when $N$ is large enough, w.p. $> 1- 2N^{-9}$, 
\[ 
 q^{(0)}_{s}(u) =  \frac{1}{N} \|u\|^2 \left( p + O_\calM(\sqrt{\frac{\log N}{ N s^{d/2}}}) \right),
 \quad
 \forall u \in \R^N.
\]
The notation $O_\calM(\cdot)$
indicates that the constant depends on $\calM$ and is uniform for all $u$.
\end{lemma}
\begin{proof}[Proof of Lemma \ref{lemma:qs0-concen}]
By definition, 
$ q^{(0)}_{s}(u) = \frac{1}{N} \sum_{i=1}^N u_i^2 (D_s)_i$, 
where $
( D_s)_i :=  \frac{1}{N}  \sum_{j=1}^N \calH_s(x_i, x_j)$,
and $\{ ( D_s)_i \}_{i=1}^N$ are $N$ positive valued random variables. 
It suffices to show that with large enough $N$, w.p. indicated in the lemma,
\begin{equation}\label{eq:concen-Ds}
(D_s)_i =  p + O_\calM(\sqrt{\frac{\log N}{ N s^{d/2}}}), \quad \forall i=1, \cdots, N.
\end{equation}
This can be proved using concentration argument, similar as in the proof of Lemma \ref{lemma:Di-concen} 1),
where we use the boundedness of the heat kernel \eqref{eq:H-global-boundedness} in Lemma \ref{lemma:heat}.
The proof of \eqref{eq:concen-Ds} is given in Appendix \ref{app:proofs-step1}. 
Note that \eqref{eq:concen-Ds} is a property of the r.v. $\calH_s(x_i, x_j)$ only,
which is irrelevant to the vector $u$. Thus the threshold of large $N$ in the lemma
and the constant in big-$O$ depend on $\calM$ and are uniform for all $u$. 
\end{proof}

\begin{lemma}\label{lemma:qs2-UB}
Under Assumption \ref{assump:M-p} ($p$ can be non-uniform),
$h$ being Gaussian, 
 let $0 < \alpha < 1$ be a fixed constant.
Suppose $\epsilon \to 0 +$ as $N \to \infty$, then with sufficiently small $\epsilon$, for any realization of $X$, 
\begin{equation}\label{eq:q2-eps-UB}
0 \le q^{(2)}_{ \epsilon}(u) 
= \left( 1 + {O}( \epsilon  (\log \frac{1}{ \epsilon})^2) \right)  \frac{u^T(D-W) u}{N^2} 
 + \frac{\| u \|^2}{N}  O(\epsilon^{3}),
\quad 
\forall u \in \R^N,
\end{equation}
and
\begin{equation}\label{eq:q2-alphaeps-UB}
0 \le q^{(2)}_{ \alpha \epsilon}(u) 
\le 1.1 \alpha^{-d/2}  \frac{u^T(D-W) u}{N^2} + \frac{\| u \|^2}{N}  O(\epsilon^{3}),
\quad 
\forall u \in \R^N.
\end{equation}
The constants in big-$O$ only depend on $\calM$ and are uniform for all $u$ and $\alpha$.
\end{lemma}

\begin{proof}[Proof of Lemma \ref{lemma:qs2-UB}]
For any $u \in \R^N$,
$q^{(2)}_{ \epsilon}(u) = \frac{1}{2} \frac{1}{N^2} \sum_{i,j=1}^N \calH_\epsilon (x_i,x_j) (u_i - u_j)^2 \ge 0$.
Since $\epsilon = o(1)$, take $t$ in Lemma \ref{lemma:heat} to be $\epsilon$,
when $\epsilon < \epsilon_0$, the three equations hold.
By \eqref{eq:H-eps-truncate},
truncate at an $\delta_\epsilon = \sqrt{ 6(10 + \frac{d}{2}) \epsilon \log{\frac{1}{\epsilon}}}$ Euclidean ball, 

\begin{align*}
q^{(2)}_{ \epsilon}(u) 
= \frac{1}{2} \frac{1}{N^2} \sum_{i,j=1}^N \calH_\epsilon (x_i,x_j) {\bf 1}_{\{ x_j \in B_{\delta_\epsilon}(x_i) \}}
 (u_i - u_j)^2 
 +  O(\epsilon^{10}) \frac{1}{2} \frac{1}{N^2} \sum_{i,j=1}^N  (u_i - u_j)^2.
\end{align*}
By that $\frac{1}{N^2} \sum_{i,j=1}^N (u_i - u_j)^2 \le \frac{2}{N} \| u \|^2$,
and apply \eqref{eq:H-eps-local} with the short hand that $\tilde{O}(\epsilon) $ stands for $ {O}( \epsilon  (\log \frac{1}{ \epsilon})^2) $, 
\begin{align*}
q^{(2)}_{ \epsilon}(u) 
& = \frac{1}{2} \frac{1}{N^2} \sum_{i,j=1}^N 
\left(   K_\epsilon( x_i,x_j) (1 + \tilde{O}( \epsilon)) + O(\epsilon^3) \right) 
{\bf 1}_{\{ x_j \in B_{\delta_\epsilon}(x_i) \}}
 (u_i - u_j)^2 
 +  O(\epsilon^{10}) \frac{\| u \|^2}{N} \\
 & = (1 + \tilde{O}( \epsilon) )
 \frac{1}{2} \frac{1}{N^2} \sum_{i,j=1}^N   K_\epsilon( x_i,x_j) {\bf 1}_{\{ x_j \in B_{\delta_\epsilon}(x_i) \}}   (u_i - u_j)^2 
 +  O(\epsilon^{3}) \frac{\| u \|^2}{N}.
\end{align*}
By the truncation argument for $K_\epsilon (x_i,x_j)$, we have that 
\begin{equation}\label{eq:W-form-truncate}
 \frac{1}{2} \frac{1}{N^2} \sum_{i,j=1}^N K_\epsilon (x_i,x_j) {\bf 1}_{\{  x_j \in B_{\delta_\epsilon}(x_i) \}} (u_i - u_j)^2
 =  \frac{u^T(D-W)u}{N^2}  + \frac{\| u \|^2}{N}  O(\epsilon^{10}).
\end{equation}
Putting together, we have
\[
q^{(2)}_{ \epsilon}(u) 
= (1 + \tilde{O}( \epsilon) )
\left( \frac{u^T(D-W)u}{N^2}  + \frac{\| u \|^2}{N}  O(\epsilon^{10})
\right)
 +  O(\epsilon^{3}) \frac{\| u \|^2}{N},
\]
which proves \eqref{eq:q2-eps-UB}.

To prove \eqref{eq:q2-alphaeps-UB}, since $\alpha < 1$ is a fixed positive constant, $ 0 < \alpha \epsilon < \epsilon < \epsilon_0$,
we then apply Lemma \ref{lemma:heat} with $t$ therein being $\alpha \epsilon$.
With a truncation at $\delta_{\alpha \epsilon}$-Euclidean ball, 
and by \eqref{eq:H-eps-local},
\begin{align*}
q^{(2)}_{\alpha \epsilon}(u) 
& = \frac{1}{2} \frac{1}{N^2} \sum_{i,j=1}^N 
\left(   K_{\alpha \epsilon}( x_i,x_j) (1 + \tilde{O}( \alpha \epsilon)) + O( \alpha^3 \epsilon^3) \right)
{\bf 1}_{\{ x_j \in B_{\delta_{\alpha\epsilon}}(x_i) \}}
 (u_i - u_j)^2 
 +  \frac{\| u \|^2}{N}  O(\epsilon^{10}) \\
 &=
 (1 + \tilde{O}( \epsilon))
   \frac{1}{2} \frac{1}{N^2} \sum_{i,j=1}^N 
 K_{\alpha \epsilon}( x_i,x_j) 
{\bf 1}_{\{ x_j \in B_{\delta_{\alpha\epsilon}}(x_i) \}}
 (u_i - u_j)^2 
 +  \frac{\| u \|^2}{N}  O(\epsilon^{3}).
\end{align*}
Suppose $\epsilon$ is sufficiently small such that $1+\tilde{O}(\epsilon)$ is less than 1.1.
Note that 
\begin{equation}\label{eq:K-alphaeps-K-eps}
K_{\alpha \epsilon}(x,y)  = \frac{1}{(4 \pi \alpha \epsilon)^{d/2}} e^{-\frac{ \| x - y\|^2}{4 \alpha \epsilon}}
\le \frac{1}{\alpha^{d/2}} \frac{1}{(4 \pi  \epsilon)^{d/2}} e^{-\frac{ \| x - y \|^2}{4 \epsilon}}
= \alpha^{-d/2} K_{\epsilon} (x,y),
\quad \forall x,y \in \calM,
\end{equation}
then, by that ${\bf 1}_{\{ x_j \in B_{\delta_{\alpha\epsilon}}(x_i) \}} \le {\bf 1}_{\{ x_j \in B_{\delta_{\epsilon}}(x_i) \}}$,
and again with \eqref{eq:W-form-truncate},
\begin{align*}
q^{(2)}_{\alpha \epsilon}(u) 
& \le
1.1 \frac{1}{2} \frac{1}{N^2} \sum_{i,j=1}^N 
\alpha^{-d/2}  K_{\epsilon} (x_i,x_j)
{\bf 1}_{\{ x_j \in B_{\delta_{\epsilon}}(x_i) \}}
 (u_i - u_j)^2 
 +  \frac{\| u \|^2}{N}  O(\epsilon^{3}) \\
& = 
1.1  \alpha^{-d/2} 
\left( 
\frac{u^T(D-W)u}{N^2}  + \frac{\| u \|^2}{N}  O(\epsilon^{10})
 \right)
 +  \frac{\| u \|^2}{N}  O(\epsilon^{3}),
\end{align*}
and this  proves \eqref{eq:q2-alphaeps-UB}.
\end{proof}

{We are ready to prove Proposition \ref{prop:eigvalue-LB-crude}.}

\begin{proof}[Proof of Proposition \ref{prop:eigvalue-LB-crude}]

For fixed $k_{max}$, since $\gamma_K < \mu_K$, define
\begin{equation}\label{eq:def-detla-const}
\delta := \frac{0.5 \gamma_K}{\mu_K} < 0.5,
\end{equation}
$\delta > 0$ and is a fixed constant determined by $\calM$ and $k_{max}$.
For $\epsilon > 0$, let
\[
r: = \frac{\delta \epsilon}{2}, 
\quad
t = \epsilon - 2r = (1-\delta) \epsilon.
\]
For $L_{un} v_k = \lambda_k v_k$,  where $v_k$ are normalized s.t.
\begin{equation}\label{eq:vk-normalize-un}
\frac{1}{N} v_k^T v_l = \delta_{kl}, \quad 1 \le k,l \le N,
\end{equation}
let $f_k = I_r [ v_k ]$, $k=1, \cdots, K$, then $f_k \in C^\infty(\calM) \subset H$.
Because $\epsilon^{d/2+2} > c_K \frac{\log N}{N}$, and $\epsilon = o(1)$,
 $\epsilon^{d/2} = \Omega( \frac{\log N }{N} )$.
Thus, under the assumption of the current proposition,
the condition needed in Proposition \ref{prop:eigvalue-UB} is satisfied, 
and then when $N$ is sufficiently large,
there is an event $E_{UB}$ which happens w.p. $> 1- 4K^2 N^{-10}$, under which
\begin{equation}\label{eq:lambdak-UB-hold}
\lambda_k \le \mu_k + 0.1 \mu_K \le 1.1 \mu_K, \quad 1 \le k \le K.
\end{equation}
We first show that $\{ f_j\}_{j=1}^K$ are linearly independent by considering $\langle f_k, f_l \rangle$.
By definition, for $1 \le k \le K$,
\[
\langle f_k, f_k \rangle = q_{2r}( v_k) = q^{(0)}_{\delta \epsilon }( v_k) - q^{(2)}_{\delta \epsilon }( v_k),
\]
and for $k \neq l$, $1 \le k , l \le K$,
\[
\langle (f_k \pm f_l), (f_k \pm f_l) \rangle = q_{2r}( v_k \pm v_l) 
= q^{(0)}_{\delta \epsilon }( v_k \pm v_l ) - q^{(2)}_{\delta \epsilon }( v_k \pm v_l).
\]
Because $s = \delta \epsilon$, under the condition of the proposition, $s$ satisfies the condition in   Lemma \ref{lemma:qs0-concen},
and thus, 
with sufficiently large $N$,
there is an event $E^{(0)}$ which happens w.p. $> 1- 2N^{-9}$, under which
\[
q^{(0)}_{\delta \epsilon }( v_k) 
 = p+ O( \sqrt{\frac{\log N}{ N \epsilon^{d/2}}}), \quad 1 \le k \le K;
  \quad
q^{(0)}_{\delta \epsilon }( v_k \pm v_l) 
 = 2p + O( \sqrt{\frac{\log N}{ N \epsilon^{d/2}}}), \quad k \neq l,  1 \le k,l \le K,
\]
where we used that the factor $\delta^{-d/2}$ is a fixed constant.
Meanwhile, applying \eqref{eq:q2-alphaeps-UB} in Lemma \ref{lemma:qs2-UB} where $\alpha = \delta$, 
and note that 
\[
 \frac{v_k^T(D-W) v_k}{N^2} = p \epsilon \lambda_k;
 \quad
 \frac{ (v_k \pm v_l)^T(D-W) (v_k \pm v_l)}{N^2} = p \epsilon (\lambda_k + \lambda_l), 
 \quad k \neq l, 1 \le k, l \le K,
\]
we have that
\[
\begin{split}
q^{(2)}_{\delta \epsilon }( v_k) 
& =
O( \delta^{-d/2} ) p \epsilon   \lambda_k +  O(\epsilon^{3}) , 
 \quad 1 \le k \le K, \\
q^{(2)}_{\delta \epsilon }( v_k \pm v_l) 
&  
= O( \delta^{-d/2})   p \epsilon  ( \lambda_k + \lambda_l)   + 2  O(\epsilon^{3}) ,
 \quad k \neq l,
\end{split}
\]
and by that $\lambda_k, \, \lambda_l \le 1.1 \mu_K $ which is a fixed constant, so is $\delta$, 
 we have that 
\begin{equation}\label{eq:q2-delta-eps-vkvl}
q^{(2)}_{\delta \epsilon }( v_k) = O( \epsilon), \quad 1 \le k \le K;
\quad
q^{(2)}_{\delta \epsilon }( v_k \pm v_l)= O(\epsilon), \quad k \neq l, \, 1\le k, l \le K.
\end{equation}
Putting together, we have that 
\begin{equation}\label{eq:fk-near-iso}
\begin{split}
\langle f_k, f_k \rangle 
&= p + O( \sqrt{\frac{\log N}{ N \epsilon^{d/2}}} ,  \epsilon), \quad 1 \le k \le K, \\
\langle f_k, f_l \rangle 
& = \frac{1}{4}( q_{\delta \epsilon}( v_k + v_l)  - q_{\delta \epsilon}( v_k - v_l) )
= O( \sqrt{\frac{\log N}{ N \epsilon^{d/2}}}  , \epsilon),
\quad k \neq l, \, 1\le k, l \le K.
\end{split}
\end{equation}
This proves linear independence of  $\{ f_j\}_{j=1}^K$ when $N$ is large enough,
since $O( \sqrt{\frac{\log N}{ N \epsilon^{d/2}}} , \epsilon) = o(1)$. 
 
We consider first $K$ eigenvalues of $\calL_t$, $t = (1-\delta) \epsilon$.
For each $2 \le k \le K$, let $L_k = \text{Span}\{ f_1, \cdots, f_k\}$ be a $k$-dimensional subspace in $H$,
then by \eqref{eq:eig-Lt-minmax}, 
\begin{equation}\label{eq:eigen-relation-2}
1 - e^{-(1-\delta) \epsilon \mu_k } 
\le  
 \sup_{ f \in L_k, \, \|f\|^2 \neq 0} 
\frac{ \langle f ,  \calL_t f  \rangle}{\langle f, f \rangle} 
= \frac{ \langle f ,  f  \rangle -   \langle f ,  Q_t f  \rangle }{\langle f, f \rangle}.
\end{equation}
For any $f \in L_k$, $ \|f\|^2 \neq 0$,  there is $c \in \R^k$, $c \neq 0$, such that $ f = \sum_{j=1}^k c_j f_j $.
Thus
\[
f = \sum_{j=1}^k c_j I_r [v_j] 
=  I_r  [  \sum_{j=1}^k c_j v_j] = I_r [v], 
\quad
v : = \sum_{j = 1}^k c_j v_j.
 \]
 Because $v_j$ are orthogonal, $\| v_j \|^2 = N$,  
we have that
\[
\frac{\| v \|^2}{N} = \| c \|^2,
\quad
\frac{ v^T (D-W) v}{ N^2} = \sum_{j=1}^k c_j^2 ( p \epsilon \lambda_j ) \le \lambda_k p \epsilon \| c \|^2.
\]
By definition, 
$\langle f ,  f  \rangle = q_{\delta \epsilon}( v)$, 
and 
$ \langle f ,  Q_t f  \rangle = q_{\epsilon}(v)$.

We first upper bound the numerator of the r.h.s. of \eqref{eq:eigen-relation-2}.
By that $q^{(2)}_{\delta \epsilon}( v) \ge 0$, 
\begin{align}
 \langle f ,  f  \rangle  -  \langle f ,  Q_t f  \rangle
& = q_{\delta \epsilon}( v) - q_{\epsilon}(v) 
 =  q^{(0)}_{\delta \epsilon}( v) - q^{(2)}_{\delta \epsilon}( v) - q^{(0)}_{\epsilon}(v) +q^{(2)}_{\epsilon}(v)  \nonumber \\
& \le (  q^{(0)}_{\delta \epsilon}( v) - q^{(0)}_{\epsilon}(v) ) + q^{(2)}_{\epsilon}(v).
\label{eq:proof-numerator-1}
\end{align}
We have already obtained the good event $E^{(0)}$ when applying Lemma \ref{lemma:qs0-concen} with $s= \delta \epsilon$.
We apply the lemma again to $s = \epsilon$, which gives that with sufficiently large $N$  
there is an event $E^{(1)}$ which happens $w.p. > 1- 2 N^{-9}$, 
and then under $E^{(0)} \cap E^{(1)}$, 
\begin{equation}\label{eq:q0-delta-eps-E0-E1}
q^{(0)}_{ \delta \epsilon } ( v) =  \| c \|^2 ( p + O_\calM( \sqrt{ \delta^{-d/2} \frac{\log N}{N  \epsilon^{d/2}} })), 
\quad
q^{(0)}_{ \epsilon } ( v) =  \| c \|^2 ( p + O_\calM( \sqrt{ \frac{\log N}{N  \epsilon^{d/2}} })).
\end{equation}
We track the constant dependence here: the constant in $O_\calM(\cdot)$ in Lemma \ref{lemma:qs0-concen} is only depending on $\calM$ (and not on $K$),
thus we  use the notation $O_{\calM}(\cdot)$ in \eqref{eq:q0-delta-eps-E0-E1}
and below to emphasize that the constant is $\calM$-dependent only and independent from $K$. 
Then \eqref{eq:q0-delta-eps-E0-E1} gives that 
\[
 q^{(0)}_{\delta \epsilon}( v) - q^{(0)}_{\epsilon}(v) 
 = \| c \|^2  \delta^{-d/4}  O_\calM \left(\sqrt{ \frac{\log N}{N  \epsilon^{d/2}} } \right).
\]
The UB of $q^{(2)}_{\epsilon}(v)$ follows from \eqref{eq:q2-eps-UB} in Lemma \ref{lemma:qs2-UB},
with the shorthand that $\tilde{O}(\epsilon) $ stands for $ {O}( \epsilon  (\log \frac{1}{ \epsilon})^2) $,
\[
q^{(2)}_{\epsilon}(v) = \frac{ v^T (D-W) v}{ N^2}  (1 + \tilde{O}(\epsilon)) + \| c \|^2 O( \epsilon^{3})
\le \epsilon \| c \|^2 (  \lambda_k p  (1 + \tilde{O}(\epsilon)) + O( \epsilon^{2}) ).
\]
Thus, \eqref{eq:proof-numerator-1} continues as
\begin{equation}\label{eq:numerator-UB}
\langle f ,  f  \rangle  -  \langle f ,  Q_t f  \rangle
\le  \epsilon \| c \|^2 \left( 
 \lambda_k p  (1 + \tilde{O}(\epsilon)) + O( \epsilon^{2})  
 +  \delta^{-d/4} O_\calM (  \frac{1}{\epsilon}\sqrt{ \frac{\log N}{N  \epsilon^{d/2}} })
 \right).  
\end{equation}

Next we lower bound the denominator $\langle f, f \rangle$. 
Here we use \eqref{eq:q2-alphaeps-UB} in Lemma \ref{lemma:qs2-UB}, which gives that 
\[
0 \le q^{(2)}_{\delta \epsilon}( v)
  \le \Theta( \delta^{-d/2} )  \frac{ v^T (D-W) v}{ N^2}  + \| c \|^2 O(\epsilon^{3})
\le  \epsilon  \| c \|^2   \left( \lambda_k p \Theta(\delta^{-d/2}) + O(\epsilon^{2}) \right).
\]
Note that we assume under event $E_{UB}$ so that the eigenvalue UB \eqref{eq:lambdak-UB-hold} holds,
thus 
$\lambda_k p \Theta(\delta^{-d/2}) + O(\epsilon^{2}) = O(1)$.
 Together with that $\delta$ is a fixed constant, we have that 
\[
 q^{(2)}_{\delta \epsilon}( v)  =  \| c \|^2  O( \epsilon ).
\]
Then, again under $E^{(1)}$, 
\[
\langle f, f \rangle 
= q^{(0)}_{\delta \epsilon}( v) - q^{(2)}_{\delta \epsilon}( v) 
= \| c \|^2 \left ( p + O( \sqrt{ \delta^{-d/2} \frac{\log N}{N  \epsilon^{d/2}} })  - O( \epsilon ) \right) 
\ge \| c \|^2 \left ( p - O( \epsilon,  \sqrt{ \frac{\log N}{N  \epsilon^{d/2}} })  \right).
\]

Putting together with \eqref{eq:numerator-UB}, and by that $\lambda_k \le 1.1 \mu_K$, we have that
\[
\frac{\langle f ,  f  \rangle  -  \langle f ,  Q_t f  \rangle}{ \langle f, f \rangle}
\le 
\frac{ \epsilon  \left( 
 \lambda_k p   +   \tilde{O}(\epsilon) +  \delta^{-d/4}O_\calM (  \frac{1}{\epsilon}\sqrt{ \frac{\log N}{N  \epsilon^{d/2}} })
 \right)}
 {    p - O( \epsilon,  \sqrt{ \frac{\log N}{N  \epsilon^{d/2}} })   }
 \le 
 \epsilon  \left( 
 \lambda_k    +   \tilde{O}(\epsilon) +   \frac{C}{\epsilon}\sqrt{ \frac{\log N}{N  \epsilon^{d/2}} }
 \right),
\]
where 
$C = c(\calM) \delta^{-d/4}$, and $c(\calM)$ is a constant only depending on $\calM$.
We set
\begin{equation*}
c_K 
: = ( \frac{ C}{0.1 \gamma_K  })^2
= ( \frac{ c(\calM) }{0.1   })^2  \delta^{-d/2} \gamma_K^{-2},
\end{equation*} 
and since we assume $\epsilon^{d/2+2} > c_K \frac{\log N}{N}$ in the current proposition, 
we have that 
$\frac{C}{\epsilon}\sqrt{ \frac{\log N}{N  \epsilon^{d/2}}} <  0.1 \gamma_K  $.
Then, comparing to l.h.s. of \eqref{eq:eigen-relation-2}, we have that
\[
1 - e^{-(1-\delta) \epsilon \mu_k } 
\le
\frac{\langle f ,  f  \rangle  -  \langle f ,  Q_t f  \rangle }{ \langle f, f \rangle}
\le 
\epsilon   \left( 
 \lambda_k    +   \tilde{O}(\epsilon) +   0.1 \gamma_K
 \right).
\]
By the relation that $1- e^{-x} \ge x - x^2$ for any $x \ge 0$,
$1 - e^{-(1-\delta) \epsilon \mu_k }  \ge \epsilon (1-\delta) \left(   \mu_k - (1-\delta) \epsilon \mu_k^2  \right)$,
and when $\epsilon$ is sufficiently small s.t.  $ \epsilon \mu_k^2 \le \epsilon (1.1 \mu_K)^2 < 0.1 \gamma_K$, 
\[
1 - e^{-(1-\delta) \epsilon \mu_k } 
\ge 
\epsilon (1-\delta)  \left(  \mu_k - 0.1 \gamma_K  \right) > 0.
 \]
Noting that for $k \ge 2 $, 
$\mu_k \ge \mu_2 \ge 2 \gamma_K > 0$, because $\mu_1 =0$. 
Thus, 
when 
$\epsilon$ is sufficiently small and the $\tilde{O}(\epsilon) $ term is less than $0.1 \gamma_K$, 
under the good events $E^{(1)} \cap E_{UB}$, which happens w.p. $> 1- 4K^2 N^{-10} - 4 N^{-9}$, we have that
\[
0 < (1-\delta) ( \mu_k - 0.1 \gamma_K  )
\le 
 \lambda_k    +   \tilde{O}(\epsilon) +   0.1 \gamma_K
  <  \lambda_k    + 0.2 \gamma_K.
 \]
 Recall that  by definition \eqref{eq:def-detla-const}, 
 $\delta \mu_K = 0.5 \gamma_K$, 
 then $\delta \mu_k \le \delta \mu_K = 0.5 \gamma_K$,
 also $ 0 < \delta < 0.5$.
Re-arranging the terms gives that
 $ \mu_k < \lambda_k + 0.8 \gamma_K$.
 This can be verified for all $ 2 \le k \le K$,
 and  note that the good event $E^{(1)}$ is w.r.t. $X$, 
and  $E_{UB}$ is constructed for fixed $k_{max}$,
and none is for specific $k \le K$. 
\end{proof}

\subsection{Random-walk graph Laplacian {eigenvalue crude LB}}

The counterpart result of random-walk graph Laplacian is the following proposition.
It replaces Proposition \ref{prop:eigvalue-UB} with Proposition \ref{prop:eigvalue-UB-rw} in obtaining the eigenvalue UB in the analysis,
and consequently the high probability differs slightly.

\begin{proposition}[Initial crude eigenvalue LB of $L_{rw}$]
\label{prop:eigvalue-LB-crude-rw}
Under the same condition and setting 
of $\calM$, $p$ being uniform, $h$ being Gaussian,
and $k_{max}$, $\mu_k$, $\epsilon$ same
as in Proposition \ref{prop:eigvalue-LB-crude}.
Then, for sufficiently large $N$, 
w.p.$> 1- 4K^2 N^{-10} - 6N^{-9}$, 
$
\lambda_k  (L_{rw})
> \mu_k - \gamma_K$,
 for $k=2,\cdots, K$.
\end{proposition}

The proof is similar to that of Proposition \ref{prop:eigvalue-LB-crude} and left to Appendix \ref{app:proofs-step1}.
The difference lies in that the empirical eigenvectors $v_k$ are $D$-orthonormal rather than orthonormal,
and the degree concentration Lemma \ref{lemma:Di-concen} is used to relate $\frac{\|v\|^2}{N}$ with $\frac{1}{N^2} v^T D v$
for arbitrary vector $v$.

\section{Steps 2-3 and eigen-convergence}\label{sec:step23}

\begin{figure}[t]
\captionsetup{width=0.9\linewidth}
\centering
\includegraphics[height=.14\linewidth]{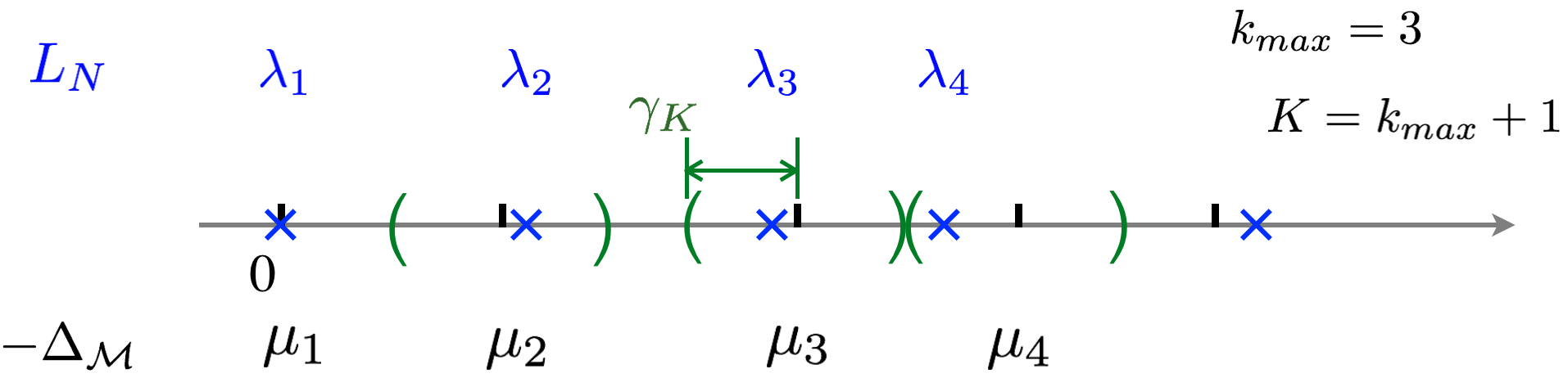} 
\caption{
\scriptsize
Population eigenvalues $\mu_k$ of $-\Delta$, 
and empirical eigenvalues $\lambda_k$ of graph Laplacian matrix $L_N$,
$L_N$ can be $L_{un}$ or $L_{rw}$.
The positive integer $k_{max}$ is fixed,
and the constant $\gamma_K$ is half of the minimum first-$K$ eigen-gaps, defined as in \eqref{eq:def-gamma-K}.
Eigenvalue UB and initial LB are proved for $k \le K$, which guarantees \eqref{eq:eigen-stay-away}.
Extending to greater than one multiplicity by defining $\gamma_K$ as in \eqref{eq:def-gamma-M-multiplicity}.
}
\label{fig:crude}
\end{figure}

In this section, we obtain eigen-convergence rate of $L_{un}$ and $L_{rw}$ from the initial crude eigenvalue bound in Step 1.
We first derive the Steps 2-3 for $L_{un}$,
and the proof for $L_{rw}$ is similar.

\subsection{Step 2 eigenvector consistency}

In Step 1, the crude bound of eigenvalue (the UB already matches the form rate, the LB is crude) gives that for fixed $k_{max}$
and at large $N$, each $\lambda_k$  will fall into the interval $(\mu_k - \gamma_K, \mu_k + \gamma_K)$, 
where $\gamma_K$ is less than half of the smallest eigenvalue gaps $(\mu_2 - \mu_1)$, $\cdots$, $ (\mu_{k_{max}+1} - \mu_{k_{max}})$,
illustrated in Fig. \ref{fig:crude}.
This means that $\lambda_k$ is separated from neighboring $\mu_{k-1}$ and $\mu_{k+1}$ by an $O(1)$ distance away.
This $O(1)$ initial separation is enough for proving eigenvector consistency up to the point-wise rate,
which is a standard argument, see e.g. proof of Theorem 2.6 part 2) in \cite{calder2019improved}. 
In below we provide an informal explanation and then the formal statement in Proposition \ref{prop:step2}, with a proof for completeness. 

We first give an  illustrative informal derivation.
Take $k=2$ for example, let $L_N = L_{un}$, $L_N u_k = \lambda_k u_k$, and we want to show that $ u_2$ and $\rho_X \psi_2$ are aligned.
\[
r_2: = L_N (\rho_X \psi_2) - \rho_X (-\Delta) \psi_2 \in \R^N, 
\quad 
r_2(i) = L_N (\rho_X \psi_2) (x_i) - (-\Delta) \psi_2 (x_i),
\]
the point-wise convergence of graph Laplacian gives $L^\infty$ bound of the residual vector $r_2$, suppose $\|r_2\|_2 \le \varepsilon \| \rho_X \psi_2 \|_2$. Meanwhile, 
for any $l = 1, 3, \cdots, N$,
the crude bound of eigenvalues $\lambda_3$ gives that 
\[ 
\lambda_3 > \mu_2 + \gamma_K,
\]
where $\gamma_K > 0$ is an $O(1)$ constant determined by $k_{max}$ and $\calM$.
Because empirical eigenvalues are sorted, $\lambda_l$ for $l \ge 3$ are also $\gamma_K$ away from $\mu_2$. 
As a result, 
\[
| \lambda_l - \mu_2  | > \gamma_K > 0, 
\quad l \neq 2, \quad 1 \le l \le N.
\]
Then we use the relation that  for each $l \neq 2$, 
$u_l^T  r_2 =  u_l^T(L_N (\rho_X \psi_2) - \mu_2 \rho_X  \psi_2) = 
(\lambda_l  - \mu_2) u_l^T(\rho_X \psi_2)$,
which gives that 
\[
| u_l^T(\rho_X \psi_2 ) | = \frac{ | u_l^T  r_2 |}{|\lambda_l  - \mu_2| } 
\le \frac{ \varepsilon  }{ \gamma_K} \| u_l\|_2 \| \rho_X \psi_2 \|_2.
\]
This shows that $\rho_X \psi_2 $ has $O(\varepsilon)$ alignment with all the other eigenvectors than $u_2$,
and since $\{ u_1, \cdots, u_N \}$ are orthogonal basis in $\R^N$,
this guarantees $1-O(\varepsilon)$ alignment between  $\rho_X \psi_2 $ and $u_2$.

To proceed, we use the point-wise rate of graph Laplacian with $C^2$ kernel $h$ as in the next theorem.
The analysis of point-wise convergence was given in \cite{singer2006graph} and \cite{cheng2020convergence}:
The original theorem in \cite{singer2006graph} considers the normalized graph Laplacian
$(I -D^{-1}W)$. The analysis is similar for $(D-W)$ and leads to the same rate, 
which was derived in \cite{cheng2020convergence} under the setting of variable kernel bandwidth. 
These previous works consider a fixed point $x_0$ on $\calM$,
and since the concentration result has exponentially high probability,
it directly gives  the version of uniform error bound at every data point $x_i$,
which is needed here.

\begin{theorem}[\cite{singer2006graph,cheng2020convergence}]\label{thm:pointwise-rate-C2h}
Under Assumptions \ref{assump:M-p} and \ref{assump:h-C2-nonnegative},
if as $N \to \infty$, $\epsilon \to 0+ $,
$\epsilon^{d/2+1} = \Omega( \frac{\log N}{N} ) $,
then for any $f \in C^4(\calM)$,

1) When $N$ is large enough, w.p. $> 1-4 N^{-9}$,
\[
 \frac{1}{\epsilon \frac{m_2}{2 m_0} } \left( (I -D^{-1}W) (\rho_X f)  \right) _i
= - \Delta_{p^2} f(x_i) + \varepsilon_i, 
\quad 
\sup_{1 \le i \le N} |\varepsilon_i| = O(\epsilon ) + O( \sqrt{\frac{\log N}{N \epsilon^{d/2+1}}} ).
\]

2) When $N$ is large enough, w.p. $> 1-2N^{-9}$,
\[
\frac{1}{\epsilon \frac{m_2}{2} p(x_i) N } \left( (D-W) (\rho_X f)\right)_i 
= - \Delta_{p^2} f(x_i) + \varepsilon_i, 
\quad 
\sup_{1 \le i \le N} |\varepsilon_i| = O(\epsilon ) + O( \sqrt{\frac{\log N}{N \epsilon^{d/2+1}}} ).
\]

The constants in the big-O notations depend on $\calM$, $p$ and the $C^4$ norm of  $f$. 
\end{theorem}

\noindent
Note that Theorem \ref{thm:pointwise-rate-C2h} holds for non-uniform $p$,
while in our eigen-convergence analysis of graph Laplacian with $W$ in below,
we only use the result when $p$ is uniform.
Meanwhile, similar to Theorem \ref{thm:form-rate},
Assumption \ref{assump:h-C2-nonnegative}(C3) may be relaxed
for Theorem \ref{thm:pointwise-rate-C2h} to hold,
cf. Remark \ref{rk:non-nagativity-not-needed}.

\begin{proof}[Proof of Theorem \ref{thm:pointwise-rate-C2h}]
Consider the $N$ events such that $\varepsilon_i $ is less than the error bound.
For each of the $i$-th event, condition on $x_i$, 
Theorem {3.8} in \cite{cheng2020convergence} can be directly used to show that the event holds
w.p. $> 1-4N^{-10}$ for the case 1) random-walk graph Laplacian.
For the case 2) un-normalized graph Laplacian,
adopting the same technique of Theorem {3.6} in  \cite{cheng2020convergence} 
proves the same rate as for the fixed-bandwidth kernel,
and gives that the event holds 
w.p. $> 1-2N^{-10}$.
Specifically, the proof is by showing the concentration of the 
$\frac{1}{\epsilon N} \sum_{j=1}^N K_\epsilon (x_i, x_j) (f(x_j) - f(x_i))$, which is an independent summation condition on $x_i$.
The r.v. $H_j : = \frac{1}{\epsilon}K_\epsilon (x_i, x_j) (f(x_j) - f(x_i))$, $j \neq i$,
has expectation $\E H_j = \frac{m_2}{2} p(x_i) \Delta_{p^2} f(x_i) + O_{f,p}(\epsilon)$,
and $\E H_j^2$ can be shown to be bounded by $\Theta(\epsilon^{-d/2-1})$,
and $|H_j|$ is also bounded by  $\Theta(\epsilon^{-d/2-1})$,
following the same calculation as in the proof of Theorem {3.6} in \cite{cheng2020convergence}.
This shows that the bias error is $O(\epsilon)$, and the variance error is $O( \sqrt{\frac{\log N}{N \epsilon^{d/2+1}}} )$, by classical Bernstein.
Same as in Theorem \ref{thm:form-rate},
$C^2$ regularity and decay up to 2nd derivative of $h$ are enough here.

Strictly speaking, the analysis in \cite{cheng2020convergence}  is for the ``$\frac{1}{N-1}\sum_{j \neq i, j= 1}^N$'' summation
and not the  ``$\frac{1}{N}\sum_{j \neq i, j= 1}^N$'' one here.
However, 
the difference between $\frac{1}{N-1}$ and $\frac{1}{N}$ only introduces an $O(\frac{1}{N})$ relative error and is of higher order,
and the $i =j$ term cancels out in the summation of $(D-W) \rho_X f$.
{In proving this large deviation bound at $x_i$,
the needed threshold for large $N$ is determined by $(\calM, f, p)$ and uniform for $x_i$.}
Then,  
{when $N$ exceeds a threshold uniform for all $x_i$},
by the independence of the $x_i$'s, the $i$-th event holds  
{w.p.$>1-4N^{-10}$ and $>1-2N^{-10}$
for cases 1) and 2) respectively.}
The current theorem, in both 1) and 2), follows by a union bound. 
\end{proof}

We are ready for Step 2 for the unnormalized graph Laplacian $L_{un} = \frac{1}{\epsilon \frac{m_2}{2} p N }(D-W)$.
Here we consider eigenvectors normalized to have 2-norm 1, i.e.,
$L_{un} u_k = \lambda_k u_k$, $u_k^T u_l = \delta_{kl}$,
and we compare $u_k$ to
\begin{equation}\label{eq:def-phik}
\phi_k : = \frac{1}{\sqrt{p N}}  \rho_X \psi_k \in \R^N, 
\end{equation}
where  $\psi_k$ are population eigenfunctions which are orthonormal in $H=L^2(\calM, dV)$, same as above.

\begin{proposition}\label{prop:step2}
Under Assumption \ref{assump:M-p}(A1),
$p$ being uniform on $\calM$, 
and $h$ is Gaussian,
for fixed $k_{max} \in \mathbb{N}$, $K = k_{max}+1$,
 assume that the eigenvalues $\mu_k$ for $k \le K$ are all single multiplicity,
 and $\gamma_K > 0$ as defined in \eqref{eq:def-gamma-K}, the constant $c_K$ as in Proposition \ref{prop:eigvalue-LB-crude}.
If as $N \to \infty$, $\epsilon \to 0+ $,
$\epsilon^{d/2+2} > c_K  \frac{\log N}{N}  $,
then for sufficiently large $N$, 
w.p. $> 1 - 4 K^2 N^{-10} - (2 K+4) N^{-9}$,
there exist scalars $\alpha_k \neq 0$, actually $|\alpha_k| =1 + o(1)$, such that 
\[
 \| u_k - \alpha_k \phi_k \|_2  = O \left(\epsilon , \sqrt{\frac{\log N}{N \epsilon^{d/2+1}}}  \right),
 \quad 1 \le k \le k_{max}.
\]
\end{proposition}

\begin{proof}[Proof of Proposition \ref{prop:step2}]
The proof uses the same approach as that of Theorem 2.6 part 2) in \cite{calder2019improved},
and since our setting is different, we include a proof for completeness.

When $k =1$, we always have
$\lambda_1 = \mu_1 =0$, 
$u_1$ is the constant vector $u_1 = \frac{1}{\sqrt{N}} \mathbf{1}_N $,
and $\psi_1 $ is the constant function, and thus $\phi_1 = u_1$ up to a sign. 
Under the condition of the current proposition, 
the assumptions of Proposition \ref{prop:eigvalue-LB-crude} are satisfied,
and because $\epsilon^{d/2+2} > c_K  \frac{\log N}{N}  $ implies that  $\epsilon^{d/2+1} = \Omega(  \frac{\log N}{N}  )$,
the assumptions of Theorem \ref{thm:pointwise-rate-C2h} 2) are also satisfied. 
We apply  Theorem \ref{thm:pointwise-rate-C2h} 2) to the $K$ functions $\psi_1, \cdots, \psi_K$.
By a union bound,
we have that when $N$ is large enough, w.p. $> 1-2K N^{-9}$,
$\| L_{un} \phi_k - \mu_k \phi_k \|_\infty 
= \frac{1}{\sqrt{pN}} (O(\epsilon) + O( \sqrt{\frac{\log N}{N \epsilon^{d/2+1}}} ))$
for $2 \le k \le K$.
By that $\| v \|_2 \le \sqrt{N} \| v \|_\infty$ for any $ v \in \R^N$, 
this gives that there is $\text{Err}_{pt} > 0$,
\begin{equation}\label{eq:pontwise-rate-2norm-bound}
\| L_{un} \phi_k - \mu_k \phi_k \|_2 \le  \text{Err}_{pt},
\quad 2 \le k \le K,
\quad \text{Err}_{pt} = O(\epsilon) + O( \sqrt{\frac{\log N}{N \epsilon^{d/2+1}}} ).
\end{equation}
The constants in big-O depend on first $K$ eigenfunctions and are absolute ones because $K$ is fixed. 
Applying Proposition \ref{prop:eigvalue-LB-crude}, and consider the intersection with the good event in Proposition \ref{prop:eigvalue-LB-crude},
we have  for each $2 \le k \le K$, $|\mu_k - \lambda_k |< \gamma_K$.
By definition of $\gamma_K$ as in \eqref{eq:def-gamma-K}, 
\begin{equation}\label{eq:eigen-stay-away}
\min_{1 \le j \le N, \, j \neq k} | \mu_k - \lambda_j | > \gamma_K > 0,
\quad 2 \le k \le k_{max}.
\end{equation}
For each $ k \le k_{max}$, let $S_k = \text{Span}\{ u_k \}$ be the 1-dimensional subspace in $\R^N$, and let $S_k^\perp$  be its orthogonal complement. 
We will show that $\|P_{S_k^\perp} \phi_k \|_2$ is  small.
By definition, 
$P_{S_k^\perp} \mu_k  \phi_k = \sum_{j\neq k, j=1}^N  \mu_k (u_j^T \phi_k) u_j$,
and meanwhile,
$P_{S_k^\perp} L_{un} \phi_k = \sum_{j\neq k, j=1}^N  (u_j^T L_{un} \phi_k) u_j = \sum_{j\neq k, j=1}^N  \lambda_j (u_j^T  \phi_k) u_j$.
Subtracting the two gives that 
$P_{S_k^\perp} ( \mu_k  \phi_k - L_{un} \phi_k )
= \sum_{j\neq k, j=1}^N  (\mu_k - \lambda_j) (u_j^T \phi_k) u_j$.
By that $u_j$ are orthonormal vectors, and \eqref{eq:eigen-stay-away},
\[
\| P_{S_k^\perp} ( \mu_k  \phi_k - L_{un} \phi_k ) \|_2^2 
= \sum_{j\neq k, j=1}^N  (\mu_k - \lambda_j)^2 (u_j^T \phi_k)^2
\ge \gamma_K^2 \sum_{j\neq k, j=1}^N   (u_j^T \phi_k)^2
= \gamma_K^2 \|   P_{S_k^\perp} \phi_k \|_2^2.
\]
Then, combined with \eqref{eq:pontwise-rate-2norm-bound}, we have that 
$\gamma_K \|   P_{S_k^\perp} \phi_k \|_2 
\le \| P_{S_k^\perp} ( \mu_k  \phi_k - L_{un} \phi_k ) \|_2
\le \|   \mu_k  \phi_k - L_{un} \phi_k  \|_2 \le \text{Err}_{pt}$, 
namely,
$\| P_{S_k^\perp} \phi_k \|_2  \le \frac{\text{Err}_{pt}}{\gamma_K } $.

By definition,
$P_{S_k^\perp} \phi_k  = \phi_k - (u_k^T \phi_k) u_k$,  where $\| u_k \|_2 = 1$.
 Note that $\phi_k $ are unit vectors up to an $O( \sqrt{ \frac{\log N}{N} })$ error:
Because the good event in Proposition \ref{prop:eigvalue-LB-crude} is under that in the eigenvalue UB Proposition \ref{prop:eigvalue-UB},
and specifically that of Lemma \ref{lemma:rhoX-isometry-whp}. 
Thus  \eqref{eq:uk-near-orthonormal} holds, which means that 
$| \| \phi_k \|^2 - 1 | \le \text{Err}_{norm}$, $1 \le k \le K$, 
where
$\text{Err}_{norm}= O( \sqrt{ \frac{\log N}{N} })$.
Then, one can verify that 
\begin{equation}\label{eq:uk-phik-align}
| u_k^T \phi_k | = 1 + O( \text{Err}_{norm}, \text{Err}_{pt}^2) = 1+o(1),
\end{equation}
and then we set $\alpha_k = \frac{1}{ u_k^T \phi_k}$,
and have that 
\[
\| \alpha_k \phi_k - u_k \|_2 
= \frac{ O( \text{Err}_{pt}  )}{ |u_k^T \phi_k|} 
\le \frac{ O( \text{Err}_{pt}  )}{ 1- O(\text{Err}_{norm}, \text{Err}_{pt}^2)}  
= O( \text{Err}_{pt}  ) (1+ O(\text{Err}_{norm}, \text{Err}_{pt}^2)) = O(\text{Err}_{pt}).
\]
The bound holds for each $k \le k_{max}$.
\end{proof}

\subsection{Step 3: refined eigenvalue LB }

{We now derive Step 3 for $L_{un}$, the result being summarized in the following proposition.}

\begin{proposition}\label{prop:step3}
Under the same condition of Proposition \ref{prop:step2},
$k_{max}$  is fixed.
Then, for sufficiently large $N$, with the same indicated high probability, 
\[
 | \mu_k - \lambda_k | =
  O \left( \epsilon, \, \sqrt{\frac{\log N}{N \epsilon^{d/2}}} \right),
 \quad 1 \le k \le k_{max}.
\]
\end{proposition}

\begin{proof}[Proof of Proposition \ref{prop:step3}]
We inherit the notations in the proof of Proposition \ref{prop:step2}.
Again $\mu_1 = \lambda_1 =0$.
For $2 \le k  \le k_{max}$, note that
\begin{equation}\label{eq:eigen-eqn-step3}
u_k^T( L_{un} \phi_k - \mu_k \phi_k) = (\lambda_k - \mu_k)  u_k^T \phi_k,
\end{equation}
and meanwhile, we have shown that 
$u_k = \alpha_k \phi_k+ \varepsilon_k$, where
$\alpha_k = 1+o(1)$ and $ \| \varepsilon_k \|_2 = O(\text{Err}_{pt})$.
Thus the l.h.s. of \eqref{eq:eigen-eqn-step3} equals
\[
(\alpha_k  \phi_k+ \varepsilon_k)^T( L_{un} \phi_k - \mu_k \phi_k) 
= \alpha_k ( \phi_k^T L_{un} \phi_k   - \mu_k \| \phi_k \|_2^2) + \varepsilon_k^T ( L_{un} \phi_k - \mu_k \phi_k)
=: \textcircled{1} +  \textcircled{2}.
\]
By definition of $\phi_k$,
$\phi_k^T L_{un} \phi_k = \frac{1}{pN} (\rho_X \psi_k)^T L_{un} (\rho_X \psi_k)
= \frac{1}{p^2} E_N( \rho_X \psi_k)$.
The good event in Proposition \ref{prop:step2} is under the good event $E_{UB}$,
under which Lemma \ref{lemma:form-rate-psi} and Lemma \ref{lemma:rhoX-isometry-whp} hold. Then by \eqref{eq:form-rate-psi}, 
$ E_N(  \rho_X \psi_k) 
 = p^2 \mu_k + O(\epsilon ,  \sqrt{  \frac{  \log N    }{ N \epsilon^{d/2  }}   } )$;
By \eqref{eq:uk-near-orthonormal}, 
$ \| \phi_k \|^2  = 1 + O( \sqrt{\frac{\log N}{N}})$.
Putting together, and by that $\alpha_k = 1+o(1) = O(1)$,
\[
 \textcircled{1}  = 
 \alpha_k ( 
 \phi_k^T L_{un} \phi_k  - \mu_k \|\phi_k\|_2^2 )
 = O(1)
\left(  \mu_k + O(\epsilon ,  \sqrt{  \frac{  \log N    }{ N \epsilon^{d/2  }}   } ) - \mu_k(  1 + O( \sqrt{\frac{\log N}{N}})) \right)
 = O(\epsilon ,  \sqrt{  \frac{  \log N    }{ N \epsilon^{d/2  }}   } ).
\]
Meanwhile, by \eqref{eq:pontwise-rate-2norm-bound}, 
$\| L_{un} \phi_k - \mu_k \phi_k \|_2  \le \text{Err}_{pt}$, 
and then
\[
|  \textcircled{2}  | 
\le \| \varepsilon_k \|_2 \| L_{un} \phi_k - \mu_k \phi_k \|_2  =  O(\text{Err}_{pt}^2) .
\]
Because $\epsilon^{d/2+2} > c_K \frac{\log N}{ N} $ for some $c_K > 0$,
$\frac{ \log N}{N \epsilon^{d/2+1}} = \epsilon \frac{ \log N}{N \epsilon^{d/2+2}} <  \frac{\epsilon}{ c_K }$, thus
$\text{Err}_{pt} = O( \epsilon + \sqrt{\frac{ \log N}{N \epsilon^{d/2+1}} }) = O(  \sqrt{\epsilon})$,
 and then 
 $\textcircled{2}   =
 O(\text{Err}_{pt}^2)  = O(\epsilon)$.
 Back to \eqref{eq:eigen-eqn-step3}, we have that 
 \[
 | \lambda_k - \mu_k |  | u_k^T \phi_k | =  |    \textcircled{1}    + \textcircled{2}   |  = O(\epsilon ,  \sqrt{  \frac{  \log N    }{ N \epsilon^{d/2  }}   } ) + O(\epsilon),
 \]
and  by \eqref{eq:uk-phik-align}, $|u_k^T \phi_k| = 1+o(1)$,
thus
$| \lambda_k - \mu_k |  
 =  \frac{ |    \textcircled{1}    + \textcircled{2}   |   }{ 1 + o(1) }  
 = O(  |    \textcircled{1}    + \textcircled{2}   |  )
 = O(\epsilon ,  \sqrt{  \frac{  \log N    }{ N \epsilon^{d/2  }}   } )$.
The above  holds for all  $ k \le k_{max}$.
\end{proof}

\subsection{Eigen-convergence rate}

We are ready to prove the main theorems on eigen-convergence of graph Laplacians,
when $p$ is uniform and the kernel function $h$ is Gaussian. 

\begin{theorem}[eigen-convergence of $L_{un}$]
\label{thm:refined-rates}
Under Assumption \ref{assump:M-p} (A1),
$p$ is uniform on $\calM$,
and $h$ is Gaussian.
For $k_{max} \in \mathbb{N}$ fixed, 
assume that the eigenvalues $\mu_k$ for $k \le K:= k_{max}+1$ are all single multiplicity,
and the constant $c_K$ as in Proposition \ref{prop:eigvalue-LB-crude}.
Consider first $k_{max}$ eigenvalues and eigenvectors of $L_{un}$,
$L_{un} u_k = \lambda_k u_k$, $u_k^T u_l = \delta_{kl}$,
and the vectors $\phi_k$ are defined as in \eqref{eq:def-phik}.
If as $N \to \infty$, $\epsilon \to 0+ $,
$\epsilon^{d/2+2} > c_K \frac{\log N}{N}  $,
then for sufficiently large $N$, 
w.p. $> 1 - 4 K^2 N^{-10} - (2 K+4) N^{-9}$, 
\begin{equation}\label{eq:eigenvalue-form-rate}
 | \mu_k - \lambda_k | =
  O \left( \epsilon, \, \sqrt{\frac{\log N}{N \epsilon^{d/2}}} \right),
  \quad
  1 \le k \le k_{max},
\end{equation}
and  there exist scalars $\alpha_k \neq 0$, actually $|\alpha_k| =1 + o(1)$, such that 
\begin{equation}\label{eq:eigenvector-pt-rate}
 \| u_k - \alpha_k \phi_k \|_2  = O \left( \epsilon , \sqrt{\frac{\log N}{N \epsilon^{d/2+1}}} \right),
\quad   1 \le k \le k_{max}.
\end{equation}
\end{theorem}

\begin{remark}[{Choice of $\epsilon$ and overall rates}]
\label{rk:eigen-rate-gaussian-h}
{The eigen-convergence bounds \eqref{eq:eigenvalue-form-rate} and \eqref{eq:eigenvector-pt-rate} are provided in the combined form of $\epsilon$ and $N$,
as long as the condition $\epsilon =o(1)$ and $\epsilon^{d/2+2} > c_K \log N/N$ holds. 
The bias error in both cases is $O(\epsilon)$, and the variance error has a different inverse power of $\epsilon$ 
($-d/4$ and $-d/4-1/2$ respectively).
The eigenvalue convergence \eqref{eq:eigenvalue-form-rate} achieves the form rate 
${\rm Err}_{form} = O \left( \epsilon, \, \sqrt{\frac{\log N}{N \epsilon^{d/2}}} \right)$,
which is the rate of the Dirichlet form convergence, cf. Theorem \ref{thm:form-rate}.
The (2-norm) eigenvector convergence \eqref{eq:eigenvector-pt-rate} achieves the point-wise rate
${\rm Err}_{pt} =  O \left( \epsilon , \sqrt{\frac{\log N}{N \epsilon^{d/2+1}}} \right)$,
which is the rate of point-wise convergence of graph Laplacian, cf. Theorem \ref{thm:pointwise-rate-C2h}.
}

{The different powers of $\epsilon$ lead to different optimal choice of $\epsilon$, in order of $N$, to achieve the best overall rates for eigenvalue and eigenvector convergence respectively. 
Specifically, 
\begin{itemize}
\item
The optimal choice of $\epsilon$ to minimize ${\rm Err}_{form} $ is 
when $\epsilon = ( c' \frac{\log N}{N} )^{1/(d/2+2)} $ for $c' > c_K$
(which is also the smallest order of $\epsilon$ allowed by the theorem).
This choice leads to
\[
 | \mu_k - \lambda_k | = {O} \left(  ( { \log N }/{N}) ^{{1}/{(d/2+2)}} \right) = \tilde{O}( N^{-1/(d/2+2)} ),
 \quad 1 \le k \le k_{max},
\]
which is the best overall rate of eigenvalue convergence by our theory. 
We use $\tilde{O}(\cdot)$ to denote the involvement of certain factor of $\log N$.
In this case, 
$  \| u_k - \alpha_k \phi_k \|_2  = {O}(  (\frac{ \log N }{N}) ^{{1}/{(d+4)}})$.
\item
The optimal choice of $\epsilon$ to minimize ${\rm Err}_{pt}$ is when $\epsilon \sim (\log N/N)^{{1}/{(d/2+3)}}$,
which leads to
\[
  \| u_k - \alpha_k \phi_k \|_2  
  = {O} \left(  ({ \log N }/{N}) ^{{1}/{(d/2+3)}} \right) = \tilde{O}( N^{- {1}/{(d/2+3)} } ),
  \quad 1 \le k \le k_{max},
\]
which is the best overall rate of eigenvector convergence. 
In this case, $ | \mu_k - \lambda_k | = \tilde{O}( N^{- {1}{(d/2+3)}})$.
\end{itemize}
We can see that the overall rate of eigenvalue convergence achieves the best overall rate of form convergence $\tilde{O}( N^{-1/(d/2+2)} )$,
and that of eigenvector (2-norm) convergence achieves the best overall rate of point-wise convergence $\tilde{O}( N^{-1/(d/2+3)} )$,
at the optimal $\epsilon$  for each convergence respectively.  
}
\end{remark}

\begin{proof}[Proof of Theorem \ref{thm:refined-rates}]
Under the condition of the theorem, 
the eigenvector and eigenvalue error bounds have been proved in Proposition \ref{prop:step2} and Proposition \ref{prop:step3}. 
For the two specific asymptotic scaling of $\epsilon$,
the rate follows from the bounds involving both $\epsilon$ and $N$.
\end{proof}

\begin{remark}[{Comparison to compactly supported $h$}]
\label{rk:eigen-rate-indicator-h}
For $h= {\bf 1}_{[0,1)}$ (see also Remark \ref{rk:indicator-h-form-rate}),
the point-wise convergence of graph Laplacian is known to have the rate as
$\text{Err}_{pt, ind} =   O \left( \sqrt{\epsilon}, \, \sqrt{\frac{\log N}{N \epsilon^{d/2+1}}} \right)$,
see \cite{hein2005graphs,belkin2007convergence,singer2006graph,calder2019improved} among others.
While our way of Step 1 cannot be applied to such $h$, 
 \cite{calder2019improved} covered this case when $d \ge 2$,
and provided the eigenvalue and eigenvector consistency up to $\text{Err}_{pt, ind}$ 
when $\epsilon^{d/2+2} = \Omega( \frac{ \log N}{N})$.
The scaling $\epsilon^{d/2+2} = \tilde{\Theta} ( N^{-1} )$ is the optimal one to balance the bias and variance errors in $\text{Err}_{pt, ind}$,
and then it gives the overall error rate as $\tilde{O}(N^{-1/(d+4)})$,
which agrees with the eigen-convergence rate in \cite{calder2019improved}.
Here $\tilde{O}(\cdot)$ and  $\tilde{\Theta}(\cdot)$ indicate that the constant is possibly multiplied by a factor of certain power of $\log N$.
Meanwhile, we note that, {if following our approach of using  the Dirichlet form convergence rate,
the eigenvalue consistency can be improved to be squared namely $\tilde{O}(N^{-1/(d/2+2)})$ when $\epsilon = \tilde{\Theta}(N^{-1/(d/2+2)})$}.
Specifically, 
by Remark \ref{rk:indicator-h-form-rate}, the Dirichlet form convergence with indicator $h$ is  
$\text{Err}_{form, ind}= O( \epsilon, \sqrt{\frac{\log N}{N \epsilon^{d/2}}})$.
Then, once the initial crude eigenvalue LB is established, 
in Step 2,
the eigenvector 2-norm consistency can be shown to be $\text{Err}_{pt, ind}$.
In Step 3, the eigenvalue consistency for the first $k_{max}$ eigenvalues can be shown to be $O(\text{Err}_{form, ind}, \text{Err}_{pt, ind}^2) =O( \epsilon, \sqrt{\frac{\log N}{N \epsilon^{d/2}}}) $.
This would imply the  eigenvalue convergence rate  of $\tilde{O}(N^{-1/(d/2+2)})$ under the regime where $\epsilon = \tilde{\Theta}(N^{-1/(d/2+2)})$, while the eigenvector consistency remains $\tilde{O}(N^{-1/(d+4)})$.
{Compared to Remark \ref{rk:eigen-rate-gaussian-h}, these rates are the same as Gaussian kernel when setting $\epsilon = \tilde{\Theta}(N^{-1/(d/2+2)} )$
(the optimal order to minimize the eigenvalue rate which is ${\rm Err}_{form}$). However, using Gaussian kernel allows to obtain a better rate for eigenvector convergence, namely $\tilde{O}(N^{-1/(d/2+3)})$, by setting $\epsilon \sim  \tilde{\Theta}(N^{-1/(d/2+3)})$ (the optimal order to minimize the eigenvector convergence rate which is ${\rm Err}_{pt}$). 
This improved eigenvector (2-norm) rate is due to the improved point-wise rate of smooth kernel ${\rm Err}_{pt}$ than that of the indicator kernel  $\text{Err}_{pt,ind}$, and specifically, the bias error is $O(\epsilon)$ 
instead of $O(\sqrt{\epsilon})$.}
\end{remark}
\vspace{5pt}

\begin{remark}[{Extension to larger eigenvalue multiplicity}]
\label{rk:multiplicity}
The result extends  when the population eigenvalues $\mu_k$ have multiplicity greater than one.
Suppose we consider $0 =  \mu^{(1)}  < \mu^{(2)} < \cdots < \mu^{(M)} < \cdots $, 
which are distinct eigenvalues, and $\mu^{(m)}$ has multiplicity $l_m \ge 1$. 
Then let $k_{max} = \sum_{m=1}^M l_m$, $K = \sum_{m=1}^{M+1} l_m$, $\mu_K = \mu^{(M+1)}$,
and $\{ \mu_k, \psi_k \}_{k=1}^K$ are sorted eigenvalues and associated eigenfunctions.  
Step 0. eigenvalue UB holds, since Proposition \ref{prop:eigvalue-UB} does not require single multiplicity. 
In Step 1, 
the only place in Proposition \ref{prop:eigvalue-LB-crude}
 where single multiplicity of $\mu_k$ is used is in the definition of $\gamma_K$.
Then, 
by changing to 
\begin{equation}\label{eq:def-gamma-M-multiplicity}
\gamma^{(M)} = \frac{1}{2} \min_{1 \le m \le M} (\mu^{(m+1)} - \mu^{(m)}) > 0,
\end{equation}
and  defining $\delta = 0.5 \frac{\gamma^{(M)}}{\mu_K}$, $0< \delta < 0.5 $ is a positive constant depending on $\calM$ and $K$,  
Proposition \ref{prop:eigvalue-LB-crude} proves that $| \lambda_k  - \mu^{(m)}| <\gamma^{(M)} $ for all 
$k \le K$, i.e. $m \le M+1$.
This allows to extend Step 2 Proposition \ref{prop:step2} by considering the projection $P_{S^\perp}$
where the subspace in $\R^N$ is spanned by eigenvectors whose eigenvalues $\lambda_k$ approaches $\mu_k = \mu^{(m)}$,
similar as in the original proof of Theorem 2.6 part 2) in \cite{calder2019improved}. 
Specifically,
suppose $\mu_i = \cdots = \mu_{i+l_m-1} = \mu^{(m)}$,  $2 \le m \le M$,
let $S^{(m)} = \text{Span} \{ u_i, \cdots, u_{i+l_m-1}\}$, and the index set $I_m:= \{ i, \cdots, i+l_m-1\}$.
For eigenfunction $\psi_k$, $k \in I_m$,
then  $\mu_k =  \mu^{(m)}$,
similarly as in the proof of Proposition \ref{prop:step2}, 
one can verify that 
\[
\| P_{(S^{(m)})^\perp} ( \mu_k \phi_k - L_{un} \phi_k ) \|_2^2 
= \sum_{j\notin I_m}  (\mu_k - \lambda_j)^2 (u_j^T  \phi_k )^2
\ge (\gamma^{(M)})^2 \sum_{j\notin  I_m}   (u_j^T \phi_k)^2
= (\gamma^{(M)})^2 \|   P_{(S^{(m)})^\perp} \phi_k \|_2^2,
\]
which gives that 
$\|  \phi_k -  P_{S^{(m)}} \phi_k \|_2
= \|   P_{(S^{(m)})^\perp} \phi_k \|_2
\le \frac{1}{\gamma^{(M)}}\text{Err}_{pt}$, for all $k \in I_m$.
By that $\{ \phi_k\}_{k=1}^K$ are near orthonormal with large $N$ (Lemma \ref{lemma:rhoX-isometry-whp}),
this proves that there exists an $l_m$-by-$l_m$ orthogonal transform $Q_m$,
and $|\alpha_k| = 1+o(1)$,
such that
$ \| u_k - \alpha_k \phi_k'  \|_2 
= O(\text{Err}_{pt})  = O(\epsilon , \sqrt{\frac{\log N}{N \epsilon^{d/2+1}}} )$, $k \in I_m$,
where
$[ \phi_k' ]_{k \in I_m}= [ \phi_k ]_{k \in I_m} Q_m  $,
and the notation $[v_j]_{j \in J}$ stands for the $N$-by-$|J|$ matrix formed by concatenating the vectors $v_j$ as columns.
This proves consistency of empirical eigenvectors $u_k$ up to the point-wise rate for $k \le k_{max}$.
Finally, Step 3 Proposition \ref{prop:step3} extends by considering \eqref{eq:eigen-eqn-step3} for 
$u_k$ and $\phi_k'$, 
making use of $ \| u_k - \alpha_k \phi_k'  \|_2 = O(\text{Err}_{pt}) $,
the Dirichlet form convergence of $E_N(\rho_X \psi_k)$ (Lemma \ref{lemma:form-rate-psi}),
and that $\{ \phi_k' \}_{k \in I_m}$ is transformed from $\{ \phi_k \}_{k \in I_m}$  by an orthogonal matrix $Q_m$.
\end{remark}
\vspace{5pt}

To address the eigen-convergence of $L_{rw}$, we define the $D/N$-weighted 2-norm as  
\[
 \| u \|_{\frac{D}{N}}^2 = \frac{1}{N} u^T D u, 
 \]
 and recall that eigenvectors of $L_{rw}$ are $D$-orthogonal. 
The following theorem is the counterpart of Theorem \ref{thm:refined-rates}
for $L_{rw}$, obtaining the same rates.

\begin{theorem}[eigen-convergence of $L_{rw}$]
\label{thm:refined-rates-rw}
Under the same condition and setting of $\calM$, $p$ being uniform, $h$ being Gaussian,
and $k_{max}$, K, $\mu_k$, $\epsilon$ same
as in Theorem \ref{thm:refined-rates}.
Consider first $k_{max}$ eigenvalues and eigenvectors of $L_{rw}$,
$L_{rw} v_k = \lambda_k v_k$, 
$ v_k^T D v_l = \delta_{kl} Np$, 
i.e. $\| v_k \|_{\frac{D}{N}}^2 = p$,
and the vectors $\phi_k$ defined as in \eqref{eq:def-phik}.
Then, for sufficiently large $N$, 
w.p. $> 1 - 4 K^2 N^{-10} - (4 K+ 6) N^{-9}$,
$\| v_k\|_2 = 1+o(1)$,
and the same bound 
of $ | \mu_k - \lambda_k |$  and $ \| v_k - \alpha_k \phi_k \|_2$
as in Theorem \ref{thm:refined-rates} hold for $ 1 \le k \le k_{max}$,
with certain scalars $\alpha_k$  satisfying $|\alpha_k| = 1+o(1)$,
\end{theorem}

The extension to when $\mu_k$ has greater than 1 multiplicity is possible, similarly as in Remark \ref{rk:multiplicity}. 
The proof of $L_{rw}$ uses almost the same method as for $L_{un}$, and the difference is that $v_k$ are no longer orthonormal but $D$-orthogonal. 
This is handled by that $\| u \|_2^2$ and $ \frac{1}{p} \| u\|_{D/N}^2$ agrees in relative error up to the form rate, due to the concentration of $D_i/N$ (Lemma \ref{lemma:Di-concen}). 
The detailed proof is left to Appendix \ref{app:proofs-step23}.

\section{Density-corrected graph Laplacian}\label{sec:density-corrected}

We consider $p$ as in Assumption \ref{assump:M-p}(A2).
The density-corrected graph Laplacian is  defined as \cite{coifman2006diffusion}
\[
\tilde{L}_{rw} = \frac{1}{ \frac{m_2}{ 2 m_0} \epsilon} (I - \tilde{D}^{-1}\tilde{W}), 
\quad 
\tilde{W}_{ij} = \frac{W_{ij}}{D_i D_j}, 
\quad
\tilde{D}_{ii} =\sum_{j=1}^N \tilde{W}_{ij},
\]
where $W_{ij} = K_\epsilon(x_i, x_j)$ as before, and $D$ is the degree matrix of $W$. 
The density-corrected graph Laplacian recovers Laplace-Beltrami operator when $p$ is not uniform.
In this section,
we extend the theory of point-wise convergence,
Dirichlet form convergence,
and eigen-convergence to such graph Laplacian.

\subsection{Point-wise convergence of $\tilde{L}_{rw}$}

This subsection proves Theorem \ref{thm:pointwise-rate-dencity-correct},
which shows that 
the point-wise rate of $\tilde{L}_{rw}$ is same as that of $L_{rw}$ without the density-correction.
The result is for general differentiable $h$ satisfying Assumption \ref{assump:h-C2-nonnegative},
which can be of independent interest.

We first establish the counterpart of Lemma \ref{lemma:Di-concen} 
about the concentration of all $\frac{1}{N}D_i = \frac{1}{N} \sum_{j=1}^N W_{ij}$ when $p$ is not uniform.
The deviation bound is uniform for all $i$ and has an bias error at $O(\epsilon^2)$.

\begin{lemma}\label{lemma:Di-concen-eps2}
Under Assumptions \ref{assump:M-p} and \ref{assump:h-C2-nonnegative}, 
suppose as $N \to \infty$, $\epsilon \to 0+ $, $\epsilon^{d/2} = \Omega( \frac{\log N}{N} ) $.
Then,

1) When $N$ is large enough, w.p. $> 1- 2 N^{-9}$,  $D_i > 0$ for all $i$ s.t. $\tilde{W}$ is well-defined, and 
\begin{equation}\label{eq:degree-D-concen-eps2}
\frac{1}{N} D_i 
= m_0 \tilde{p}_\epsilon(x_i) +  O \left( \epsilon^2, \sqrt{ \frac{\log N}{N \epsilon^{d/2}} }\right),
\quad \tilde{p}_\epsilon := p + \tilde{m} \epsilon ( \omega p + \Delta p),
\quad 
1 \le i \le N.
\end{equation}
where $\omega \in C^{\infty}(\calM)$ is determined by manifold extrinsic coordinates,
and $\tilde{m}[h] = \frac{m_2[h]}{2  m_0[h]}$.

2) When $N$ is large enough, w.p. $> 1- 4 N^{-9}$, $\tilde{D}_i > 0$ for all $i$ s.t. $ \tilde{L}_{rw}$ is well-defined, and 
\begin{equation}\label{eq:denominator}
\sum_{j=1}^N W_{i j} \frac{ 1}{  D_j} 
 =  1+  O \left(\epsilon,  \sqrt{ \frac{\log N}{N \epsilon^{d/2}} } \right) , 
 \quad 
1 \le i \le N.
\end{equation}

The constants in big-$O$ in parts 1) and 2)  depend on ($\calM, p)$, and are uniform for all $ i$.
\end{lemma}

\noindent
The proof is left to Appendix \ref{app:proofs-density-corrected}. The following theorem proves the point-wise rate of $\tilde{L}_{rw}$.

\begin{theorem}\label{thm:pointwise-rate-dencity-correct}
Under Assumptions \ref{assump:M-p} and \ref{assump:h-C2-nonnegative}, 
if as $N \to \infty$, $\epsilon \to 0+ $,
$\epsilon^{d/2+1} = \Omega( \frac{\log N}{N} ) $,
then for any $f \in C^4(\calM)$,  
when $N$ is large enough, w.p. $> 1- 8 N^{-9}$,
\[
\frac{1}{\epsilon \frac{m_2}{2 m_0} } (I - \tilde{D}^{-1} \tilde{W}) (\rho_X f) (x_i )
= - \Delta f(x_i) + \varepsilon_i, 
\quad 
\sup_{1 \le i \le N} |\varepsilon_i| = O(\epsilon ) + O \left( \sqrt{\frac{\log N}{N \epsilon^{d/2+1}}} \right).
\]
The constants in the big-O notation depend on $\calM$, $p$ and the $C^4$ norm of  $f$. 
\end{theorem}

The theorem slightly improves the point-wise convergence rate of $O(\epsilon,  \sqrt{\frac{\log N}{N \epsilon^{d/2+2}}})$ in \cite{singer2016spectral}.
It is proved using the same techniques as the analysis of point-wise convergence of $L_{rw}$ in \cite{singer2006graph,cheng2020convergence},
and we include a proof  for completeness here.

\begin{proof}[Proof of Theorem \ref{thm:pointwise-rate-dencity-correct}]
By definition,
\begin{equation}\label{eq:point-wise-density-correct}
- \frac{1}{\epsilon \frac{m_2}{2 m_0} } (I - \tilde{D}^{-1} \tilde{W}) (\rho_X f) (x_i )
= \frac{1}{\epsilon \frac{m_2}{2 m_0} }
\frac{\sum_{j=1}^N W_{ij} \frac{f(x_j) - f(x_i)}{D_j}}{\sum_{j=1}^N W_{ij} \frac{1}{D_j}}.
\end{equation}
The proof of Lemma \ref{lemma:Di-concen-eps2} has constructed  two good events $E_1$ and $E_2$
($E_1$ is for Part 1) to hold, Part 2) assumes $E_1$ and $E_2$), 
such that with large enough $N$, $E_1 \cap E_2$  happens w.p. $>1-4N^{-9}$, under which 
$D_i$, $\tilde{D}_i > 0$ for all $i$, $\tilde{W}$ and $\tilde{L}_{rw}$ are well-defined, 
and equations 
\eqref{eq:degree-D-concen-eps2}, 
\eqref{eq:degree-D-concen-rel}, and \eqref{eq:denominator} hold.
\eqref{eq:denominator} provides the concentration of the denominator of the r.h.s. of \eqref{eq:point-wise-density-correct}.
We now consider the numerator. 
Note that, with sufficiently small $\epsilon$, $\tilde{p}_\epsilon$  is uniformly bounded from below by $O(1)$ constant $p_{min}'$.
This is because $\omega, p \in C^\infty( \calM)$, $\calM$ is compact, then $(\omega p + \Delta p)$ is uniformly bounded, 
and meanwhile $p$ is uniformly bounded from below.
Thus, under $E_1$, 
\[
\frac{1}{N}\sum_{j=1}^N W_{i j} \frac{ f(x_j )- f(x_{i})}{  \frac{1}{N}D_j} 
= \frac{1}{N}\sum_{j=1}^N  \frac{ W_{i j} ( f(x_j )- f(x_{i}))}{  m_0 \tilde{p}_\epsilon (x_j) ( 1+ \varepsilon_j)},
\quad \max_{1 \le j \le N} |\varepsilon_j| = O(\epsilon^2,  \sqrt{ \frac{\log N}{N \epsilon^{d/2}} }),
\]
and the equation equals
\begin{align*}
\frac{1}{N}\sum_{j=1}^N  \frac{ W_{i j} ( f(x_j )- f(x_{i}))}{  m_0 \tilde{p}_\epsilon (x_j) } ( 1+ \varepsilon_j')
& = \frac{1}{N}\sum_{j=1}^N  \frac{ W_{i j} ( f(x_j )- f(x_{i}))}{  m_0 \tilde{p}_\epsilon (x_j) } 
+ \frac{1}{N}\sum_{j=1}^N  \frac{ W_{i j} ( f(x_j )- f(x_{i}))}{  m_0 \tilde{p}_\epsilon (x_j) } \varepsilon_j' \\
&
=: \textcircled{1} + \textcircled{2},
\quad \max_{1 \le j \le N} |\varepsilon_j'| = O(\epsilon^2,  \sqrt{ \frac{\log N}{N \epsilon^{d/2}} })
\end{align*}
and we analyze the two terms respectively.

To bound $| \textcircled{2}|$,  we use $W_{ij} \ge 0$ and again that $\tilde{p}_\epsilon (x) \ge p_{min}' > 0 $ to have 
\begin{align*}
|  \textcircled{2}|
 \le \frac{1}{N} \sum_{j=1}^N  \frac{ W_{i j} | f(x_j )- f(x_{i})| }{  m_0 \tilde{p}_\epsilon (x_j) } |\varepsilon_j'| 
 \le  
 \frac{ \max_{1 \le j \le N} |\varepsilon_j'| }{m_0 p_{min}'} 
  \cdot  \frac{1}{N} \sum_{j=1}^N  { W_{i j} | f(x_j )- f(x_{i})| }.
\end{align*}
We claim that, {for large enough $N$,}
w.p. $> 1- 2N^{-9}$, and we call this good event $E_3$, under which
\begin{equation}\label{eq:W-abs-diff-f-concen}
\frac{1}{N} \sum_{j=1}^N  { W_{i j} | f(x_j )- f(x_{i})| } = O(\sqrt{\epsilon}),
\quad 1 \le i \le N,
\end{equation}
and the proof is in below. 
With \eqref{eq:W-abs-diff-f-concen}, under $E_3$, $| \textcircled{2}|$ can be bounded by
\begin{equation}\label{eq:circle2-whp}
| \textcircled{2}| = 
(\max_{1 \le j \le N} |\varepsilon_j'| ) O(\sqrt{\epsilon}) 
=O(\epsilon^2,  \sqrt{ \frac{\log N}{N \epsilon^{d/2}} })  O(\sqrt{\epsilon})
=O(\epsilon^{5/2},  \sqrt{ \frac{\log N}{N \epsilon^{d/2-1}} }).
\end{equation}

The analysis of $\textcircled{1}$ uses concentration of independent sum again. Condition on $x_i$ and consider
\[
\textcircled{1}' = 
\frac{1}{  N-1}\sum_{ j \neq i, j=1}^N   K_\epsilon(x_i, x_j) \frac{  f(x_j )- f(x_{i}) }{  \tilde{p}_\epsilon (x_j) }
=: \frac{1}{  N-1}\sum_{ j \neq i, j=1}^N  Y_j,
\]
and we have
$\textcircled{1} = \frac{1}{m_0}(1-\frac{1}{N}) \textcircled{1}'$.
Due to uniform boundedness of $\tilde{p}_\epsilon$ from below by $p_{min}' > 0$,
$|Y_j|$ are bounded by $L_Y = \Theta(\epsilon^{-d/2})$.
We claim that the expectation (proof in below)
\begin{equation}\label{eq:exp-Yj-ptrate-density-correct-term1}
\E Y_j 
= \int_{\calM} K_\epsilon( x_i, y)  \frac{ f(y) p(y) }{ \tilde{p}_\epsilon(y)}  dV(y)
 -  f(x_i) \int_{\calM} K_\epsilon( x_i, y)\frac{ p(y) }{ \tilde{p}_\epsilon(y)}  dV(y)
=  \frac{m_2}{2} \epsilon \Delta f (x_i) 
 + O(\epsilon^2).
\end{equation}
The variance of $Y_j$ is bounded by
\begin{align*}
\E Y_j^2 
& =   \int_{\calM} K_\epsilon(x_i, y)^2 \left( \frac{  f(y)- f(x_{i}) }{  \tilde{p}_\epsilon (y) } \right)^2 p(y) dV(y) \\
& \le \frac{1}{p_{min}'^2}
\int_{\calM} K_\epsilon(x_i, y)^2 \left(  f(y)- f(x_{i})  \right)^2 p(y) dV(y)
\le \nu_Y = \Theta_{f,p}( \epsilon^{-d/2+1}),
\end{align*}
which follows the same derivation as in the proof of the point-wise convergence of $L_{rw}$ without density-correction, 
cf. Theorem \ref{thm:pointwise-rate-C2h} 1),
and can be directly  verified by a similar calculation as in \eqref{eq:calc-inside-ball}.
We attempt at the large deviation bound at $\Theta( \sqrt{\frac{\log N}{N} \nu_Y } ) \sim (\frac{\log N}{N \epsilon^{d/2-1}})^{1/2}$
which is of small order than $\frac{\nu_Y}{L_Y} = \Theta(\epsilon)$
under the theorem condition that $\epsilon^{d/2+1}=\Omega(\frac{\log N}{N})$.
Thus the classical Bernstein gives that for large enough $N$,
{where the threshold is determined by $(\calM, f, p)$ and uniform for $x_i$,}
w.p. $> 1-2N^{-10}$,
\[
\textcircled{1}' =  \E Y_j + O(   \sqrt{\frac{\log N}{N} \nu_Y } ) 
= \frac{m_2}{2} \epsilon \Delta f (x_i) 
 + O(\epsilon^2) + O( \sqrt{ \frac{\log N}{N \epsilon^{d/2-1}}} ),
\]
and as a result,
\begin{equation}\label{eq:circle1-whp}
\textcircled{1} = \tilde{m} \epsilon \Delta f (x_i) 
 + O(\epsilon^2) + O( \sqrt{ \frac{\log N}{N \epsilon^{d/2-1}}} ).
\end{equation}
By a union bound over the events needed at $N$ points, we have that \eqref{eq:circle1-whp}
holds at all $x_i$ under a good event $E_4$ which happens w.p. $>1-2N^{-9}$.

Putting together, under $E_3$ and $E_4$, by \eqref{eq:circle2-whp} and \eqref{eq:circle1-whp},
at all $x_i$,
\begin{align*}
\frac{1}{\epsilon} \sum_{j=1}^N W_{i j} \frac{ f(x_j )- f(x_{i})}{ D_j} 
& = \tilde{m}  \Delta f (x_i) 
 + O(\epsilon) + O( \sqrt{ \frac{\log N}{N \epsilon^{d/2+1}}} )
+O(\epsilon^{3/2},  \sqrt{ \frac{\log N}{N \epsilon^{d/2+1}} }) \\
& =\tilde{m}  \Delta f (x_i) 
 + O(\epsilon, \sqrt{ \frac{\log N}{N \epsilon^{d/2+1}}} ).
\end{align*}
Combined with \eqref{eq:denominator}, under $E_1, E_2, E_3, E_4$,
\begin{align*}
\frac{1}{\epsilon \tilde{m}}  \frac{ \sum_{j=1}^N W_{i j} \frac{ f(x_j )- f(x_{i})}{ D_j} }{ \sum_{j=1}^N W_{i j} \frac{ 1}{  D_j} }
& = 
\frac{  \Delta f (x_i)  + O(\epsilon, \sqrt{ \frac{\log N}{N \epsilon^{d/2+1}}} )}
{1+  O(\epsilon,  \sqrt{ \frac{\log N}{N \epsilon^{d/2}} })} 
 = \Delta f (x_i)  + O(\epsilon, \sqrt{ \frac{\log N}{N \epsilon^{d/2+1}}} ).
\end{align*}
It remains to establish \eqref{eq:W-abs-diff-f-concen} and \eqref{eq:exp-Yj-ptrate-density-correct-term1} 
to finish the proof of the theorem. 
\\

\underline{Proof of  \eqref{eq:W-abs-diff-f-concen}}:
Define r.v. $Y_j = W_{i j} | f(x_j )- f(x_{i})|$ and condition on $x_i$, for $j \neq i$,
$\E Y_j = \int_{\calM} K_\epsilon( x_i, y) | f(y) - f(x_i) | p(y) dV(y)$. 
Let $\delta_\epsilon = \sqrt{ (\frac{d+10}{a}) \epsilon \log {\frac{1}{\epsilon}}} $, 
for any $x\in \calM$,
$K_\epsilon(x,y) = O(\epsilon^{10})$ when $y \notin B_{\delta_\epsilon}(x)$, then
\begin{align*}
& \int_{\calM} K_\epsilon( x, y) | f(y) - f(x) | p(y) dV(y) \\
& = \int_{B_{\delta_{\epsilon}}(x)} K_\epsilon( x, y) | f(y) - f(x) | p(y) dV(y)  
+ O(\epsilon^{10}) \| f\|_\infty \|p\|_\infty \\ 
& \le \int_{B_{\delta_{\epsilon}}(x)} K_\epsilon( x, y)  ( \| \nabla f\|_\infty \| y - x\| ) p(y) dV(y)  
+ O_{f,p}(\epsilon^{10})\\
&= O_{f,p}(\sqrt{\epsilon}) + O_{f,p}(\epsilon^{10}) = O( \sqrt{ \epsilon}).
\end{align*}
The $O_{f,p}(\sqrt{\epsilon})$ is obtained because $\|p\|_\infty$,  $\| \nabla f\|_\infty$  are finite constants, and 
\begin{align}
& \frac{1}{\sqrt{\epsilon}} \int_{B_{\delta_{\epsilon}}(x)} K_\epsilon( x, y)  \| y - x\| dV(y)  
= \int_{B_{\delta_{\epsilon}}(x)} \epsilon^{-d/2} h(\frac{\|x - y\|^2}{\epsilon}  )\frac{\| y - x\|}{\sqrt{\epsilon}} dV(y)   \nonumber \\
&  ~~~
\le \int_{B_{\delta_{\epsilon}}(x)} \epsilon^{-d/2} a_0 e^{-a \frac{\|x-y\|^2}{\epsilon}}\frac{\| y - x\|}{\sqrt{\epsilon}} dV(y) \nonumber  \\
& ~~~
\le \int_{ \| u\| < 1.1 \delta_\epsilon, \, u\in \R^d}  a_0 e^{- \frac{a}{1.1} \| u\|^2}  \frac{ \| u \|}{0.9} (1+ O( \|u \|^2))du  
= O(1),
\label{eq:calc-inside-ball}
\end{align}
where $u \in \R^d$ is the projected coordinates in the tangent plane $T_{x}(\calM)$,
and the comparison of $\|x- y\|_{\R^D}$ to $\| u \|$ (namely $ 0.9 \|x -y\|_{\R^D} < \|u\| < 1.1 \|x -y\|_{\R^D}$)
and the volume comparison (namely $dV(y) = (1+O( \| u\|^2)) du$)
hold when $\delta_\epsilon < \delta_0(\calM)$ which is a constant depending on $\calM$, see e.g. Lemma A.1 in \cite{cheng2020convergence}.

Meanwhile, 
$|Y_j| $ is bounded by $L_Y = \|f\|_\infty \Theta(\epsilon^{-d/2})$,
and the variance of $Y_j$ is bounded by $\E Y_j^2$ and then bounded by $\nu_Y =\Theta(\epsilon^{-d/2 + 1})$,
by a similar calculation as in \eqref{eq:calc-inside-ball}.
We attempt at the large deviation bound at $\Theta( \sqrt{\frac{\log N}{N} \nu_Y } ) \sim (\frac{\log N}{N \epsilon^{d/2-1}})^{1/2}$
which is of small order than $\frac{\nu_Y}{L_Y} = \Theta(\epsilon)$
under the theorem condition that $\epsilon^{d/2+1}=\Omega(\frac{\log N}{N})$.
Thus, {for each $i$, 
when $N$ is enough
where the threshold is determined by $(\calM, f, p)$ and uniform for $x_i$,}
w.p. $> 1- 2N^{-10}$, 
\[
\frac{1}{N-1} \sum_{j\neq i} Y_j = \E Y_j + O( \sqrt{\frac{\log N}{ N \epsilon^{d/2-1}}})
 =  O(\sqrt{\epsilon}) 
  + o(\epsilon)
 = O(\sqrt{\epsilon}).
\]
{The $j=i$ term in \eqref{eq:W-abs-diff-f-concen} equals zero.}
By the same argument of independence of $x_i$ from $\{ x_j \}_{j \neq i}$
and the union bound over $N$ events, we have proved \eqref{eq:W-abs-diff-f-concen}.
\\

\underline{Proof of \eqref{eq:exp-Yj-ptrate-density-correct-term1}}:
Note that
\[
\frac{p}{\tilde{p}_\epsilon} 
= \frac{1}{1+ \epsilon \tilde{m} (\omega + \frac{\Delta p }{p})} 
= 1 -  \epsilon \tilde{m} (\omega + \frac{\Delta p }{p}) + \epsilon^2 r_\epsilon
=1 -  \epsilon r_1 + \epsilon^2 r_\epsilon, 
\]
where $r_1 :=  \tilde{m} (\omega + \frac{\Delta p }{p})$ is a deterministic function, $r_1 \in C^{\infty}(\calM)$;
$r_\epsilon \in C^{\infty}(\calM)$, and $\| r_\epsilon\|_\infty = O(1)$ when $\epsilon$ is less than some $O(1)$ threshold
due to that $\| \omega + \frac{\Delta p }{p} \|_\infty = O(1) $.
Then,
\begin{align*}
& \int_{\calM} K_\epsilon( x_i, y)  \frac{ f p }{ \tilde{p}_\epsilon}(y)  dV(y) 
= \int_{\calM} K_\epsilon( x_i, y) f(y)(1 - \epsilon r_1 +  \epsilon^2 r_\epsilon)(y) dV(y) \\
& ~~~
= \int_{\calM} K_\epsilon( x_i, y) f(y)  dV(y) 
- \epsilon   \int_{\calM} K_\epsilon( x_i, y) (f r_1) (y) dV(y) 
+  \epsilon^2 \int_{\calM} K_\epsilon( x_i, y) (f r_\epsilon)(y)dV(y)  \\
& ~~~
= \left( m_0 f(x_i) +  \frac{m_2}{2} \epsilon (\omega f +\Delta f) (x_i) + O(\epsilon^2) \right) 
-  \epsilon \left( m_0 fr_1( x_i) +  O(\epsilon) \right) + O(\epsilon^2)  \\
& ~~~
=  m_0 f(x_i) +  \frac{m_2}{2} \epsilon (\omega f +\Delta f - \frac{1}{\tilde{m}} fr_1) (x_i) 
 + O(\epsilon^2),  
\end{align*}
and taking $f=1$ gives  that
\[
\int_{\calM} K_\epsilon( x_i, y)  \frac{ p }{ \tilde{p}_\epsilon}(y)  dV(y) 
=  m_0 +  \frac{m_2}{2} \epsilon (\omega - \frac{1}{\tilde{m}} r_1) (x_i) 
 + O(\epsilon^2).
\]
Putting together and subtracting the two terms in \eqref{eq:exp-Yj-ptrate-density-correct-term1}
proves that 
$\E Y_j = 
 \frac{m_2}{2} \epsilon \Delta f (x_i) 
 + O(\epsilon^2)$.
\end{proof}

\subsection{Dirichlet form convergence of density-corrected graph Laplacian}

The graph Dirichlet form of density-corrected graph Laplacian is defined as 
\begin{equation}\label{eq:def-tildeENu}
\tilde{E}_N(u):=\frac{1}{ \frac{m_2}{  2 m_0^2}\epsilon } u^T( \tilde{D} - \tilde{W}) u 
= \frac{1}{ \frac{m_2}{   m_0^2}\epsilon } \sum_{i,j=1}^N \tilde{W}_{i,j} (u_i - u_j)^2 
= \frac{1}{ \frac{m_2}{   m_0^2}\epsilon } \sum_{i,j=1}^N W_{i,j} \frac{ (u_i - u_j)^2  }{D_i D_j}.
\end{equation}
We establish the counter part of Theorem \ref{thm:form-rate}, which achieves the same form rate.
The theorem is for general differentiable $h$, which can be of independent interest.

\begin{theorem}
\label{thm:form-rate-density-correction}
Under Assumptions \ref{assump:M-p} and \ref{assump:h-C2-nonnegative}, 
if as $N \to \infty$, $ \epsilon \to 0+$, $ \epsilon^{d/2 } N = \Omega( \log N)$,
then for any $f \in C^{\infty} ({\calM})$, 
 when $N$ is sufficiently large,
w.p. $> 1- 2 N^{-9}-2 N^{-10}$,
\[
\tilde{E}_N( \rho_X f )
= \langle f, -\Delta f \rangle
+ O_{p,f} \left( \epsilon,  \sqrt{  \frac{  \log N    }{ N \epsilon^{d/2  }} } \right).
\]
\end{theorem}

\begin{proof}[Proof of Theorem \ref{thm:form-rate-density-correction}] 
By definition \eqref{eq:def-tildeENu},
\[
\tilde{E}_N( \rho_X f ) 
=  \frac{1}{ \frac{m_2}{   m_0^2}\epsilon } \frac{1}{N^2} \sum_{i,j=1}^N W_{i,j} \frac{ (f(x_i) - f(x_j))^2  }{ \frac{D_i}{N} \frac{D_j}{N}}.
\]
The following lemma (proved in  Appendix \ref{app:proofs-density-corrected}) 
makes use of the concentration of $D_i/N$ to reduce the graph Dirichlet form to be a V-statistics up to a relative error at the form rate.

\begin{lemma}\label{lemma:tildeENu-V-stat}
Under the good event in  Lemma \ref{lemma:Di-concen-eps2} 1), 
\[
\tilde{E}_N(  u )
= 
\left( \frac{1}{  m_2[h] \epsilon } 
\frac{1}{N^2} \sum_{i,j=1}^N W_{i,j} \frac{ (u_i -  u_j )^2  }{  p(x_i)p(x_j) } \right)
\left(1 + O(\epsilon,  \sqrt{ \frac{\log N}{N \epsilon^{d/2}} }) \right),
\quad  \forall u \in \R^N,
\]
and the constant in big-$O$ is determined by $(\calM, p)$ and uniform for all $u$.
\end{lemma}

We consider under the good event in Lemma \ref{lemma:Di-concen-eps2}  1),
which is called $E_1$ and  happens w.p. $> 1- 2 N^{-9}$.
Then applying Lemma \ref{lemma:tildeENu-V-stat} with $u = \rho_X f$, we have that 
\begin{equation}
\tilde{E}_N( \rho_X f )
= 
\left\{ \frac{1}{  m_2 \epsilon } 
\frac{1}{N^2} \sum_{i,j=1}^N W_{i,j} \frac{ (f(x_i) - f(x_j))^2  }{  p(x_i)p(x_j) } \right\}
(1 + O(\epsilon,  \sqrt{ \frac{\log N}{N \epsilon^{d/2}} }) )
=:  \textcircled{3} (1 + O(\epsilon,  \sqrt{ \frac{\log N}{N \epsilon^{d/2}} }) )
\label{eq:form-pf-1}
\end{equation}
The term
 $\textcircled{3}$ in \eqref{eq:form-pf-1} equals $\frac{1}{N^2} \sum_{i,j=1}^N V_{i,j} $,
where
$V_{i,j} : = \frac{1}{  m_2 \epsilon }  K_\epsilon( x_i, x_j) \frac{ (f(x_i) - f(x_j))^2  }{  p(x_i)p(x_j) }$,
and $ V_{i,i} =0$.
We follow the same approach as in the proof of Theorem {3.4} in \cite{cheng2020convergence} to analyze this V-statistic,
and show  that  (proof in Appendix \ref{app:proofs-density-corrected})
\begin{equation}\label{eq:proof-form-density-correct-Vstat}
\{\text{ $\textcircled{3}$ in \eqref{eq:form-pf-1} }\}
 = 
   \langle f, - \Delta f \rangle + O_{f,p}(\epsilon, \sqrt{ \frac{\log N }{N \epsilon^{d/2}} }) .
\end{equation}
Back to \eqref{eq:form-pf-1}, we have shown that under $E_1 \cap E_3$, 
\begin{align*}
\tilde{E}_N( \rho_X f )
& =  \textcircled{3} (1 + O(\epsilon,  \sqrt{ \frac{\log N}{N \epsilon^{d/2}} }) ) 
 = \left(    \langle f, - \Delta f \rangle + O(\epsilon, \sqrt{ \frac{\log N }{N \epsilon^{d/2}} }) \right)
(1 + O(\epsilon,  \sqrt{ \frac{\log N}{N \epsilon^{d/2}} }) ) \\
& = \langle f, - \Delta f \rangle + O(\epsilon,  \sqrt{ \frac{\log N }{N \epsilon^{d/2}} }),
\end{align*}
and the constant in big-$O$ depends on $\calM$, $f$ and $p$. 
\end{proof}

\subsection{Eigen convergence of $\tilde{L}_{rw}$}

In this subsection, let $\lambda_k$ be eigenvalues of $\tilde{L}_{rw}$ and $v_k$ the associated eigenvectors. 
By \eqref{eq:def-tildeENu}, recall that $\tilde{m} =\frac{m_2}{ 2 m_0}$,
the analogue of \eqref{eq:lambdak-rw-min-max} is the following

\begin{equation}\label{eq:lambdak-rw-density-correct}
\lambda_k
= \min_{ L \subset \R^N, \, dim(L) = k} \sup_{ v \in L,  v \neq 0} 
\frac{ \frac{1}{\epsilon \tilde{m} } v^T(\tilde{D}-\tilde{W})v}{   v^T \tilde{D} v }
= \frac{  \frac{1}{m_0 } \tilde{E}_N(v) }{   v^T \tilde{D} v },
\quad 1 \le k \le N.
\end{equation}
The methodology is same as before,
with  a main difference  in the definition of the heat interpolation mapping with weights $p(x_j)$ as in \eqref{eq:def-tilde-Ir}.
This gives to the $p$-weighted quadratic form $\tilde{q}_s(u)$ defined in \eqref{eq:def-tilde-qs}, 
for which we derive the concentration argument of for $\tilde{q}^{(0)}_s$ in \eqref{tildeq0-u-E0'}
and the upper bound of $\tilde{q}^{(2)}_s$ in Lemma \ref{lemma:qs2-UB-density-correct}.
The other difference is that the $\tilde{D}$-weighted 2-norm is considered because the eigenvectors are $\tilde{D}$-orthogonal. 
All the proofs of the Steps 0-3 {and Theorem \ref{thm:refined-rates-rw-density-correct}} are  left to Appendix \ref{app:proofs-density-corrected}.
\vspace{5pt}

\noindent
\underline{Step 0}. 
We first establish eigenvalue UB based on
Lemma \ref{lemma:Di-concen-eps2} and the form convergence in Theorem \ref{thm:form-rate-density-correction}.

\begin{proposition}[Eigenvalue UB of $\tilde{L}_{rw}$]
\label{prop:eigvalue-UB-rw-density-correct}
Under Assumptions \ref{assump:M-p} and \ref{assump:h-C2-nonnegative},
for fixed $K \in \mathbb{N}$, 
Suppose $0 = \mu_1  <\cdots < \mu_{K} < \infty$ are all of single multiplicity.
If as $N \to \infty$, $\epsilon \to 0+$, and $\epsilon^{d/2} = \Omega( \frac{\log N}{N} ) $, 
then for sufficiently large $N$, 
w.p. $> 1- 4 N^{-9} - 4 K^2 N^{-10}$,  $\tilde{L}_{rw}$ is well-defined, and 
\[
\lambda_k  
\le \mu_k + O \left(\epsilon, \sqrt{ \frac{\log N}{N  \epsilon^{d/2} } } \right) ,
 \quad k=1,\cdots, K.
\]
\end{proposition}

\noindent
\underline{Step 1}. 
Eigenvalue crude LB. We prove with the $p$-weighted interpolation mapping defined as
\begin{equation}\label{eq:def-tilde-Ir}
\tilde{I}_r [u] = \frac{1}{N} \sum_{j=1}^N \frac{u_j}{ p(x_j)} H_r( x, x_j) = I_r [ \tilde{u}],
\quad \tilde{u}_i = u_i/p(x_i).
\end{equation}
Then, same as before, 
$\langle \tilde{I}_r [u], \tilde{I}_r  [u] \rangle = q_{\delta \epsilon} (\tilde{u})$,
and 
$\langle \tilde{I}_r [u], Q_t \tilde{I}_r  [u] \rangle = q_{ \epsilon} (\tilde{u})$,
where 
for $s > 0$,
\begin{equation}\label{eq:def-tilde-qs}
\begin{split}
\tilde{q}_s( u)  
& := \frac{1}{N^2} \sum_{i,j=1}^N \frac{ \calH_s(x_i, x_j) }{p(x_i) p(x_j) }u_i u_j
 = q_s( \tilde{u})   
= \tilde{q}^{(0)}_s(u) - \tilde{q}^{(2)}_s(u), \\
\tilde{q}^{(0)}_s(u)  
& := \frac{1}{N} \sum_{i=1}^N u_i^2 \left(  \frac{1}{N} \sum_{j=1}^N  \frac{ \calH_s(x_i, x_j) }{p(x_i) p(x_j) } \right), 
\quad
\tilde{q}^{(2)}_s(u) 
 := \frac{1}{2 N^2} \sum_{i,j=1}^N\frac{ \calH_s(x_i, x_j) }{p(x_i) p(x_j) } (u_i - u_j)^2.
\end{split} 
\end{equation}

\begin{proposition}[Initial crude eigenvalue LB of $\tilde{L}_{rw}$]
\label{prop:eigvalue-LB-crude-rw-density-correct}
Under Assumption \ref{assump:M-p},
$h$ is Gaussian.
For fixed $k_{max} \in \mathbb{N}$, $K = k_{max}+1$,
and $\mu_k$, $\epsilon$ and  $N$ satisfy the same condition as in Proposition \ref{prop:eigvalue-LB-crude},
where the definition of $c_K$ is the same except that $c$ is a constant depending on $(\calM,p)$.
Then, for sufficiently large $N$,  w.p.$> 1- 4K^2 N^{-10} - 8N^{-9}$, 
$\lambda_k > \mu_k - \gamma_K$,
 for $k=2,\cdots, K$.
\end{proposition}

\noindent
\underline{Steps 2-3}.
We prove eigenvector consistency and refined eigenvalue convergence rate. Define 
\begin{equation}\label{eq:def-u-tildeD-norm}
\| u \|_{\tilde{D}}^2 : = \sum_{i=1}^N u_i^2 \tilde{D}_i, \quad \forall u \in \R^N.
\end{equation}
The proof uses same techniques as before, and the differences are in handling the $\tilde{D}$-orthogonality of the eigenvectors 
and using the concentration arguments in Lemma \ref{lemma:Di-concen-eps2}.
Same as before, extension to when $\mu_k$ has greater than 1 multiplicity is possible (Remark \ref{rk:multiplicity}).

\begin{theorem}[eigen-convergence of $\tilde{L}_{rw}$]
\label{thm:refined-rates-rw-density-correct}
\rev{Under Assumption \ref{assump:M-p},}
$h$ being Gaussian,
and $k_{max}$, $K$, $\mu_k$, $\epsilon$ same
as in Theorem \ref{thm:refined-rates},
where the definition of $c_K$ is the same except that $c$ is a constant depending on $(\calM,p)$.
Consider first $k_{max}$ eigenvalues and eigenvectors of $\tilde{L}_{rw}$,
$\tilde{L}_{rw} v_k = \lambda_k v_k$, 
and $v_k$ are normalized s.t. 
$ N \| v_k\|_{\tilde{D}}^2 =1 $.
Define for $1 \le k\le K$,
\[
\tilde{\phi}_k := \rho_X \left(  \frac{ 1}{ \sqrt{ N } } \psi_k \right).
\]
Then, for sufficiently large $N$, 
w.p.$> 1- 4K^2 N^{-10} - ( 4K + 8)N^{-9}$, 
$ \| v_k \|_2 = \Theta(1)$, and 
the same bounds as in Theorem \ref{thm:refined-rates} hold
for $ | \mu_k - \lambda_k |$  and $ \| v_k - \alpha_k \tilde{\phi}_k \|_2$, for $ 1 \le k \le k_{max}$,
with certain scalars $\alpha_k$  satisfying $|\alpha_k| = 1+o(1)$,
\end{theorem}

\begin{figure}
\captionsetup{width=0.9\linewidth}
\hspace{-40pt}
\begin{subfigure}{0.5 \textwidth}
\includegraphics[trim =  40 0 40 0, clip, height=.48\textwidth]{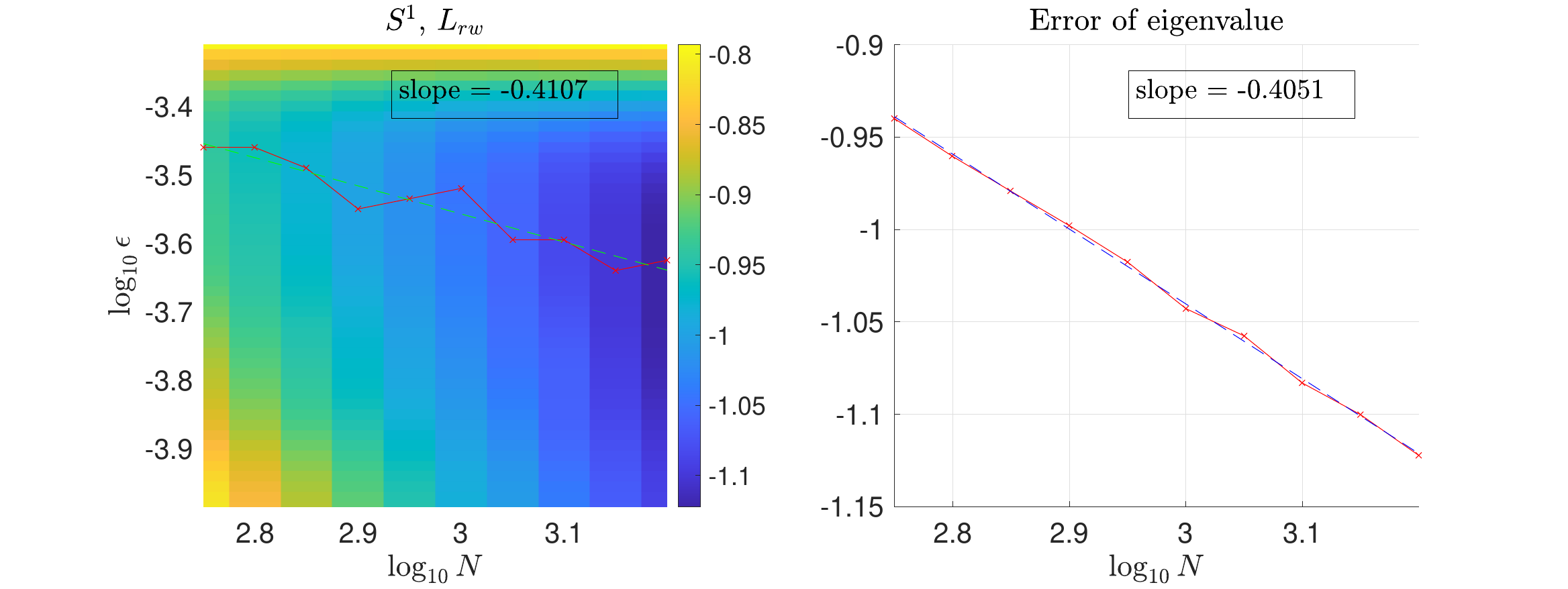} 
\caption{}
\end{subfigure}
\hspace{20pt}
\begin{subfigure}{0.5 \textwidth}
\includegraphics[trim =  40 0 40 0, clip, height=.48\linewidth]{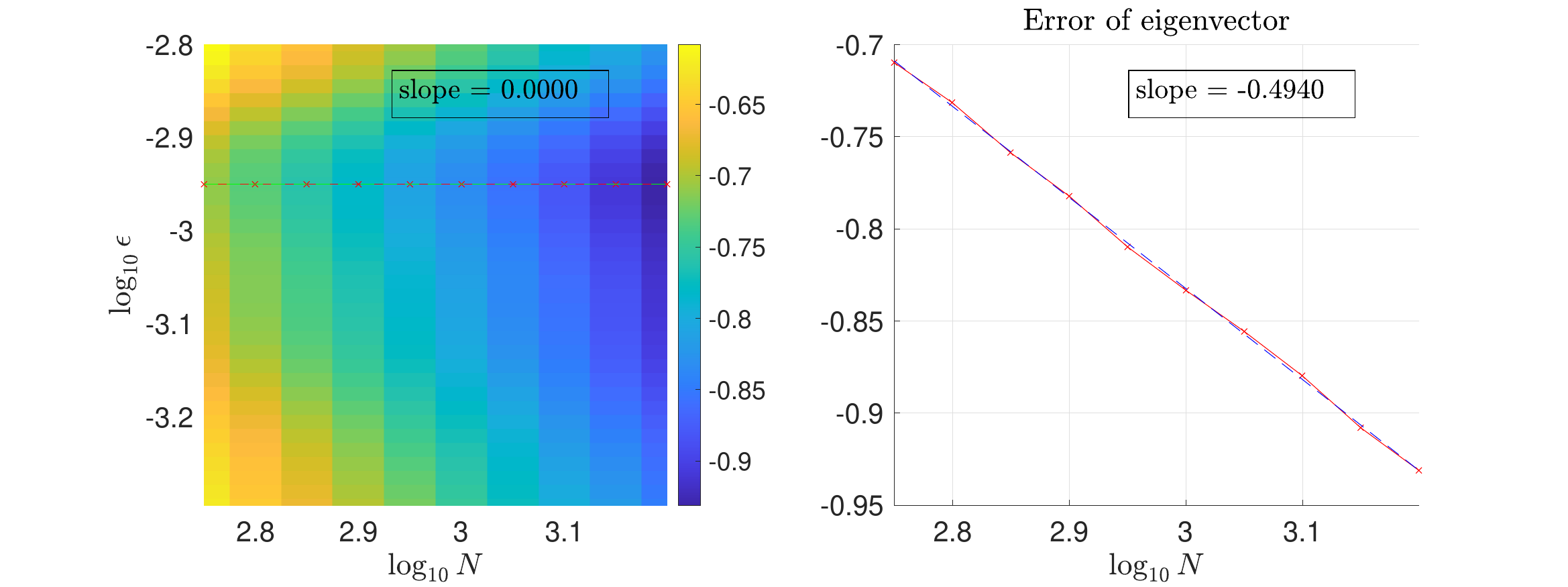} 
\caption{}
\end{subfigure}

\caption{
\scriptsize
Data points are sampled uniformly on $S^1$ embedded in $\R^4$.
(a) The eigenvalue relative error $\text{RelErr}_{\lambda}$, 
visualized (in $\log_{10}$) as a field on a grid of ($\log_{10}$)   $N$ and $\epsilon$,
$k_{max} = 9$.
The red curve on the left plot indicates the post-selected optimal $\epsilon$ which minimizes the error,
and that minimal error as a function of $N$ is plotted on the right in log-log scale. 
(b) Same plot as (a) for eigenvector relative error  $\text{RelErr}_{v}$.
The relative errors are defined in \eqref{eq:RelErr-eig}.
The empirical errors are averaged over 500 runs of experiments,
and the log error values are smoothed over the grid for better visualization.
Plots of the raw values are shown in Fig. \ref{fig:S1-Lrw-nosmooth}.
}
\label{fig:S1-Lrw}
\end{figure}

\begin{figure}[t]
\captionsetup{width=0.9\linewidth}
\hspace{-40pt}
\begin{subfigure}{0.5 \textwidth}
\includegraphics[trim =  40 0 40 0, clip, height=.48\textwidth]{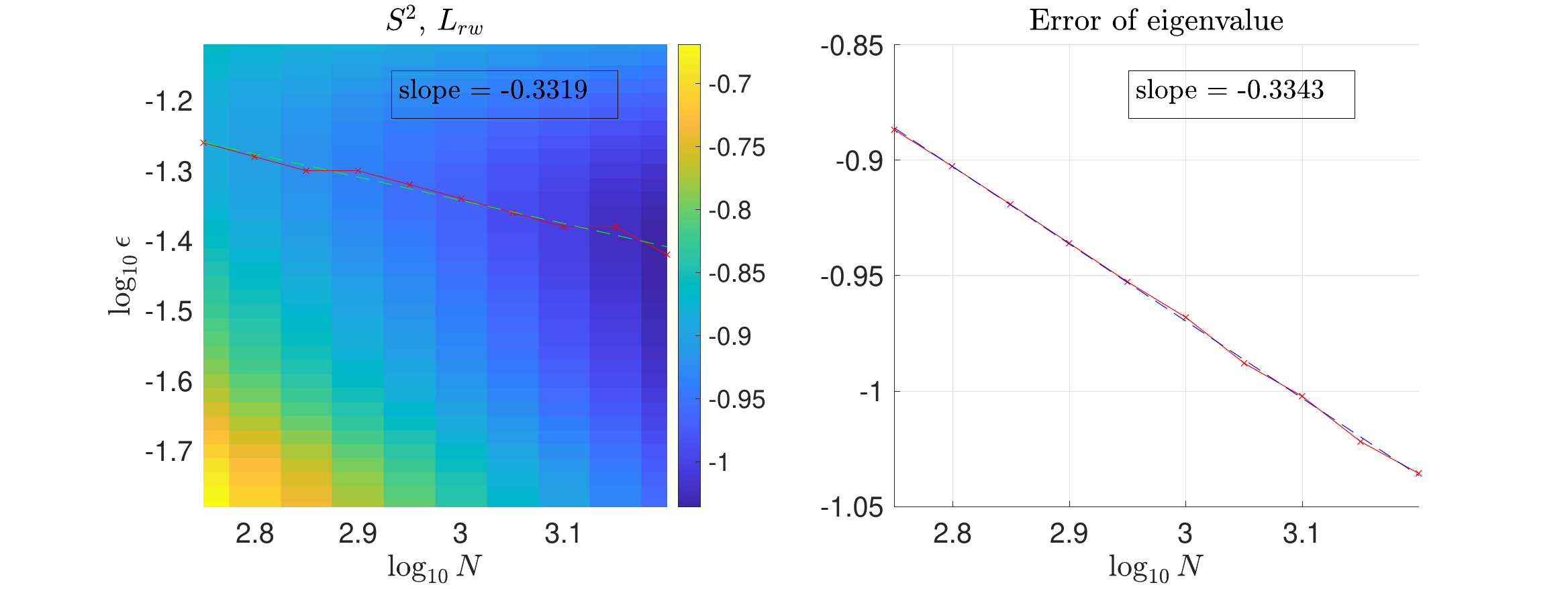} 
\caption{}
\end{subfigure}
\hspace{20pt}
\begin{subfigure}{0.5 \textwidth}
\includegraphics[trim =  40 0 40 0, clip, height=.48\linewidth]{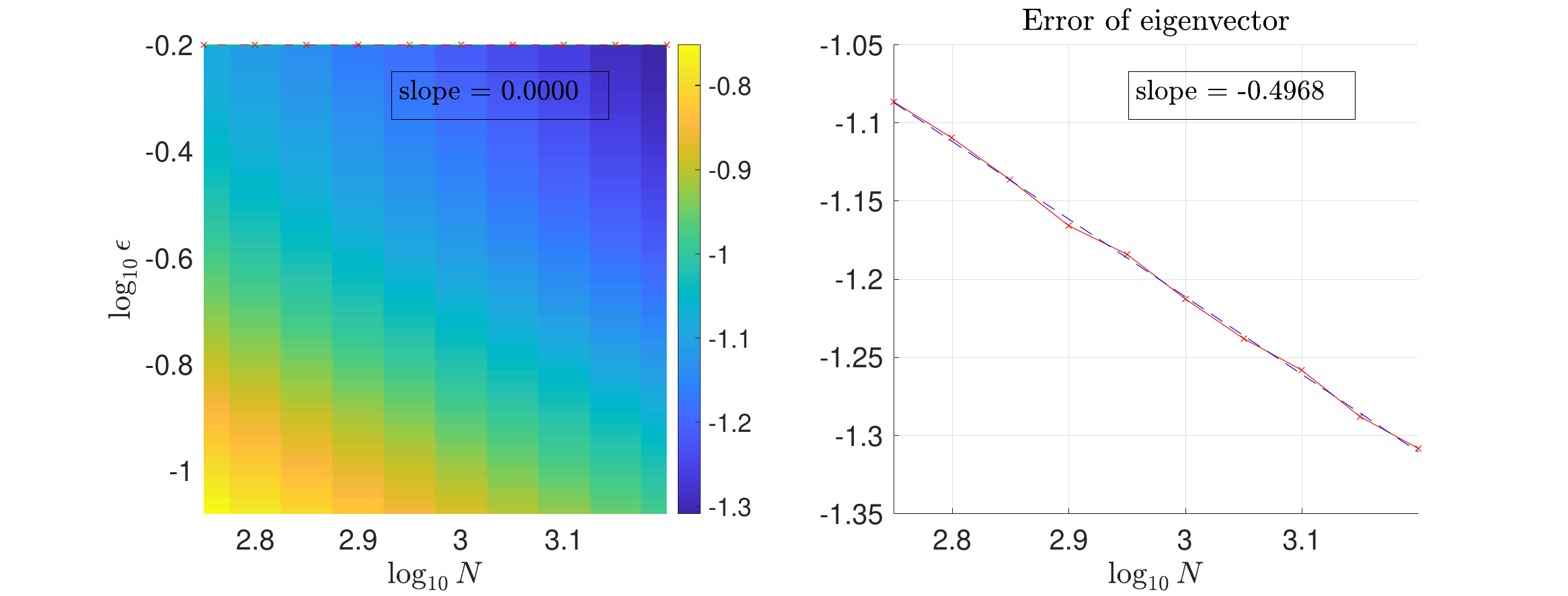} 
\caption{}
\end{subfigure}
\caption{
\scriptsize
Data points are sampled uniformly on $S^2$ embedded in $\R^3$,
same plots as Fig. \ref{fig:S1-Lrw}. 
$k_{max} = 9$,
and the plots of raw values are shown in Fig. \ref{fig:S2-Lrw-nosmooth}.
}
\label{fig:S2-Lrw}
\end{figure}

\section{Numerical experiments}\label{sec:exp}

In this section gives numerical results of point-wise convergence and eigen-convergence of graph Laplacians built from simulated manifold data.
Codes are released at
\url{https://github.com/xycheng/eigconvergence_gaussian_kernel}.

\subsection{Eigen-convergence of $L_{rw}$}

We test on two simulated datasets, which are uniformly sampled on $S^1$ (embedded in $\R^4$, the formula is in Appendix \ref{app:numerics})
and unit sphere $S^2$ (embedded in $\R^3$). 
For both datasets,  we compute over an increasing number of samples $N = \{ 562,  \cdots, 1584 \}$
and  a range of values of $\epsilon $, 
where the grid points of both $N$ and $\epsilon$ are evenly spaced in log scale.
For each value of $N$ and $\epsilon$, we generate $N$ data points, construct the  kernelized matrix
$W_{ij}=K_\epsilon(x_i, x_j)$ as defined in  \eqref{eq:def-K-eps} with Gaussian $h$,
 and compute the first 10 eigenvalues $\lambda_k $ and eigenvectors $v_k$ of $L_{rw}$.
The errors are computed by 
\begin{equation}\label{eq:RelErr-eig}
\text{RelErr}_{\lambda} = \sum_{k=2}^{k_{max}}  \frac{|\lambda_k - \mu_k|}{\mu_k},
\quad
\text{RelErr}_{v} = \sum_{k=2}^{k_{max}}  \frac{ \| v_k - \phi_k \|_2}{ \| \phi_k \|_2},
\end{equation}
where $\phi_k $ is as defined by \eqref{eq:def-phik}. 
 The experiment is repeated for 500 replicas from which the averaged empirical errors are computed. 
For the data on $S^1$, $\epsilon = \{ 10^{-2.8}, \cdots, 10^{-4} \}$.
The manifold (in first 3 coordinates) is illustrated in Fig. \ref{fig:S1-tildeLrw-pt}(a) but the density is uniform here.
See more details  in Appendix \ref{app:numerics}.
For the data on $S^2$, 
$\epsilon = \{ 10^{-0.2}, \cdots, 10^{-1.8} \}$.
These ranges are chosen so that the minimal error over $\epsilon$  for each $N$ are observed, at least for $\text{RelErr}_{\lambda}$. 
Note that for $S^1$, the population eigenvalues starting from $\mu_2$ are of multiplicity 2,
and for $S^2$, the multiplicities are  3, 5, $\cdots$.

The results are shown in Figures \ref{fig:S1-Lrw} and \ref{fig:S2-Lrw}.
For data on $S^1$, 
Fig. \ref{fig:S1-Lrw} (a) shows that $\text{RelErr}_{\lambda}$ as a function of $N$ (with post-selected best $\epsilon$)
shows a convergence order of about $N^{- 0.4 }$,
which is consistent with the theoretical bound of $N^{-1/(d/2+2)}$ in Theorem \ref{thm:refined-rates-rw}, since $d=1$ here. 
In the left plot of colored field, 
 the log error values are smoothed over the grid of $N$ and $\epsilon$, 
 and the best $\epsilon$ scales with $N$ as about $N^{-0.4}$.
The empirical scaling of optimal $\epsilon$ is less stable to observe:
depending on the level of smoothing,
the slope of $\log_{10} \epsilon$ varies between -0.2 and -0.5 (the left plot),
while the slope for best (log) error is always about -0.4 (the right plot). 
The result without smoothing is shown in Fig. \ref{fig:S1-Lrw-nosmooth}.
The eigenvector error in  Fig. \ref{fig:S1-Lrw}(b)  shows an order of  about $N^{-0.5}$, which is better than the theoretical prediction. 
For the data on $S^2$,
the eigenvalue convergence shows an order of about $N^{-0.33}$,
in agreement with the theoretical rate of $N^{-1/(d/2+2)}$ when $d=2$.
The eigenvector error again shows an order of about $N^{-0.5}$ which is better than theory.
The small error of eigenvector estimation at very large value of $\epsilon$ may be due to the symmetry of the simple manifolds $S^1$ and $S^2$. 
In both experiments, the eigenvector estimation prefers a much larger value of $\epsilon$ than the eigenvalue estimation, 
which is consistent with the theory.

\begin{figure}[t]
\captionsetup{width=0.9\linewidth}
\hspace{-30pt}
\begin{subfigure}{0.24 \textwidth}
\includegraphics[trim =  10 0 10 0, clip, height=.94\textwidth]{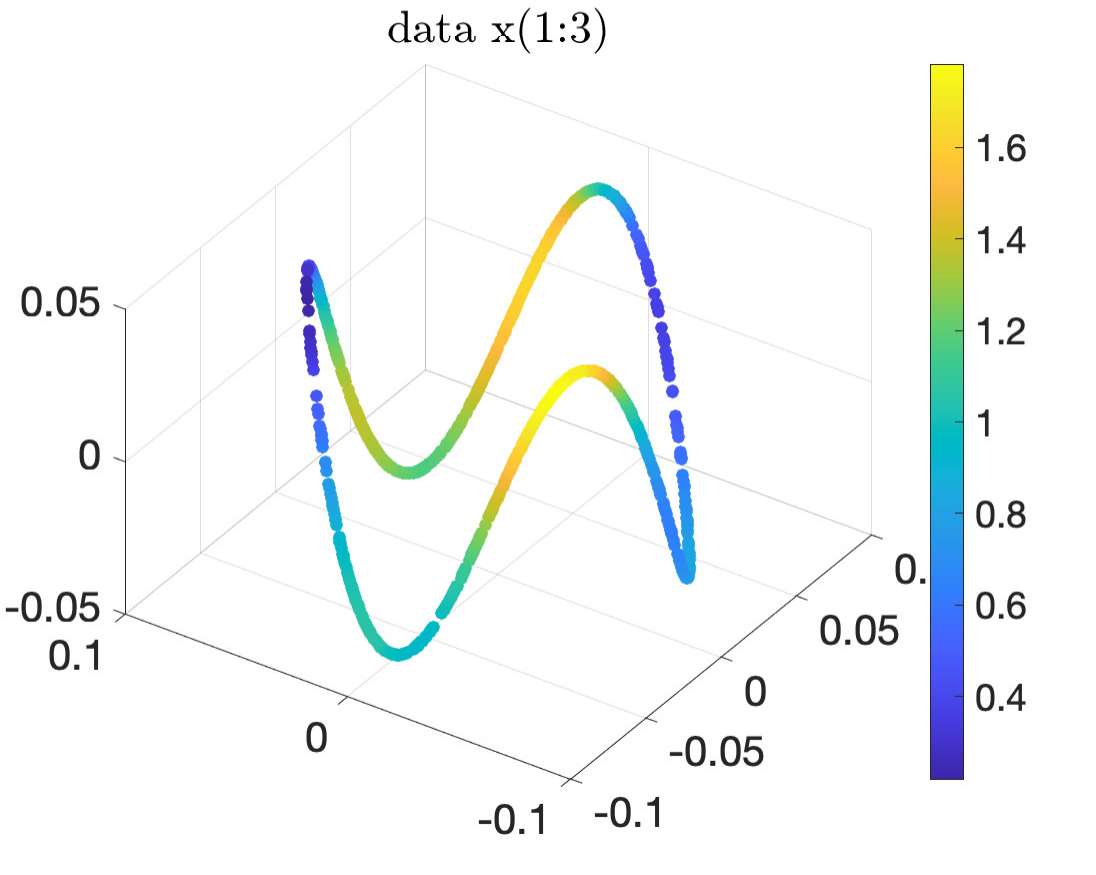} 
\caption{}
\end{subfigure}
\hspace{5pt}
\begin{subfigure}{0.24 \textwidth}
\includegraphics[trim =  10 0 10 0, clip, height=.93\linewidth]{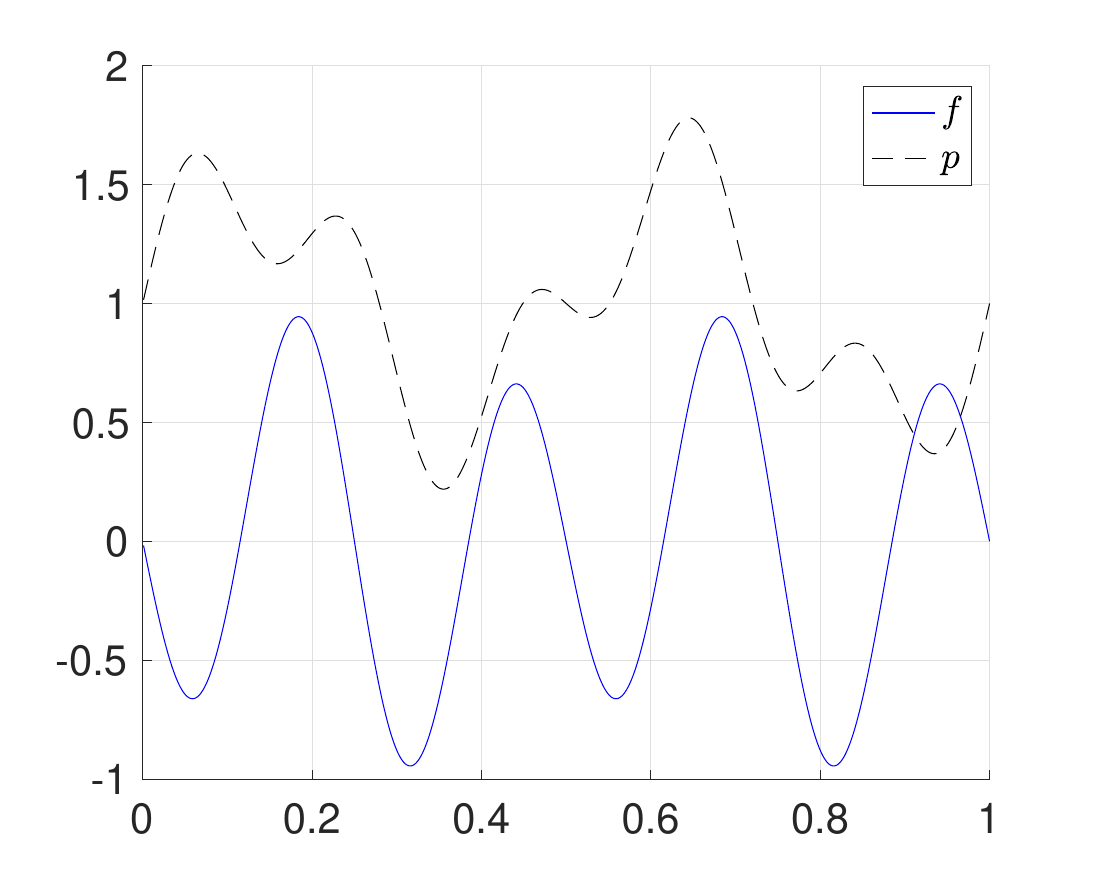} 
\caption{}
\end{subfigure}
\hspace{5pt}
\begin{subfigure}{0.24 \textwidth}
\includegraphics[trim =  10 0 10 0, clip, height=.93\linewidth]{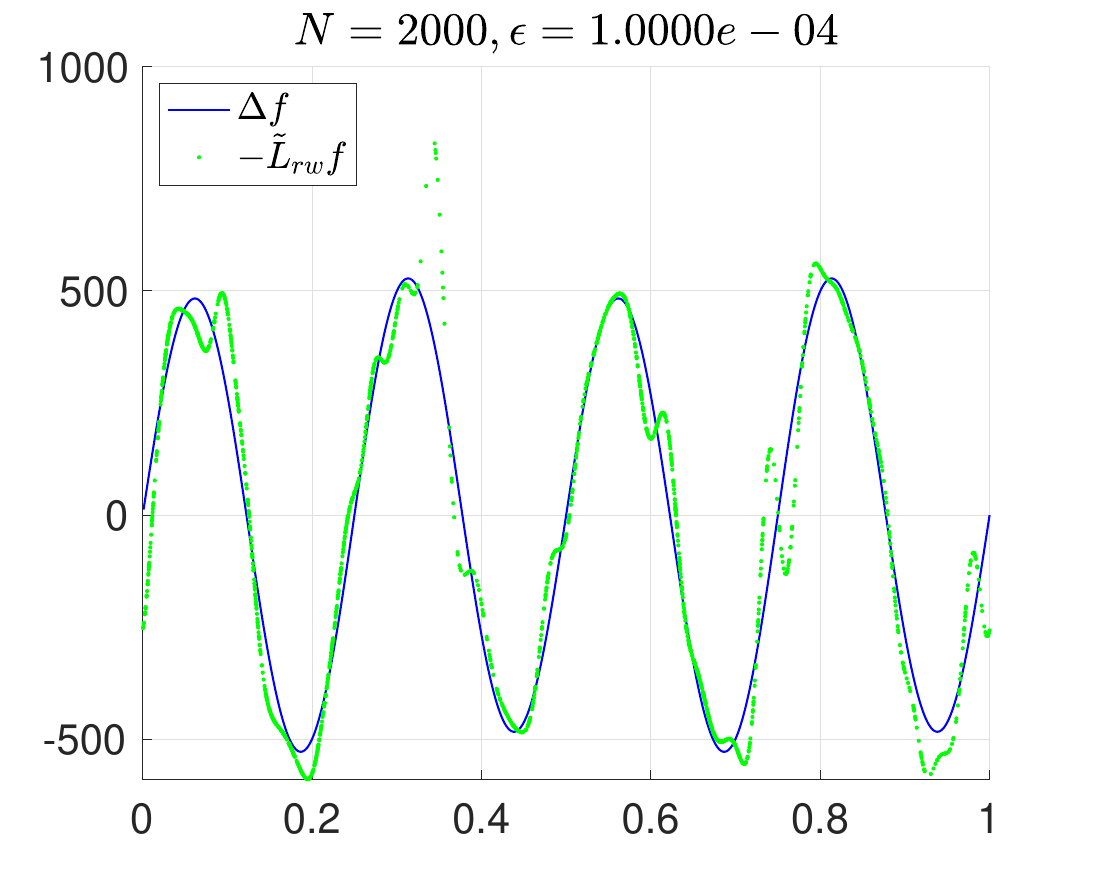} 
\caption{}
\end{subfigure}
\hspace{5pt}
\begin{subfigure}{0.249 \textwidth}
\includegraphics[trim =  40 0 10 0, clip, height=.96\textwidth]{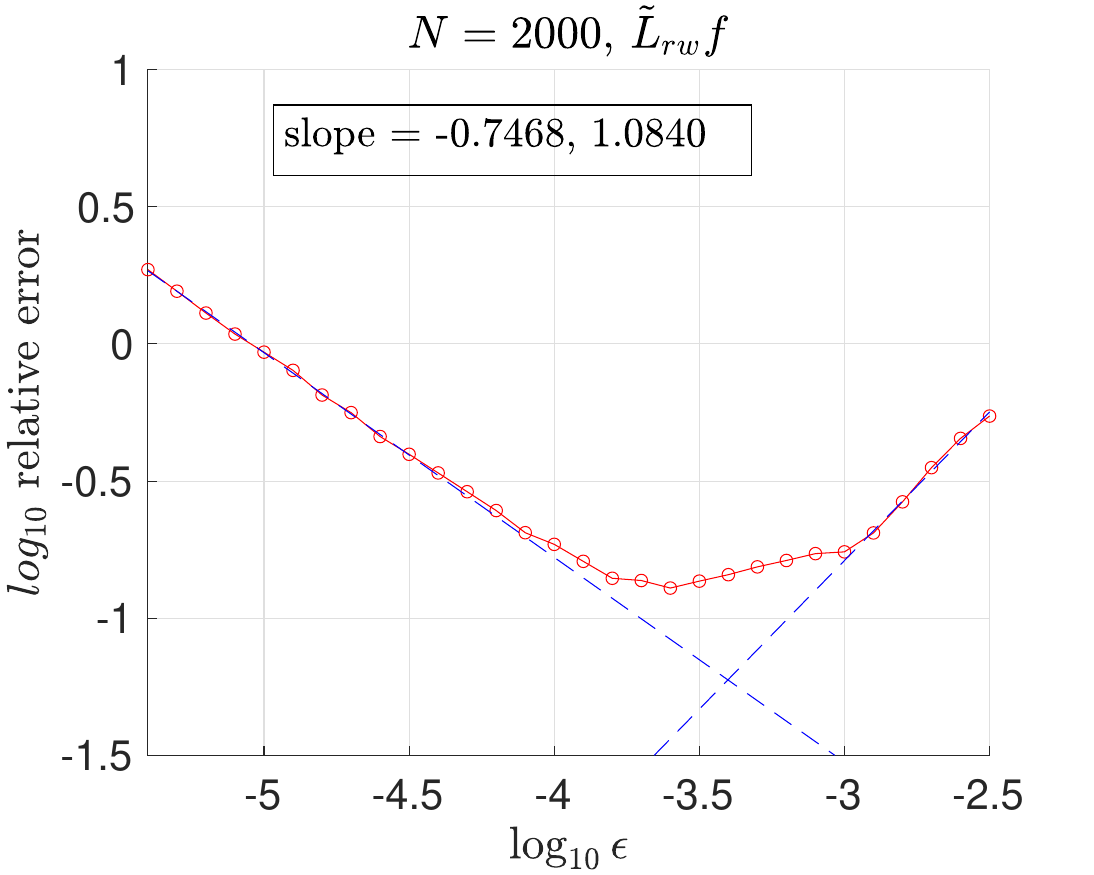} 
\caption{}
\end{subfigure}
\caption{
\scriptsize
(a) Random sampled data on $S^1$ embedded in $\R^4$, the first 3 coordinates are shown, and colored by the density. 
(b) Density $p$ and the test function $f$ plotted as a function of intrinsic coordinate (arc-length) on $[0,1)$ of $S^1$. 
(c) One realization of $\tilde{L}_{rw} (\rho_X f)$ plotted in comparison with the true function of $\rho_X (\Delta f)$.
(d) Log relative error $\log_{10} \text{RelErr}_{pt}$, as defined in \eqref{eq:def-err-pt}, computed over a range of values of $\epsilon$, 
averaged over 50 runs of repeated experiments. 
The two fitted lines show the approximate scaling of $\text{RelErr}_{pt}$ at small $\epsilon$, where variance error dominates, and at large $\epsilon$, where bias error dominates. 
}
\label{fig:S1-tildeLrw-pt}
\end{figure}

\begin{figure}
\captionsetup{width=0.9\linewidth}
\hspace{-40pt}
\begin{subfigure}{0.5 \textwidth}
\includegraphics[trim =  40 0 40 0, clip, height=.48\textwidth]{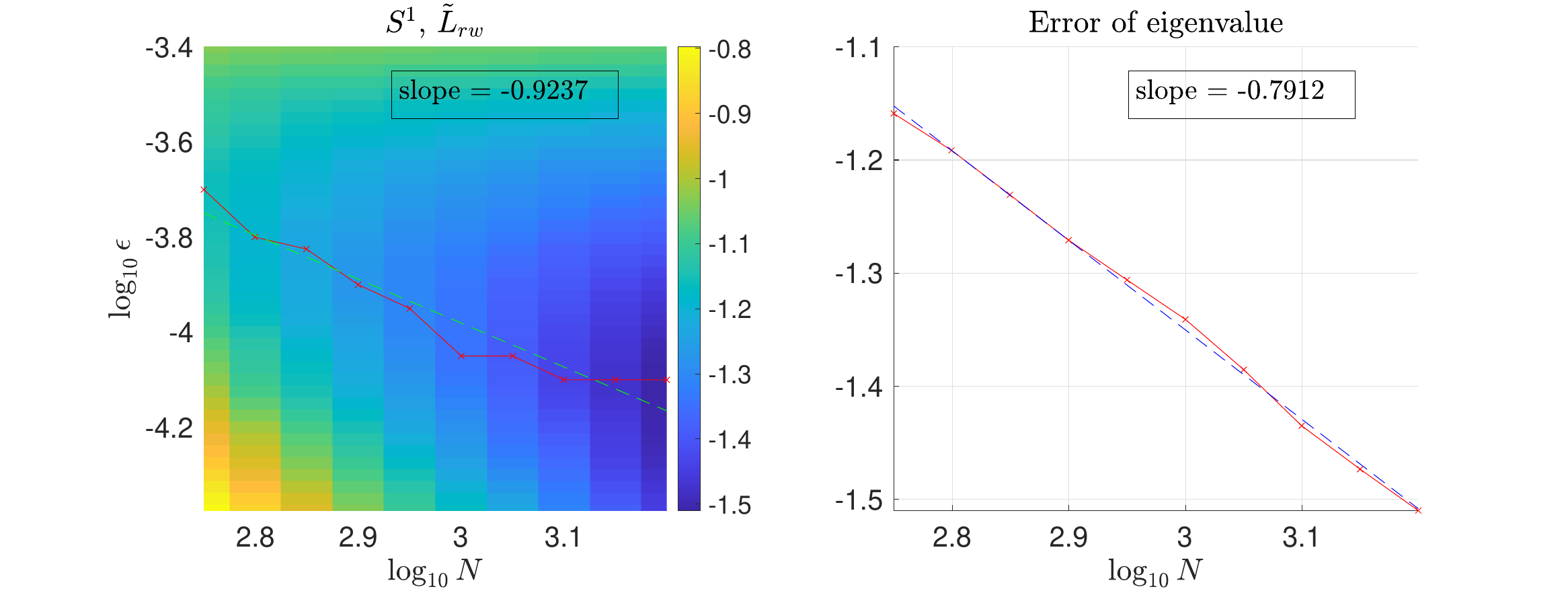} 
\caption{}
\end{subfigure}
\hspace{20pt}
\begin{subfigure}{0.5 \textwidth}
\includegraphics[trim =  40 0 40 0, clip, height=.48\linewidth]{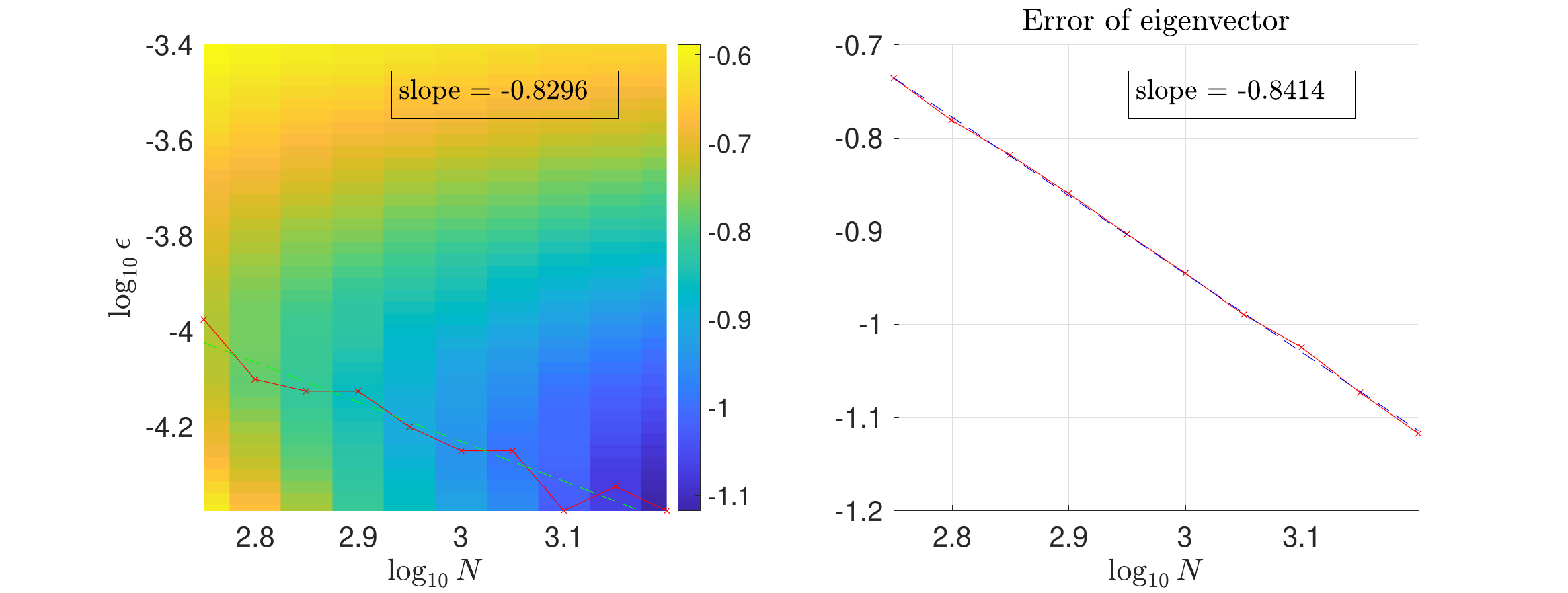} 
\caption{}
\end{subfigure}
\caption{
\scriptsize
Same eigenvalue and eigenvector relative error plots as Fig. \ref{fig:S1-Lrw},
where data are non-uniformly sampled on $S^1$ as in Fig. \ref{fig:S1-tildeLrw-pt}(a).
$k_{max} = 9$,
and the plots of raw values are shown in Fig. \ref{fig:S1-tildeLrw-nosmooth}.
}
\label{fig:S1-tildeLrw}
\end{figure}

\subsection{Density-corrected graph Laplacian}\label{subsec:exp-density-correct}

To examine the density-corrected graph Laplacian,
we switch to non-uniform density on $S^1$, illustrated in Fig. \ref{fig:S1-tildeLrw-pt}(a).
We first investigate the point-wise convergence of $-\tilde{L}_{rw} f$ to $\Delta f$,
on a test function $f : S^1 \to \R$, see more details in Appendix \ref{app:numerics}.
The error is computed as 
\begin{equation}\label{eq:def-err-pt}
\text{RelErr}_{pt} =  \frac{ \| -\tilde{L}_{rw} \rho_X f - \rho_X (\Delta f) \|_1 }{ \| \rho_X (\Delta f) \|_1 },
\end{equation}
and the result is shown in Fig. \ref{fig:S1-tildeLrw-pt}.
Theorem \ref{thm:pointwise-rate-dencity-correct}
predicts the bias error to be $O(\epsilon)$
and the variance error to be $O(\epsilon^{-d/4-1/2}) = O(\epsilon^{-3/4})$ since $N$ is fixed, 
which agrees with Fig. \ref{fig:S1-tildeLrw-pt}(d). 

The results of $\text{RelErr}_{\lambda}$ and $\text{RelErr}_{v}$
are shown in Fig. \ref{fig:S1-tildeLrw}.
The order of convergence with best $\epsilon$ appears to be about $N^{-0.8}$ for both eigenvalue and eigenvector errors,
which is better than those of $L_{rw}$ (when $p$ is uniform) in Fig. \ref{fig:S1-Lrw},
and better than the theoretical prediction in Theorem \ref{thm:refined-rates-rw-density-correct}.

\section{Discussion}

The current result may be extended in several directions.
First, for manifold with smooth boundary, 
the random-walk graph Laplacian recovers the Neumann Laplacian \cite{coifman2006diffusion},
and one can expect to prove the spectral convergence as well, such as in \cite{lu2020graph}.
Second, extension to kernel with variable or adaptive bandwidth \cite{berry2016variable,cheng2020convergence},
and other normalization schemes, e.g., bi-stochastic normalization \cite{marshall2019manifold,landa2021doubly,wormell2021spectral},
would be important to improve the robustness against low sampling density and noise in data,
and even the spectral convergence as well. 
Related is the problem of spectral convergence to other manifold diffusion operators, e.g., the Fokker-Planck operator, on $L^2(\calM, p dV)$.
It would also be interesting to extend to more general types of kernel function $h$ which is not Gaussian,
and even not symmetric \cite{wu2018think}, for the spectral convergence. 
{Relaxing the condition on the kernel bandwidth $\epsilon$ can also be useful:
the optimal transport approach was able to show spectral consistency in the regime just beyond graph connectivity, namely when $\epsilon^{d/2} \gg \log N/N$ \cite{calder2019improved}, which is less restrictive than the condition needed by Gaussian kernel in the current paper.
Being able to extend the analysis to very sparse graph is important for applications.
}
At last, further investigation is needed to explain the good spectral convergence observed in experiments, 
particularly that of the eigenvector convergence and the faster rate with density-corrected graph Laplacian.
For the eigenvector convergence, 
the current work focuses on the 2-norm consistency, 
while  the $\infty$-norm consistency as has been derived in \cite{dunson2021spectral,calder2020lipschitz} is also important to study.

\section*{Acknowledgement}
The authors thank Hau-Tieng Wu for helpful discussion. 
Cheng thanks Yiping Lu for helpful discussion on the eigen-convergence problem. 
The work is supported by NSF DMS-2007040.
XC is also partially supported by NSF (DMS-1818945, DMS-1820827, DMS-2134037), NIH (R01GM131642), and the Alfred P. Sloan Foundation.

\bibliographystyle{plain}
\bibliography{kernel.bib}

\normalsize

\appendix

\setcounter{figure}{0} \renewcommand{\thefigure}{A.\arabic{figure}}
\setcounter{table}{0} \renewcommand{\thetable}{A.\arabic{table}}
\setcounter{equation}{0} \renewcommand{\theequation}{A.\arabic{equation}}
\setcounter{remark}{0} \renewcommand{\theremark}{A.\arabic{remark}}
\section{Details of numerical experiments}\label{app:numerics}

In the example of $S^1$ data, the isometric embedding in $\R^4$ is by
\[
\iota (t) = \frac{1}{2\pi \sqrt{5}} \left( \cos(2\pi t), \sin (2\pi t), \frac{2}{3} \cos(2\pi 3t),    \frac{2}{3} \sin(2\pi 3t)  \right),
\]
where $t \in [0,1)$ is the intrinsic coordinate of $S^1$ (arc-length).  
In the example in Section. \ref{subsec:exp-density-correct} where $p$ is not uniform,
$p(t) =1 + \frac{1}{2} \sin( 2\pi 2 t)   +  \frac{0.6}{2}  \sin( 2\pi 5 t)$,
and the test function
$f(t) = 0.2 \sin(4 \pi t)-0.8 \sin(4 \pi 2 t)$.
In the example of $S^2$ data,
sample are on unit sphere in $\R^3$.

In both plots of the raw error data without smoothing, Figures \ref{fig:S1-Lrw-nosmooth} and \ref{fig:S2-Lrw-nosmooth} 
the slope of error convergence rates (about -0.4 and - 0.33) are about the same.
The slope of post-selected optimal (log) $\epsilon$ as a function of (log) $N$ changes,
due to the closeness of the values over the multiple values of $\epsilon$.

\begin{figure}
\captionsetup{width=0.9\linewidth}
\hspace{-40pt}
\begin{subfigure}{0.5 \textwidth}
\includegraphics[trim =  40 0 40 0, clip, height=.48\textwidth]{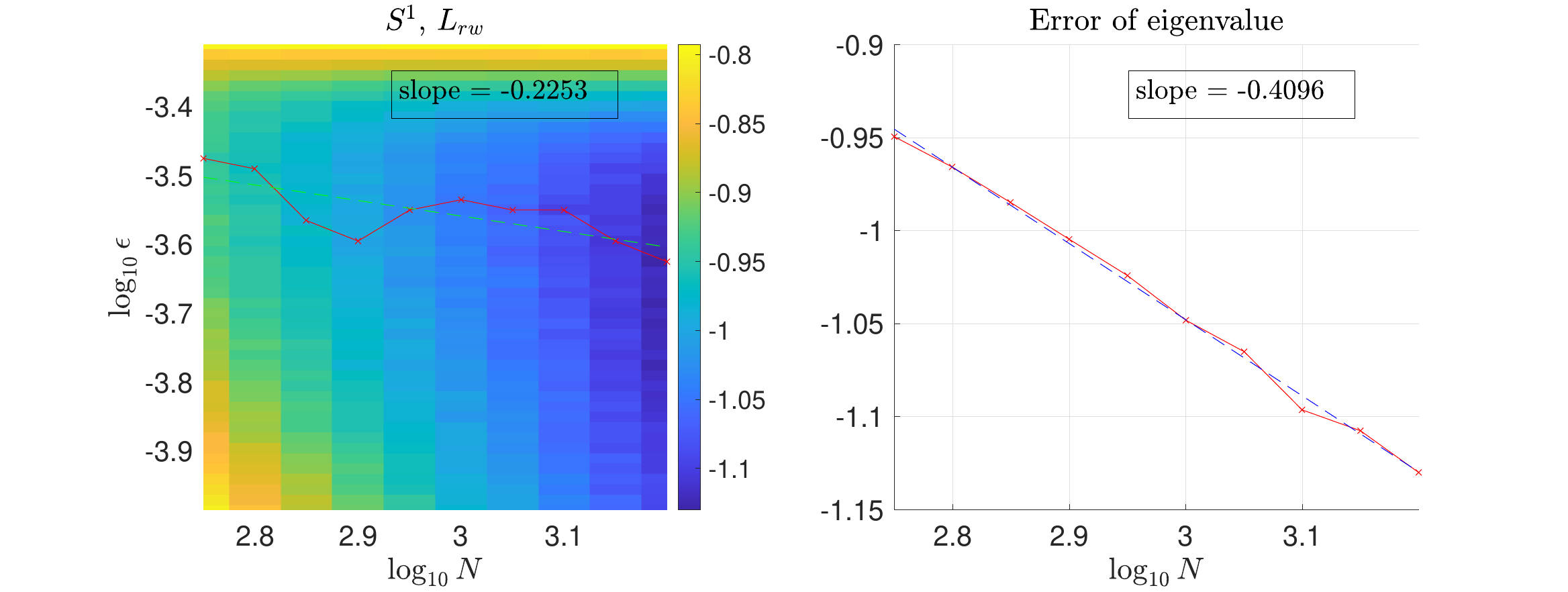} 
\caption{}
\end{subfigure}
\hspace{20pt}
\begin{subfigure}{0.5 \textwidth}
\includegraphics[trim =  40 0 40 0, clip, height=.48\linewidth]{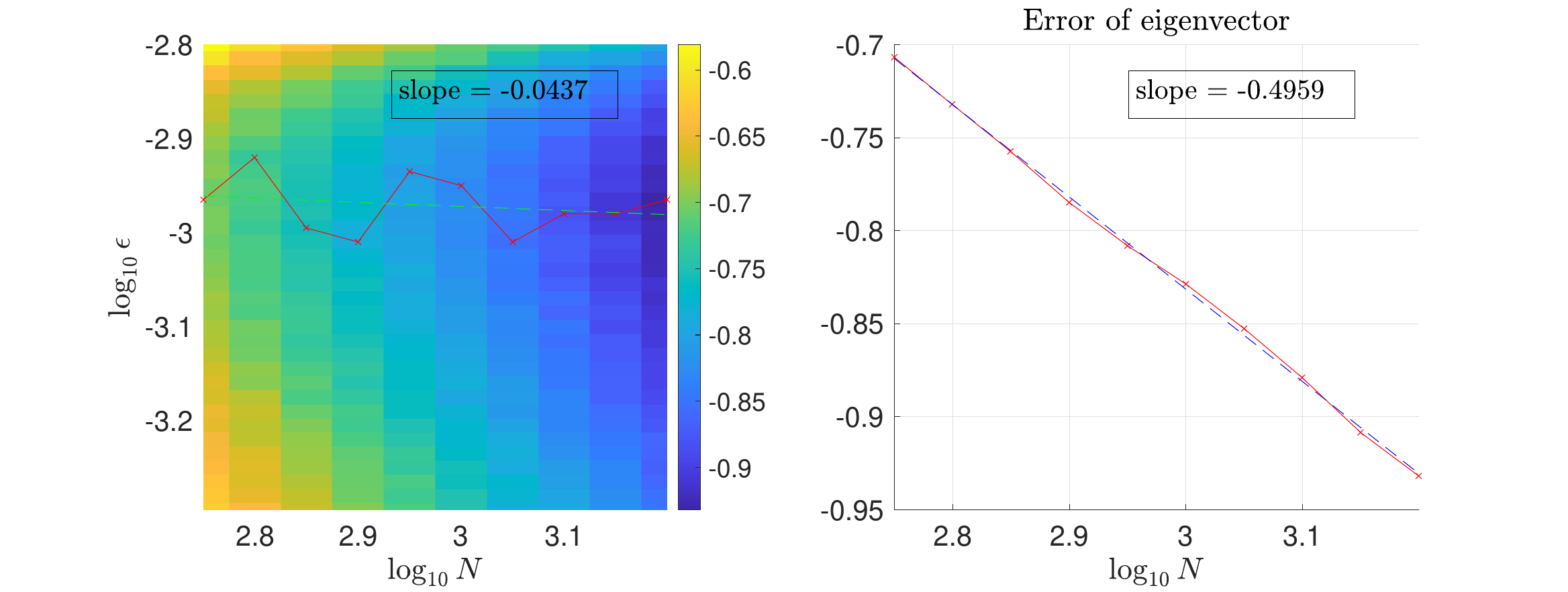} 
\caption{}
\end{subfigure}

\caption{
\scriptsize
Same plots as Fig. \ref{fig:S1-Lrw}
where the log error values on the (log) grid of $N$ and $\epsilon$ are without smoothing. 
}
\label{fig:S1-Lrw-nosmooth}
\end{figure}

\begin{figure}
\captionsetup{width=0.9\linewidth}
\hspace{-40pt}
\begin{subfigure}{0.5 \textwidth}
\includegraphics[trim =  40 0 40 0, clip, height=.48\textwidth]{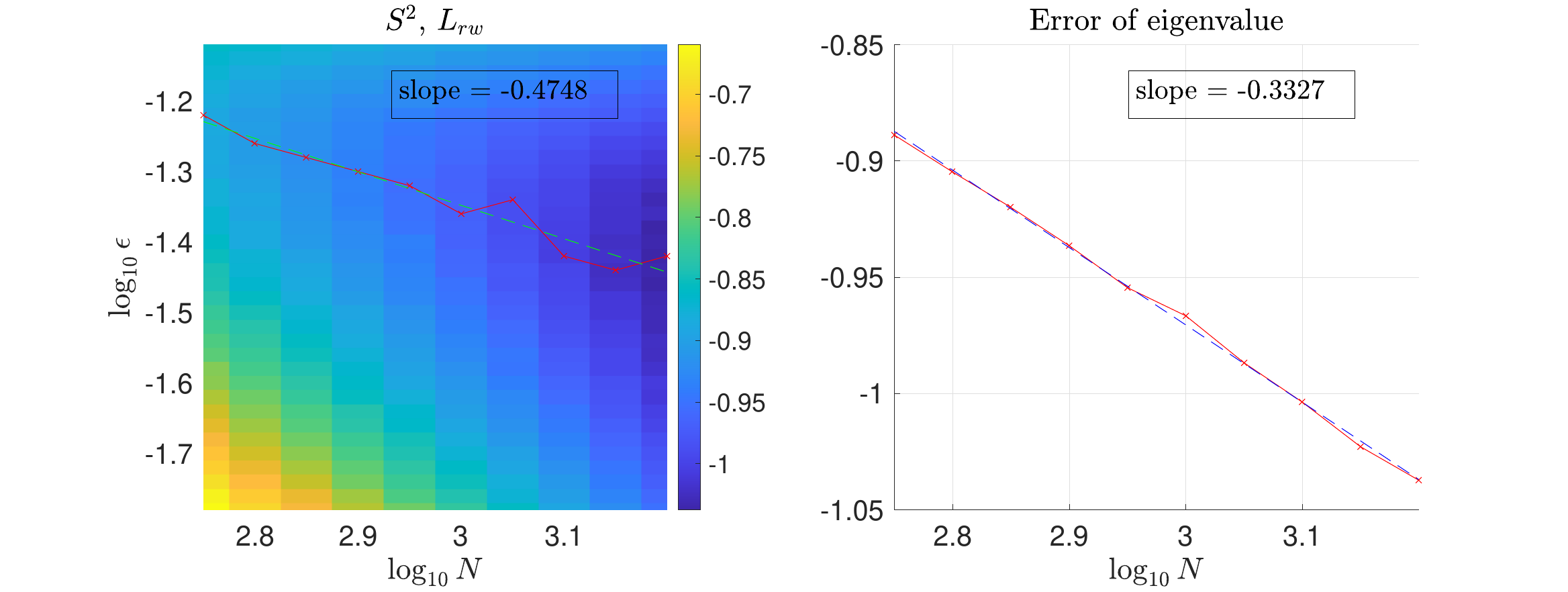} 
\caption{}
\end{subfigure}
\hspace{20pt}
\begin{subfigure}{0.5 \textwidth}
\includegraphics[trim =  40 0 40 0, clip, height=.48\linewidth]{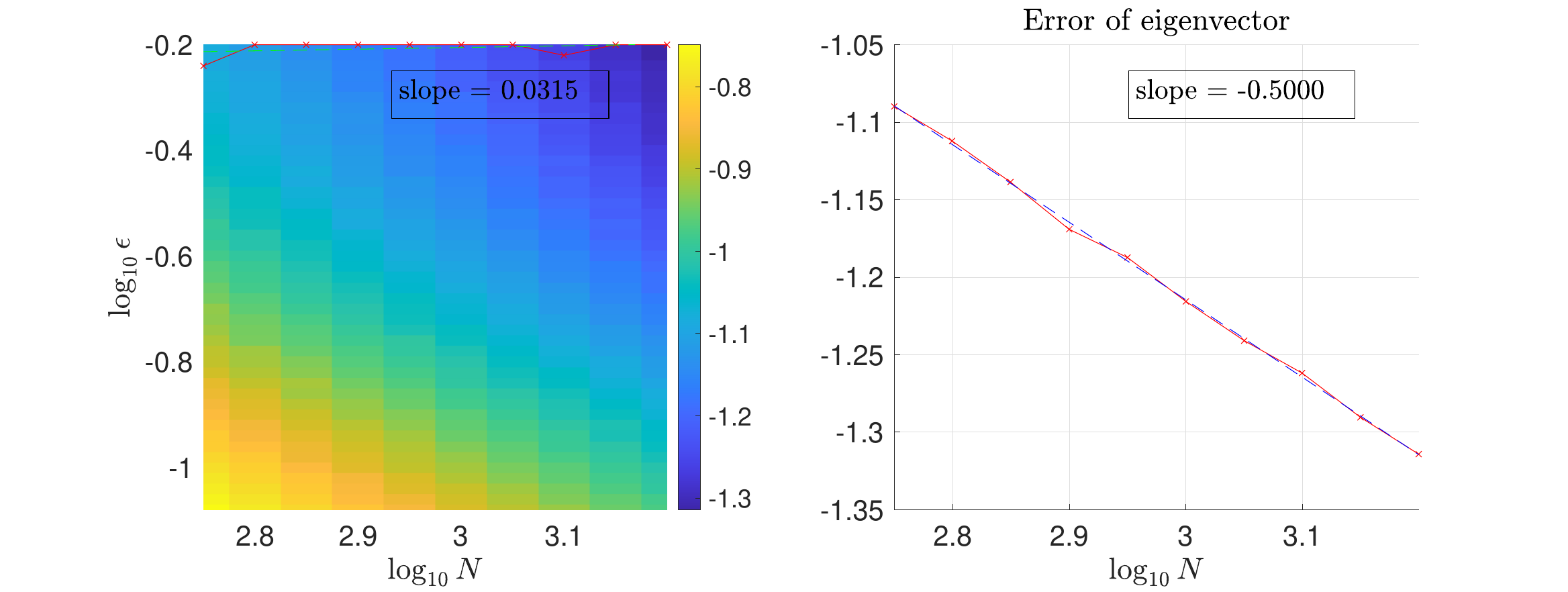} 
\caption{}
\end{subfigure}

\caption{
\scriptsize
Same plots as Fig. \ref{fig:S2-Lrw}
where the log error values on the (log) grid of $N$ and $\epsilon$ are without smoothing.
}
\label{fig:S2-Lrw-nosmooth}
\end{figure}

\begin{figure}
\captionsetup{width=0.9\linewidth}
\hspace{-40pt}
\begin{subfigure}{0.5 \textwidth}
\includegraphics[trim =  40 0 40 0, clip, height=.48\textwidth]{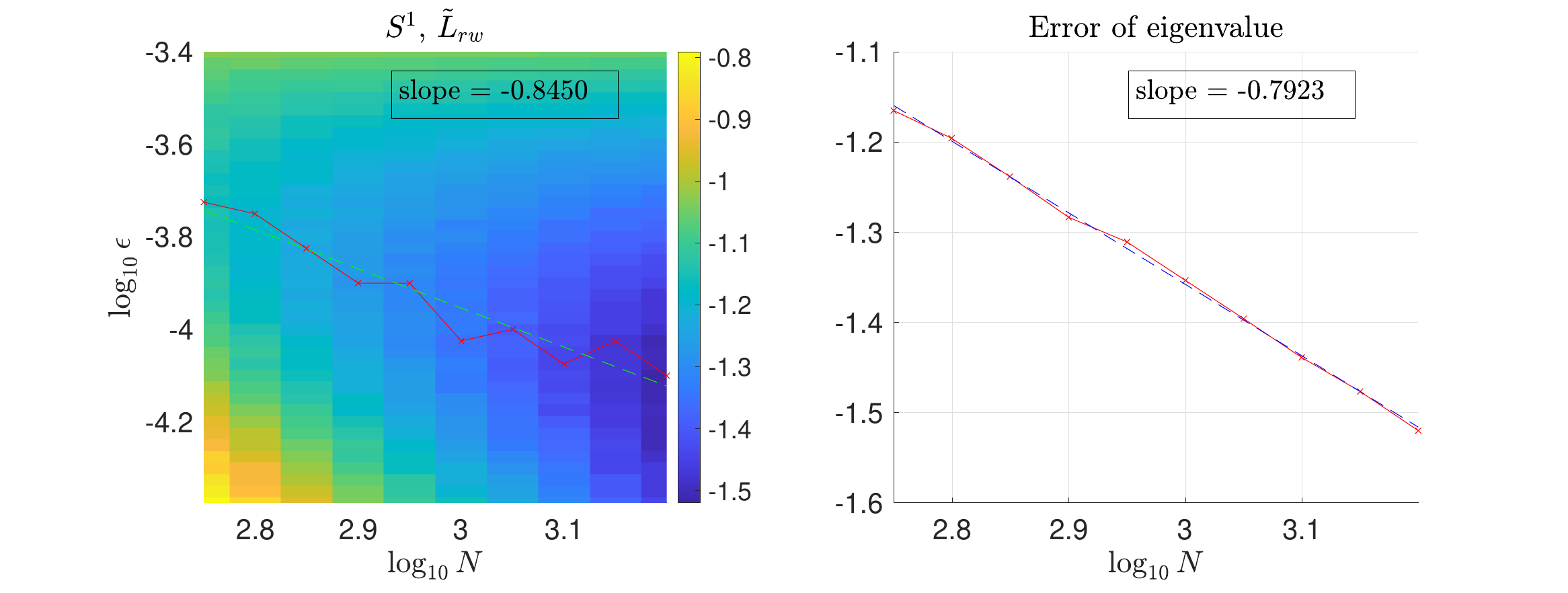} 
\caption{}
\end{subfigure}
\hspace{20pt}
\begin{subfigure}{0.5 \textwidth}
\includegraphics[trim =  40 0 40 0, clip, height=.48\linewidth]{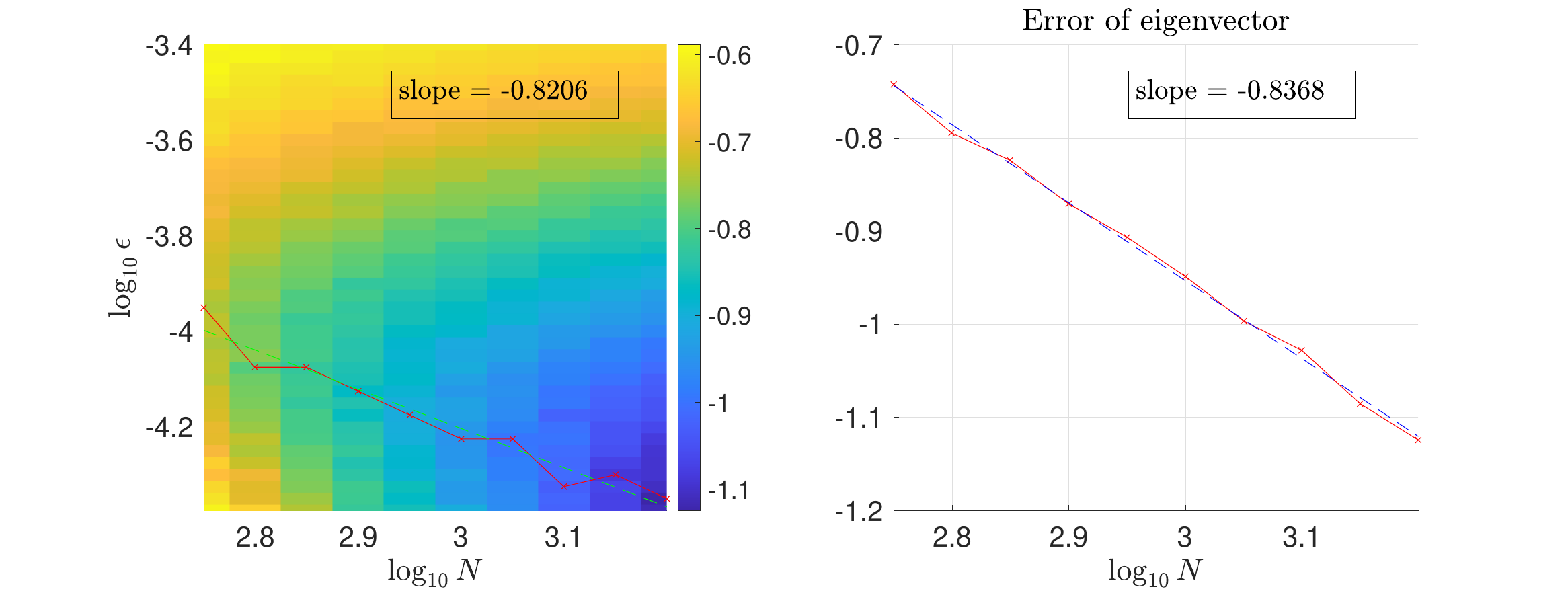} 
\caption{}
\end{subfigure}
\caption{
\scriptsize
Same plots as Fig. \ref{fig:S1-tildeLrw}
where the log error values on the (log) grid of $N$ and $\epsilon$ are without smoothing.
}
\label{fig:S1-tildeLrw-nosmooth}
\end{figure}

\section{More preliminaries}\label{app:proofs-prelim}

Throughout the paper, we use the following version of classical Bernstein inequality,
where the tail probability uses $\nu > 0$ which is an upper bound  of the variance. 
We use the sub-Gaussian near-tail,
which holds  when the tempted deviation threshold $t < \frac{3 \nu }{L}$.

\begin{lemma}[Classical Bernstein]\label{lemma:bern}
Let $\xi_j$ be i.i.d. bounded random variables, $j = 1, \cdots, N$,
$\E \xi_j = 0$.
 If $|\xi_j | \le L$ and $\E \xi_j^2 \le \nu$ for $L, \nu > 0$,
 then
\[
\Pr [ \frac{1}{N}\sum_{j=1}^N \xi_j  > t ], \,
\Pr [ \frac{1}{N}\sum_{j=1}^N \xi_j  < - t ] 
\le \exp \{-\frac{ t^2 N}{ 2 ( \nu + \frac{t L}{3} )} \},
\quad \forall t > 0.
\]
In particular, when $t L < 3 \nu$, both the tail probabilities are bounded by $ \exp \{ - \frac{1}{4} \frac{Nt^2}{\nu}\}$.
\end{lemma}

Additional proofs in Section \ref{sec:prelim}:

\begin{proof}[Proof of Theorem \ref{thm:Heat-short-time}]
Part 1): We provide a direct verification of \eqref{eq:parametrix-m} based on the parametrix construction for completeness,
which is not explicitly included in \cite{rosenberg1997laplacian}.

First note that there is $t_0$,  determined by $\calM$ s.t. when $ t < t_0$,
\[
\int_{\calM} G_t( x,y) dV(y) = \int_{\calM} G_t( y,x) dV(y) \le C_6, \quad \forall x \in \calM,
\]
 for some $C_6 > 0$ depending on $\calM$.
This is because $\int_{\calM} G_t( x,y) dV(y) $ up to an $O(t)$ truncation error equals
the integral on $B_t : = \{y \in \calM, d_{\calM}(x,y) < \delta_t := \sqrt{ (d/2+1)  t \log \frac{1}{t} }\}$.
By change to the projected coordinate $u$ in $T_x (\calM)$, the integral domain of $u$ is contained in $1.1\delta_t$-ball in $\R^d$
for small enough $\delta_t$, then
\begin{align*}
\int_{B_t} G_t( x,y) dV(y) 
& =  \frac{1}{(4\pi t)^{d/2}}  \int_{B_t  }
  e^{ - \frac{d_\calM(x,y)^2}{ 4t}}dV(y) 
   \le \frac{1}{(4\pi t)^{d/2}}  \int_{ u \in \R^d, \, \| u \|< 1.1 \delta_t  }
  e^{ - \frac{0.9 \| u\|^2}{ 4t}} (1+ O(   \delta_t^2 ))du \\
  & 
  \le \Theta(1) (1+ O( t \log{\frac{1}{t}})) = O(1).
\end{align*}

Next, as has been shown in Chapter 3 of \cite{rosenberg1997laplacian},
there exist $u_l \in C^{\infty}(\calM \times \calM)$ for $l=1, \cdots, m$, $u_0$ satisfies the needed property, 
and we define $P_m(t,x,y) = G_t(x,y) \left( \sum_{l=0}^m t^l u_l(x,y) \right) $,
$P_m \in C^\infty( (0,\infty), \calM \times \calM)$.
By Theorem 3.22 of \cite{rosenberg1997laplacian}, 
\[
\calH_t(x,y) - P_m(t, x,y) = \int_0^t  ds \int_{\calM} Q_m(t-s, x,z) P_m(s, z, y) dV(z), 
\]
where by Lemma 3.18 of \cite{rosenberg1997laplacian}, there is $C_7(t_0)$ and thus is determined by $\calM$ s.t.
\[
\sup_{x,y \in \calM} | Q_m( s, x, y) | \le C_7 s^{m - d/2}, \quad \forall  0 \le s \le t_0.
\]
As a result, for $ t < t_0$,
\begin{align*}
&| \calH_t(x,y) - P_m(t, x,y) |
 \le  \int_0^t  ds \int_{\calM} | Q_m(t-s, x,z)|  G_s(z,y) \left| \sum_{l=0}^m t^l u_l(z,y) \right| dV(z) \\
& ~~~
\le  C_7 t^{m - d/2} ( \sum_{l=0}^m \| u_l \|_\infty ) \int_0^t  ds \int_{\calM}  G_s(z,y)  dV(z) \\
& ~~~
\le C_7 t^{m - d/2} ( \sum_{l=0}^m \| u_l \|_\infty )  C_6 t = O(t^{m-d/2+1}).
\end{align*}

Part 2) is a classical result proved in several places, 
see e.g. Theorem 1.1 in \cite{grigor1997gaussian} combined with $\sup_{x \in \calM}\calH_t(x,x) \le C_5 t^{-d/2} $ for some $C_5$ depending on manifold,
which can be deduced from Part 1). 
The constant 5 in $5t$ in the exponential in \eqref{eq:H-decay} can be made any constant greater than 4, and the constant $C_3$ change accordingly.
\end{proof}

\begin{proof}[Proof of Lemma \ref{lemma:heat}]
Let $m = \lceil \frac{d}{2}+3 \rceil$, $m$ is a positive integer $m - \frac{d}{2}\ge 3$.
Since $t  \to 0$, and $\delta_t = o(1)$,
the Euclidean ball of radius $\delta_t$   contains $\delta_t$-geodesic ball
and is contained  ($1.1 \delta_t$)-geodesic ball, for small enough $t$.
Then both claims in Theorem \ref{thm:Heat-short-time} hold when $t < \epsilon_0$ for some $\epsilon_0$ depending on $\calM$,
and in 1) for $y \in B_{\delta_t}(x) \cap \calM$,
$C_2 t^{m-d/2+1} = O(t^3)$.
Here by choosing larger $m$ can make the term of higher order of $t$, 
yet $O(t^3)$ is enough for our later analysis.

Proof of \eqref{eq:H-eps-local}: 
We use the shorthand notation $\tilde{O}(t)$ to denote $O( t \log \frac{1}{t}) $. 
In Theorem \ref{thm:Heat-short-time}, $m$ is fixed, $\| u_l \|_\infty$ for $l\le m$ are finite constants depending on $\calM$,
thus
\[
 \calH_t( x,y) =  G_t(x,y) \left(  u_0(x,y) + O(t) \right) + O( t^3).
\] 
Note that $d_\calM (x,y)^2 = \| x-y\|^2 (1+ O(\| x-y\|^2))$,
and thus when $y \in B_{\delta_t}(x)$,
$d_\calM (x,y)^2 =  O(\| x-y\|^2) = O(\delta_t^2) =\tilde{O}(t)$.
By the property of $u_0$,
\[
u_0(x,y) = 1 + O( d_\calM(x,y)^2) = 1 + \tilde{O}( t  ).
\]
Meanwhile,  by mean value theorem  and that $d_\calM( x,y) \ge \|x -y\|$,
\[
e^{ - \frac{d_\calM(x,y)^2}{t}} 
= e^{ - \frac{ \| x-y\|^2 (1+ O(\| x-y\|^2)) }{t}} 
= e^{ - \frac{ \| x-y\|^2 }{t}}( 1+ O( \frac{\| x-y\|^4}{t})) ,
\]
and then
\[
G_t(x,y) = K_t(x,y)  (1 + O( \frac{\| x-y\|^4}{t} )) = K_t(x,y)  (1 + {O}( t  (\log \frac{1}{t})^2) ).
\]
Thus, for any $y \in B_{\delta_t}(x) \cap \calM$, 
\[
 \calH_t( x,y) =  K_t(x,y)  ( 1 + {O}( t  (\log \frac{1}{t})^2) )
 \left(  1 + \tilde{O}( t  ) + O(t) \right) + O( t^3),
\] 
which proves \eqref{eq:H-eps-local}, and the constants in big-$O$ are all determined by $\calM$.

Proof of \eqref{eq:H-eps-truncate} and \eqref{eq:H-global-boundedness}: 
When $y$ is outside the $\delta_t$-Euclidean ball, it is outside the $\delta_t$-geodesic ball.
Then,
by Theorem \ref{thm:Heat-short-time} 2) and the definition of $\delta_t$,
$\calH_t (x,y) \le C_3 t^{-d/2} e^{ - \frac{ \delta_t^2 }{5 t}}
\le C_3 t^{10}$, which proves \eqref{eq:H-eps-truncate}.
\eqref{eq:H-global-boundedness} directly follows from \eqref{eq:H-decay}.
\end{proof}

\section{Proofs about graph Laplacians with $W$}

\subsection{Proofs in Section \ref{sec:step0}}\label{app:proofs-step0}

\begin{proof}[Proof of \eqref{eq:bias-error-remark-indicator-h} in Remark \ref{rk:indicator-h-form-rate}]
We want to show that 
\[
\frac{1}{\epsilon} \int_{\calM} \int_{\calM} K_\epsilon(x,y) (f(x)-f(y))^2 p(x)p(y) dV(x) dV(y)
= m_2 [h] \langle f, -\Delta_{p^2} f \rangle_{p^2} + O(\epsilon).
\]
First consider when $p$ is uniform.
Denote by $B_r(x)$ the Euclidean ball in $\R^D$ centered at $x$ with radius $r$.
When $y \in B_{\sqrt{\epsilon}}(x) \cap \calM$,
$(f(x)-f(y))^2 = ( \nabla f(x)^T u)^2 + Q_{x,3}(u) + O( \| u\|^4) $, where $u \in \R^d$ is the local projected coordinate,
i.e., let $\phi_x$ be the projection onto $T_x(\calM)$, $u = \phi_x( y-x)$, 
also $\| u \| \le \| y - x\| < \sqrt{\epsilon}$.
$Q_{x,3}( \cdot)$ is a three-order polynomial where the coefficients depend on the derivatives of extrinsic coordinates of $\calM$ and $f$ at $x$. 
Then, 
\begin{align}
& \frac{1}{\epsilon} \int_{ \calM} K_\epsilon( x,y )  (f(x)-f(y))^2   dV(y)
= \int_{ \calM}  \epsilon^{-d/2} h(\frac{ \| x-y\|^2}{\epsilon}) \frac{(f(x)-f(y))^2}{\epsilon}  dV(y) 
\label{eq:EVij-indicator-h}
\\
& = \epsilon^{-d/2} \int_{ \tilde{B}}  \left(  \frac{( \nabla f(x)^T u)^2 }{\epsilon} 
+ \frac{Q_{x,3}(u)}{\epsilon} + O( \epsilon) \right) (1 + O( \epsilon ))du,
\quad 
\tilde{B}: = \phi_x ( B_{\sqrt{\epsilon}}(x) \cap \calM)
\nonumber
\end{align}
and $\tilde{B} \subset B_{\sqrt{\epsilon}}(0; \R^d)$,
where we used the volume comparison relation  $dV(y) = (1+O( \| u\|^2)) du$.
By the metric comparison, $\| y -x \| = \|u\|(1+O(\|u\|^2)) $,
 thus 
 \[
 \Vol( B_{\sqrt{\epsilon}}(0; \R^d) \backslash  \tilde{B}) \le \Vol( B_{\sqrt{\epsilon}}(0; \R^d) \backslash  B_{\sqrt{\epsilon}(1-O(\epsilon))}(0; \R^d))  = \epsilon^{d/2}O(\epsilon).
 \]
Meanwhile, the integration of odd power of $u$ vanishes on $\int_{B_{\sqrt{\epsilon}}(0; \R^d)} du$.
Thus one can verify that 
$\epsilon^{-d/2} \int_{ \tilde{B}}   \frac{( \nabla f(x)^T u)^2 }{\epsilon} du = m_2[h] |\nabla f(x)|^2 + O(\epsilon)$,
$\epsilon^{-d/2} \int_{ \tilde{B}}   \frac{Q_{x,3}(u)}{\epsilon} du =  O(\epsilon^{3/2})$,
and thus
the l.h.s. of \eqref{eq:EVij-indicator-h}
$= m_2[h] |\nabla f(x)|^2 + O(\epsilon)$.
Integrating over $\int_{\calM} dV(x)$ proves that the bias error is $O(\epsilon)$. 
When $p$ is not uniform, 
one can similarly show that 
$ \frac{1}{\epsilon} \int_{ \calM} K_\epsilon( x,y )  (f(x)-f(y))^2 p(y)   dV(y) = m_2[h] |\nabla f(x)|^2 p(x) + O(\epsilon)$
and the proof extends.
\end{proof}

\begin{proof}[Proof of Lemma \ref{lemma:form-rate-psi}]
Since $p$ is a constant, $\Delta_{p^2} = \Delta$.
Apply Theorem \ref{thm:form-rate} to when $f =  \psi_k $, and $(\psi_k \pm \psi_l)$ where $k \neq l$,
which are $K^2$ cases and are all in $C^\infty(\calM)$.
Since the set $\{ \psi_k \}_{k=1}^K$ is orthonormal in $L^2(\calM, dV)$,
\[
p^{-1} \langle \psi_k, -\Delta \psi_k \rangle_{p^2} = p \mu_k ;
\quad
p^{-1} \langle \psi_k \pm  \psi_l, -\Delta (\psi_k \pm \psi_l) \rangle_{p^2} =p(\mu_k + \mu_l), 
\quad k \neq l, 1 \le k,l \le K.
\]
Under the intersection of the $K^2$ good events which happens with the indicated high probability,  \eqref{eq:form-rate-psi} holds.
The needed threshold of $N$ is the max of the $K^2$ many ones.
These thresholds and  the constants in the big-$O$'s depend on $p$ and $\psi_k$ for $k$ up to $K$,
and $K$ is a fixed integer.
This means that these constants are determined by $\calM$, and thus are treated as absolute ones.
\end{proof}

\begin{proof}[Proof of Lemma \ref{lemma:rhoX-isometry-whp}]
First, for any $f \in C(\calM)$, when $N > N_f$ depending on $f$, w.p. $>1-2N^{-10} $,
  \begin{equation}\label{eq:rhoXf-iso}
\frac{1}{N}\|  \rho_X f \|_2^2 =  \langle f, f \rangle_p+ O_f( \sqrt{ \frac{\log N}{N }} ).
  \end{equation}
This is because, by definition,
$\frac{1}{N}\|  \rho_X f \|_2^2  = \frac{1}{N} \sum_{j=1}^N f(x_i)^2$,
which is independent sum of r.v. $Y_j := f(x_i)^2$. 
$\E Y_j = \int_{\calM} f(y)^2 p dV(y)  = \langle f, f \rangle_p$,
and 
boundedness $|Y_j| \le L_Y := \|f \|_{\infty,\calM }^2$ which is $O_f(1)$ constant.
The  variance of $Y_j$ is bounded by
$\E Y_j^2 = \int_{\calM} f(y)^4 p dV(y) := \nu_Y$,
which again is $O_f(1)$ constant. 
Since $\log N/N = o(1)$, \eqref{eq:rhoXf-iso}
follows by  the classical Bernstein.

Now consider the $K$ vectors $u_k = \frac{1}{\sqrt{p}} \rho_X \psi_k$.
Apply \eqref{eq:rhoXf-iso} to when $f = \frac{1}{\sqrt{p}} \psi_k$ and $\frac{1}{\sqrt{p}}(\psi_k \pm \psi_l)$ for $k \neq l$, 
and consider the intersection of the $K^2$ good events,
which happens w.p. $> 1- 2K^2 N^{-10}$, when $N$ exceeds the maximum thresholds of $N$ for the $K^2$ cases.
By $\langle \psi_k, \psi_l \rangle_p = p \delta_{kl}$,
and the polar formula $4 u_k^T u_l = \|u_k + u_l \|^2 - \|u_k - u_l \|^2$,
this gives \eqref{eq:uk-near-orthonormal}.
Both the $K^2$ thresholds and all the constants 
in big-O in \eqref{eq:uk-near-orthonormal} depend on $\{ \psi_k \}_{k=1}^K$.
\end{proof}

\begin{proof}[Proof of Lemma \ref{lemma:Di-concen}]
Suppose Part 1) has been shown with uniform constant in big-$O$ for each $i$,
then under the good event of Part 2), Part 2) holds automatically. 
In particular,
since \eqref{eq:concen-Di} is a property of the random r.v. $W_{ij}$ only,
where $W_{ij}$ are determined by the random points $x_i$ 
and irrelevant to the vector $u$,
the threshold of large $N$ is determined by when Part 1) holds and is uniform for all $u$.

It suffices to prove Part 1) to finish proving the lemma. 
For each $i$,
we construct an event under which the bound in \eqref{eq:concen-Di} holds for $D_i$, and then apply a union bound. 
For $i$ fixed, 
\[
\frac{1}{N}D_i = \frac{1}{N} K_\epsilon(x_i, x_i) 
+ \frac{1}{N} \sum_{j \neq i} K_\epsilon(x_i, x_j) =: \textcircled{1} + \textcircled{2}.
\]
By Assumption \ref{assump:h-C2-nonnegative}(C2), 
$K_\epsilon(x_i,x_i)  = \epsilon^{-d/2} h(0) \le \Theta( \epsilon^{-d/2})$.
and thus $\textcircled{1} = O(N^{-1} \epsilon^{-d/2})$.
Consider 
$ \textcircled{2}' := \frac{1}{N-1} \sum_{j \neq i} K_\epsilon (x_i, x_j)$,
which is an independent sum condition on $x_i$ and over the randomness of $\{ x_j \}_{j \neq i}$.
The $(N-1)$ r.v. 
\[
Y_j := K_\epsilon( x_i, x_j), \quad j \neq i,
\]
satisfies that (Lemma 8 in \cite{coifman2006diffusion}, Lemma {A.3} in \cite{cheng2020convergence}) 
\[
\E Y_j = \int_\calM K_\epsilon( x_i, y) p dV(y) = p m_0 + O(\epsilon).
\]
Boundedness: again by Assumption \ref{assump:h-C2-nonnegative}(C2),
$| Y_j |\le L_Y = \Theta( \epsilon^{-d/2}) $. 
Variance of $Y_j$ is bounded by
\begin{align*}
\E Y_j^2 & = \int_\calM K_\epsilon( x_i, y)^2 pdV(y) 
 = p \int_\calM  \epsilon^{-d} h^2 ( \frac{ \| x_i- y \|^2}{\epsilon} )  dV(y),
\end{align*}
where since $h^2(r)$ as a function on $[0,\infty)$ also satisfies Assumption \ref{assump:h-C2-nonnegative},
\[
\E Y_j^2 
= \epsilon^{-d/2} p(m_0[h^2] + O(\epsilon) )
\le \nu_Y = \Theta( \epsilon^{-d/2}).
\]
The constants in the big-$\Theta$ notation of $L_Y$ and $\nu_Y$ 
are absolute ones depending on $\calM$ and do not depend on $x_i$.
Since $\sqrt{\frac{\log N}{ N \epsilon^{d/2}}} = o(1)$, the classical Bernstein gives that 
when $N$ is sufficiently large
w.p. $> 1 - 2N^{-10}$,
\[
| \textcircled{2}'  -  \E Y_j | = O( \sqrt{ \nu_Y \frac{\log N }{N}}) 
= O(\sqrt{\frac{\log N}{N \epsilon^{d/2}}})
\quad | \text{ condition on } x_i.
\]
Under this event, $\textcircled{2}'  = O(1)$, and 
then $\textcircled{2} = (1-\frac{1}{N}) \textcircled{2}'$ gives that 
\[
\textcircled{2}  = m_0 p + O(\epsilon)
+ O(\sqrt{\frac{\log N}{N \epsilon^{d/2}}})+O(\frac{1}{N}) 
= m_0 p +O(\epsilon, \sqrt{\frac{\log N}{N \epsilon^{d/2}}}),
\]
and then
\[
\frac{1}{N}D_i = O(N^{-1} \epsilon^{-d/2}) + m_0p +O(\epsilon, \sqrt{\frac{\log N}{N \epsilon^{d/2}}}) = m_0p +O( \epsilon, \sqrt{\frac{\log N}{N \epsilon^{d/2}}}).
\]
By that $x_i$ is independent from $\{ x_j\}_{j \neq i}$, and that the bound is uniform for all location of $x_i$, 
we have that w.p. $> 1-2N^{-10}$,  the bound in \eqref{eq:concen-Di} for $i$,
and applying union bound to the $N$ events proves Part 1).
\end{proof}

\begin{proof}[Proof of Proposition \ref{prop:eigvalue-UB-rw}]
Under the condition of the current proposition,
Lemma \ref{lemma:Di-concen} applies.
For fixed $K$,  take the intersection of the good events in 
Lemma \ref{lemma:Di-concen}, \ref{lemma:rhoX-isometry-whp} and \ref{lemma:form-rate-psi},
which happens w.p. $> 1- 4K^2 N^{-10} - 2N^{-9}$ for large enough $N$.
Same as before, let $u_k = \frac{1}{\sqrt{p}} \rho_X \psi_k$,
and by \ref{lemma:rhoX-isometry-whp},
the set $\{ u_1, \cdots, u_K\}$ is linearly independent.
Let $L = \text{Span}\{ u_1, \cdots, u_k\}$, then $dim(L) = k$ for each $k \le K$.
For any $v \in L$,  $ v \neq 0$,
there are $c_j$, $1 \le j \le k $, such that 
$v = \sum_{j=1}^k c_j u_j$.
Again, by \eqref{eq:uk-near-orthonormal}, we have
$
 \frac{1}{N} \|v\|^2 = \| c \|^2 ( 1+ O( \sqrt{ \frac{\log N}{ N}}) )$,
and together with Lemma \ref{lemma:Di-concen} 2),
\begin{align}
\frac{1}{m_0} \frac{1}{N^2} v^T D v 
& =   \frac{1}{N} \|v\|^2 ( p + O( \epsilon, \sqrt{\frac{\log N}{ N \epsilon^{d/2}}})) 
 =   \| c \|^2 ( 1+ O( \sqrt{ \frac{\log N}{ N}}) )( p + O( \epsilon, \sqrt{\frac{\log N}{ N \epsilon^{d/2}}})) \nonumber \\
& =  \| c \|^2   p (1 + O( \epsilon, \sqrt{\frac{\log N}{ N \epsilon^{d/2}}})),
\label{eq:vDv}
\end{align}
and the constant in $O(\cdot)$ is uniform for all $v$. 
For $E_N(v)$, \eqref{eq:UB-ENv} still holds, and by that $K$ is fixed it  gives 
\begin{align*}
E_N( v) 
& \le \| c \|^2  \left(  p \mu_k +  O(\epsilon, \sqrt{  \frac{  \log N    }{ N \epsilon^{d/2  }}   } ) \right).
\end{align*}
Together with \eqref{eq:vDv}, we have that
\[
\frac{ 
E_N(v)}{  \frac{1}{m_0}  \frac{1}{N^2}   v^T D v }
\le 
\frac{    p \mu_k +  O(\epsilon, \sqrt{  \frac{  \log N    }{ N \epsilon^{d/2  }}   } )  }
{   p (1 + O( \epsilon, \sqrt{\frac{\log N}{ N \epsilon^{d/2}}}))}
= \mu_k +  O(\epsilon, \sqrt{  \frac{  \log N    }{ N \epsilon^{d/2  }}   } ),
\]
and the r.h.s.  upper bounds $\lambda_k(L_{rw})$ by \eqref{eq:lambdak-rw-min-max}.
\end{proof}
\subsection{Proofs in Section \ref{sec:step1}}\label{app:proofs-step1}

\begin{proof}[Proof of \eqref{eq:concen-Ds} in Lemma \ref{lemma:qs0-concen}]
Suppose $s$ is small enough such that Lemma \ref{lemma:heat} holds with $\epsilon$ being  $s$ here. 
For each $i$, we construct an event under which the bound in \eqref{eq:concen-Ds} holds for $(D_s)_i$, and then apply a union bound. 
For $i$ fixed, 
\[
(D_s)_i = \frac{1}{N} \calH_s(x_i, x_i) 
+ \frac{1}{N} \sum_{j \neq i} \calH_s(x_i, x_j) =: \textcircled{1} + \textcircled{2}.
\]
By \eqref{eq:H-global-boundedness},
$\calH_s(x_i, x_i) = O(s^{-d/2})$,
and thus $\textcircled{1} = O(N^{-1} s^{-d/2})$.
Consider 
$ \textcircled{2}' := \frac{1}{N-1} \sum_{j \neq i} \calH_s(x_i, x_j)$,
which is an independent sum condition on $x_i$ and over the randomness of $\{ x_j \}_{j \neq i}$.
The $(N-1)$ r.v. 
$Y_j := \calH_s( x_i, x_j)$, $j \neq i$,
satisfies that 
$\E Y_j = \int_\calM \calH_s( x_i, y) pdV(y) = p$,
and  boundedness: again by \eqref{eq:H-global-boundedness}, $| Y_j |\le L_Y = \Theta( s^{-d/2}) $. 
Variance of $Y_j$ is bounded by
$\E Y_j^2 = \int_\calM \calH_s( x_i, y)^2 pdV(y) = p \calH_{2s}( x_i, x_i) \le \nu_Y = \Theta( s^{-d/2})$.
The constants in the big-$\Theta$ notation of $L_Y$ and $\nu_Y$ 
are from \eqref{eq:H-global-boundedness} which only depend on $\calM$ and not on $x_i$.
We use the notation $O_{\calM}(\cdot)$ to stress this.
Since $\sqrt{\frac{\log N}{ N s^{d/2}}} = o(1)$, the classical Bernstein gives that with sufficiently large  $N$, 
w.p. $> 1 - 2N^{-10}$,
\[
| \textcircled{2}'  - p | = O( \sqrt{ \nu_Y \frac{\log N }{N}}) = O_\calM (\sqrt{\frac{\log N}{N s^{d/2}}})
\quad | \text{ condition on } x_i.
\]
The rest of the proof is the same as that of Lemma \ref{lemma:Di-concen} 1),
namely, by that $\textcircled{2} = (1-\frac{1}{N}) \textcircled{2}'$,
one can verify that 
both $\textcircled{2}$ and then $(D_s)_i$ equals $p +O_{\calM}(\sqrt{\frac{\log N}{N s^{d/2}}})$
w.p. $> 1-2N^{-10}$,
 and then  \eqref{eq:concen-Ds} follows from applying union bound to the $N$ events.
\end{proof}

\begin{proof}[Proof of Proposition \ref{prop:eigvalue-LB-crude-rw}]
The proof is by the same method as  that of Proposition \ref{prop:eigvalue-LB-crude},
and the difference is that the eigenvectors are $D$-orthogonal here and normalized differently. 
 Denote $\lambda_k(L_{rw})$ as $\lambda_k$,
and let $L_{rw} v_k = \lambda_k v_k$, normalized s.t.
\[
\frac{1}{N^2}v_k^T D v_l = \delta_{kl},  
\quad
1\le  k,l \le N.
\]
Note that this normalization of $v_k$ differs from what is used in the final eigen-convergence rate result,
Theorem \ref{thm:refined-rates-rw}, because the current proposition concerns eigenvalue only.

Because $\epsilon^{d/2+2} > c_K  \frac{\log N}{N} $,  $\epsilon^{d/2} = \Omega( \frac{\log N}{N} )$, then
the conditions needed in  Proposition \ref{prop:eigvalue-UB-rw}
are satisfied. 
Thus,
 with sufficiently large $N$, 
 there is an event $E_{UB}'$ 
 which happens w.p. $> 1- 2 N^{-9} - 4 K^2 N^{-10}$, 
 under which  $D_i > 0$ for all $i$ s.t. $L_{rw}$ is well-defined, 
 and 
 \eqref{eq:lambdak-UB-hold} holds for $\lambda_k = \lambda_k(L_{rw})$.
Because the good event $E_{UB}'$ in Proposition \ref{prop:eigvalue-UB-rw} 
assumes the good event in Lemma \ref{lemma:Di-concen}, 
then \eqref{eq:concen-uDu} also holds for all the $v_k$ and $v_k \pm v_l$,
which gives that  ($m_0 = 1$ because $h$ is Gaussian)
\begin{equation*}
\begin{split}
1 = \frac{1}{N^2} v_k^T D v_k 
 & =  \frac{1}{N} \|v_k\|^2 ( 
 p + O( \epsilon, \sqrt{\frac{\log N}{ N \epsilon^{d/2}}})),
  \quad  1 \le k  \le K,\\
2 = \frac{1}{N^2} (v_k\pm v_l)^T D (v_k \pm v_l)
 &= \frac{1}{N} \|v_k \pm v_l \|^2 ( 
 p + O( \epsilon, \sqrt{\frac{\log N}{ N \epsilon^{d/2}}}))
 \quad k\neq l, 1 \le k,l \le K,
 \end{split}
\end{equation*}
and, equivalently (because $p > 0$ is a constant)
\begin{equation}
\begin{split}
\frac{1}{N} \|v_k\|^2
& = \frac{1}{ 
p} (1 + O( \epsilon, \sqrt{\frac{\log N}{ N \epsilon^{d/2}}})),
\quad  1 \le k \le K, \\
 \frac{1}{N} \|v_k \pm v_l \|^2
& = \frac{1}{ 
p} (2 + O( \epsilon, \sqrt{\frac{\log N}{ N \epsilon^{d/2}}})),
 \quad k\neq l, 1 \le k,l \le K.
\end{split}
\end{equation}
We set $\delta$, $r$, $t$, in the same way,
and  
 let  $f_k = I_r [ v_k ]$, $f_k \in C^\infty(\calM)$.
Because the good event $E^{(0)}$ only concerns randomness of $H_{\delta \epsilon}(x_i, x_j)$,
under $E^{(0)}$ which happens w.p. $> 1- 2N^{-9}$,
\begin{equation}\label{eq:q0-delta-eps-vkvl-rw}
\begin{split}
q^{(0)}_{\delta \epsilon }( v_k) 
&  =  \frac{1}{N} \| v_k \|^2 ( p + O(\sqrt{\frac{\log N}{ N \epsilon^{d/2}}})) 
 =  
 1 + O( \epsilon, \sqrt{\frac{\log N}{ N \epsilon^{d/2}}}),
 \quad 1 \le k \le K, \\
q^{(0)}_{\delta \epsilon }( v_k \pm v_l) 
&  =  \frac{1}{N} \| v_k \pm v_l \|^2 ( p + O(\sqrt{\frac{\log N}{ N \epsilon^{d/2}}})) 
 =  
2 + O( \epsilon, \sqrt{\frac{\log N}{ N \epsilon^{d/2}}}), 
\quad k \neq l, 1 \le k,l \le K.
\end{split}
\end{equation}
Next, note that  since $(D-W)v_k = \tilde{m} \epsilon \lambda_k D v_k$, and with Gaussian $h$, $\tilde{m}=1$,
and $v_k$ are $D$-orthogonal,
\begin{equation}\label{eq:form-vk-vl-1}
\begin{split}
 & ~~~
 \frac{v_k^T(D-W) v_k}{N^2}  
 =  \epsilon 
 \lambda_k \frac{1}{N^2} v_k^T D v_k 
 = \epsilon 
  \lambda_k,
 \quad 1  \le k \le K, \\
 &  \frac{ (v_k \pm v_l)^T(D-W) (v_k \pm v_l)}{N^2} 
 = \epsilon 
 ( \lambda_k + \lambda_l), 
 \quad k \neq l, 1 \le k,l \le K.
 \end{split}
\end{equation}
Then, \eqref{eq:q2-alphaeps-UB} in Lemma \ref{lemma:qs2-UB} where $\alpha = \delta$ gives that
\[
\begin{split}
q^{(2)}_{\delta \epsilon }( v_k) 
& =
O( \delta^{-d/2} ) \epsilon 
\lambda_k 
 +  O(\epsilon^{3}) , 
 \quad 1 \le k \le K, \\
q^{(2)}_{\delta \epsilon }( v_k \pm v_l) 
&  
= O( \delta^{-d/2})   \epsilon 
( \lambda_k + \lambda_l) 
+ 2  O(\epsilon^{3}) ,
 \quad k \neq l, \, 1 \le k,l \le K,
\end{split}
\]
then same as in \eqref{eq:q2-delta-eps-vkvl}, they are both $O(\epsilon)$.
Together with \eqref{eq:q0-delta-eps-vkvl-rw}, this gives that 
\begin{equation}
\begin{split}
\langle f_k, f_k \rangle 
&= 
 1 + O( \epsilon, \sqrt{\frac{\log N}{ N \epsilon^{d/2}}})
+  O( \epsilon), \quad 1 \le k \le K, \\
\langle f_k, f_l \rangle 
& = \frac{1}{4}( q_{\delta \epsilon}( v_k + v_l)  - q_{\delta \epsilon}( v_k - v_l) )
= O( \epsilon, \sqrt{\frac{\log N}{ N \epsilon^{d/2}}}) 
+ O(\epsilon),
\quad k \neq l, \, 1\le k, l \le K.
\end{split}
\end{equation}
Then due to that $O( \epsilon, \sqrt{\frac{\log N}{ N \epsilon^{d/2}}}) = o(1)$, 
we have linear independence of $\{ f_j\}_{j=1}^K$ with large enough $N$.

Again, we let $L_k = \text{Span}\{ f_1, \cdots, f_k\}$,
and have \eqref{eq:eigen-relation-2}.
For any $f \in L_k$, 
$ f = \sum_{j=1}^k c_j f_j $,
$f = I_r [v]$, $v := \sum_{j = 1}^k c_j v_j$,
 \[
 \frac{1}{N^2}v^T D v 
  = \sum_{j=1}^k c_j^2 \frac{1}{N^2} v_j^T D v_j = \| c \|^2,
\]
and, by that Lemma \ref{lemma:Di-concen} 2) holds, \eqref{eq:concen-uDu} applies to $v$ to give
$\frac{1}{N^2} v^T D v 
=  \frac{1}{N} \|v\|^2 ( 
p + O( \epsilon, \sqrt{\frac{\log N}{ N \epsilon^{d/2}}}))$,
thus 
\begin{equation}\label{eq:v2divN-rw}
\frac{1}{N} \|v\|^2  
= 
\frac{ \| c \|^2 }{ 
p}(1 + O( \epsilon, \sqrt{\frac{\log N}{ N \epsilon^{d/2}}})).
\end{equation}
Meanwhile, by \eqref{eq:form-vk-vl-1},
\begin{equation}\label{eq:form-bound-v-rw}
 \frac{ v^T (D-W) v}{ N^2} 
  =   \sum_{j=1}^k c_j^2 \frac{v_j^T (D-W) v_j }{N^2} 
 = \sum_{j=1}^k c_j^2  \epsilon 
 \lambda_j
\le  \epsilon 
 \lambda_k  \| c \|^2.
\end{equation}
 With the good  event $E^{(1)}$ same as before (Lemma \ref{lemma:qs0-concen} at $s= \epsilon$),
under $E^{(0)} \cap E^{(1)}$,   and the $O_\calM (\cdot)$ notation means that the constant depends on $\calM$ only and not on $K$,
\begin{equation} \label{eq:q0-v-rw-1}
 q^{(0)}_{\epsilon}( v ) =  \frac{1}{N} \|v\|^2 ( p + O_\calM(\sqrt{\frac{\log N}{ N \epsilon^{d/2}}})),
 \quad
  q^{(0)}_{\delta \epsilon}( v ) =  \frac{1}{N} \|v\|^2 ( p + O_\calM(\sqrt{ \delta^{-d/2}\frac{\log N}{ N \epsilon^{d/2}}})),
\end{equation}
and then, again,
\begin{align*}
 q^{(0)}_{\delta \epsilon}( v) - q^{(0)}_{\epsilon}(v) 
 & = \frac{1}{N} \|v\|^2  O_\calM( \delta^{-d/4} \sqrt{ \frac{\log N}{N  \epsilon^{d/2}} })
  =  \frac{ \| c \|^2 }{ 
  p}(1 + O( \epsilon, \sqrt{\frac{\log N}{ N \epsilon^{d/2}}})) O_\calM(  \delta^{-d/4} \sqrt{ \frac{\log N}{N  \epsilon^{d/2}} })  \\
&  = \| c \|^2  O_\calM(  \delta^{-d/4} \sqrt{ \frac{\log N}{N  \epsilon^{d/2}} }),
\end{align*}
where we used \eqref{eq:v2divN-rw} to substitute the $ \frac{1}{N} \|v\|^2$ term 
after the leading $ \frac{1}{N} \|v\|^2 p$ term is canceled in the subtraction. 
The UB of $q^{(2)}_{\epsilon}(v)$ is similar as before, 
namely, by \eqref{eq:q2-eps-UB} in Lemma \ref{lemma:qs2-UB}, inserting \eqref{eq:form-bound-v-rw},
and with the shorthand that $\tilde{O}(\epsilon) $ stands for $ {O}( \epsilon  (\log \frac{1}{ \epsilon})^2) $,
\[
q^{(2)}_{\epsilon}(v) = \frac{ v^T (D-W) v}{ N^2}  (1 + \tilde{O}(\epsilon)) + \| c \|^2 O( \epsilon^{3})
\le \epsilon \| c \|^2 (  \lambda_k 
 (1 + \tilde{O}(\epsilon)) + O( \epsilon^{2}) ).
\]
Thus we have that
\begin{align}
\langle f ,  f  \rangle  -  \langle f ,  Q_t f  \rangle
& \le (  q^{(0)}_{\delta \epsilon}( v) - q^{(0)}_{\epsilon}(v) ) + q^{(2)}_{\epsilon}(v)  \nonumber  \\
& \le  \epsilon \| c \|^2 \left( 
 \lambda_k 
 (1 + \tilde{O}(\epsilon)) + O( \epsilon^{2})  
 +  \delta^{-d/4}  O_\calM(  \frac{1}{\epsilon}\sqrt{ \frac{\log N}{N  \epsilon^{d/2}} })
 \right)  \nonumber  \\
 &=  \epsilon \| c \|^2 \left( 
 \lambda_k 
  +  \tilde{O}(\epsilon) 
 +   \delta^{-d/4}  O_\calM( \frac{1}{\epsilon}\sqrt{ \frac{\log N}{N  \epsilon^{d/2}} })
 \right).
\quad \text{(by $\lambda_k \le 1.1 \mu_K$)}
 \label{eq:numerator-UB-rw}
\end{align}

To lower bound $\langle f, f \rangle$,
again  by \eqref{eq:q2-alphaeps-UB} in Lemma \ref{lemma:qs2-UB}, inserting \eqref{eq:form-bound-v-rw},
\[
0 \le q^{(2)}_{\delta \epsilon}( v)
  \le  \Theta( \delta^{-d/2} )  \frac{ v^T (D-W) v}{ N^2}  + \| c \|^2 O(\epsilon^{3})
\le  \epsilon  \| c \|^2   \left( \lambda_k 
 \Theta(\delta^{-d/2}) + O(\epsilon^{2}) \right),
\]
and then since $ \lambda_k 
 \Theta(\delta^{-d/2}) + O(\epsilon^{2}) = O(1)$,
we again have that 
$q^{(2)}_{\delta \epsilon}( v)  =  \| c \|^2  O( \epsilon )$.
We have derived formula of $q_{\delta\epsilon}^{(0)}(v)$ in \eqref{eq:q0-v-rw-1} under $E^{(0)} \cap E^{(1)}$,   
and inserting \eqref{eq:v2divN-rw}, 
\begin{equation} 
  q^{(0)}_{\delta \epsilon}( v ) =  \frac{1}{N} \|v\|^2 ( p + O(\sqrt{ \frac{\log N}{ N \epsilon^{d/2}}}))
  =  \| c \|^2  (1 + O( \epsilon, \sqrt{\frac{\log N}{ N \epsilon^{d/2}}})).
\end{equation}
Thus,
\[
\langle f, f \rangle 
= q^{(0)}_{\delta \epsilon}( v) - q^{(2)}_{\delta \epsilon}( v) 
= \| c \|^2 \left ( 1 + O( \epsilon, \sqrt{ \frac{\log N}{N  \epsilon^{d/2}} })  - O( \epsilon ) \right) 
\ge \| c \|^2 \left ( 1 - O( \epsilon,  \sqrt{ \frac{\log N}{N  \epsilon^{d/2}} })  \right).
\]
Together with \eqref{eq:numerator-UB-rw}, this gives
\[
\frac{\langle f ,  f  \rangle  -  \langle f ,  Q_t f  \rangle}{ \langle f, f \rangle}
\le 
\frac{ 
  \epsilon  \left( 
 \lambda_k 
  +  \tilde{O}(\epsilon) 
 +   \delta^{-d/4}  O_\calM(  \frac{1}{\epsilon}\sqrt{ \frac{\log N}{N  \epsilon^{d/2}} })
 \right)
}
 {  1 - O( \epsilon,  \sqrt{ \frac{\log N}{N  \epsilon^{d/2}} }) 
 }
 \le 
 \epsilon  \left( 
 \lambda_k    +   \tilde{O}(\epsilon) +   \frac{C}{\epsilon}\sqrt{ \frac{\log N}{N  \epsilon^{d/2}} }
 \right),
\]
where the notation of $C$ is defined in the same way as in the proof of Proposition \ref{prop:eigvalue-LB-crude}.
The rest of the proof is the same,
and the intersection of all the needed good events $E^{(0)}$, $E^{(1)}$, and $E_{UB}'$,
which happens w.p.$> 1- 2N^{-9} - 4K^2 N^{-10} - 4N^{-9}$. 
\end{proof}

\subsection{Proofs in Section \ref{sec:step23}}\label{app:proofs-step23}

\begin{proof}[Proof of Theorem \ref{thm:refined-rates-rw}]
With sufficiently large $N$, 
we restrict to the intersection of 
the good events in Proposition \ref{prop:eigvalue-LB-crude-rw} 
and the $K= k_{max}+1$ good events of applying Theorem \ref{thm:pointwise-rate-C2h} 1) to $\{ \psi_k \}_{k=1}^K$, 
which happens w.p.$> 1- 4K^2 N^{-10} -( 6+4K)N^{-9}$.
The good event in  Proposition \ref{prop:eigvalue-LB-crude-rw} 
 is contained in the good event $E_{UB}'$ of Proposition \ref{prop:eigvalue-UB-rw} of the eigenvalue UB,
which is again contained in the good event of Lemma \ref{lemma:Di-concen}.
As a result, $D_i > 0$ for all $i$,  and thus $L_{rw}$ is well-defined,
and  \eqref{eq:concen-uDu} holds.

Applying \eqref{eq:concen-uDu} to $u = v_k$,  
and because $\| v_k\|_{D/N}^2 = p$,
we have that ($m_0=1$ due to that $h$ is Gaussian)
\begin{equation}\label{eq:vk-D/N-2norm}
p=   \|  v_k \|_{\frac{D}{N}}^2 
 = p  \| v_k \|_2^2 ( 1  + O( \epsilon, \sqrt{\frac{\log N}{ N \epsilon^{d/2}}})),
  \quad 1 \le k \le K.
 \end{equation}
 This verifies that 
$\|  v_k\|_2^2 = 1 + O( \epsilon, \sqrt{\frac{\log N}{ N \epsilon^{d/2}}}) = 1+o(1)$,
for $1 \le k\le K$.

Because  the good event  $E_{UB}'$  is under that in Lemma \ref{lemma:rhoX-isometry-whp},
 $ \| \phi_k\|_2^2  = 1 + O( \sqrt{ \frac{\log N}{ N} })$, $1 \le k \le K$,
and then, applying \eqref{eq:concen-uDu} to $u=\phi_k$, 
 \begin{equation}\label{eq:phik-D/N-2norm}
  \|  \phi_k \|_{\frac{D}{N}}^2 
 = p  \| \phi_k \|^2 ( 1  + O( \epsilon, \sqrt{\frac{\log N}{ N \epsilon^{d/2}}}))
 = p ( 1  + O( \epsilon, \sqrt{\frac{\log N}{ N \epsilon^{d/2}}})),
 \quad 1 \le k \le K.
 \end{equation}
 \vspace{5pt}

\underline{ Step 2. for $L_{rw}$}:
 We follow a similar approach as in Proposition \ref{prop:step2}.
When $k=1$, $\lambda_1 =0$, and  $v_1$ is always the constant vector, thus the discrepancy is zero.
Consider $2 \le k \le K$, by  Theorem \ref{thm:pointwise-rate-C2h} 1),
and that $\| u\|_2 \le \sqrt{N} \| u\|_\infty $ for any $u \in \R^N$,
\begin{equation}\label{eq:pontwise-rate-2norm-bound-rw}
\| L_{rw} \phi_k - \mu_k \phi_k \|_2  =   O(\epsilon, \sqrt{\frac{\log N}{N \epsilon^{d/2+1}}} ),
\quad 2 \le k \le K,
\end{equation}
and then by \eqref{eq:concen-uDu} which holds uniformly for all $u \in \R^N$, 
\[
\| L_{rw} \phi_k - \mu_k \phi_k \|_{\frac{D}{N}}
= \| L_{rw} \phi_k - \mu_k \phi_k \|_2 \sqrt{p} (1   + O( \epsilon, \sqrt{\frac{\log N}{ N \epsilon^{d/2}}}))
= O( \| L_{rw} \phi_k - \mu_k \phi_k \|_2).
\]
Thus, there is $\text{Err}_{pt} > 0$, s.t.
\begin{equation}\label{eq:pontwise-rate-2norm-bound-rw-D/N}
\| L_{rw} \phi_k - \mu_k \phi_k \|_{\frac{D}{N}}  \le  \text{Err}_{pt},
\quad 2 \le k \le K,
\quad \text{Err}_{pt} = O(\epsilon, \sqrt{\frac{\log N}{N \epsilon^{d/2+1}}} ).
\end{equation}
The constant in big-$O$ depends on first $K$ eigenfunctions, and is an absolute one because $K$ is fixed. 
Next, same as in the proof of Proposition \ref{prop:step2}, 
under the good event of Proposition \ref{prop:eigvalue-LB-crude-rw} and by the definition of $\gamma_K$ as the maximum (half) eigen-gap among $\{\mu_k \}_{1 \le k \le K}$,
\eqref{eq:eigen-stay-away} holds for $\lambda_k$.

Let $S_k= \text{Span}\{ (\frac{D}{N})^{1/2} v_k \}$, $S_k$ is a 1-dimensional subspace in $\R^N$.
Because $v_j$'s are $D$-orthogonal, 
$S_k^{\perp} = \text{Span}\{  (\frac{D}{N})^{1/2} v_j, \, j\neq k, 1 \le j \le N \}$. Note that
\begin{equation}\label{eq:projected-eqn-1}
P_{S_k^{\perp}} \left( (\frac{D}{N})^{1/2}  \mu_k \phi_k  \right)  
= (\frac{D}{N})^{1/2} \sum_{j \neq k, j=1}^N \frac{   v_j^T (\frac{D}{N})\phi_k }{ \|  v_j\|_{\frac{D}{N}}^2} \mu_k  v_j,
\end{equation}
and  because 
\begin{equation}\label{eq:rw-D-adjoint-relation}
L_{rw}^T D v_j = \frac{1}{\epsilon}(I - WD^{-1}) D v_j = \frac{1}{\epsilon}(D - W)  v_j = D \lambda_j v_j,
\end{equation}
\begin{align}
P_{S_k^{\perp}} \left( (\frac{D}{N})^{1/2}  L_{rw} \phi_k  \right)  
& =(\frac{D}{N})^{1/2}  \sum_{j\neq k, j=1}^N
\frac{  v_j^T(\frac{D}{N}) L_{rw }\phi_k }{ \| v_j\|_{\frac{D}{N}}^2}  v_j
= (\frac{D}{N})^{1/2}  \sum_{j\neq k, j=1}^N
\frac{  \frac{1}{N}( L_{rw }^T D v_j )^T  \phi_k }{ \| v_j\|_{\frac{D}{N}}^2}  v_j \nonumber \\
& = 
(\frac{D}{N})^{1/2}  \sum_{j\neq k, j=1}^N
\frac{  \frac{1}{N}( D v_j )^T  \phi_k }{ \| v_j\|_{\frac{D}{N}}^2}  \lambda_j  v_j .
\label{eq:projected-eqn-2}
\end{align}
Subtracting \eqref{eq:projected-eqn-1} and \eqref{eq:projected-eqn-2} gives
\[
P_{S_k^{\perp}} \left( (\frac{D}{N})^{1/2}  (  L_{rw} \phi_k -  \mu_k \phi_k ) \right) 
= 
 \sum_{j\neq k, j=1}^N
  ( \lambda_j  - \mu_k)
\frac{   v_j^T  \frac{D}{N} \phi_k }{ \| v_j\|_{\frac{D}{N}}^2}  (\frac{D}{N})^{1/2} v_j,
\]
and by that $v_j$ are $D$-orthogonal, and \eqref{eq:eigen-stay-away},
\[
\| P_{S_k^{\perp}} \left( (\frac{D}{N})^{1/2}  (  L_{rw} \phi_k -  \mu_k \phi_k ) \right) \|_2^2
=  
 \sum_{j\neq k, j=1}^N
  | \lambda_j  - \mu_k |^2
\frac{  | v_j^T  \frac{D}{N} \phi_k|^2 }{ \| v_j\|_{\frac{D}{N}}^2}  
\ge
\gamma_K^2
\sum_{j\neq k, j=1}^N
\frac{  | v_j^T  \frac{D}{N} \phi_k|^2 }{ \| v_j\|_{\frac{D}{N}}^2}.  
\]
The square-root of the l.h.s.
\[
\| P_{S_k^{\perp}} \left( (\frac{D}{N})^{1/2}  (  L_{rw} \phi_k -  \mu_k \phi_k ) \right) \|_2
\le
\|  (\frac{D}{N})^{1/2}  (  L_{rw} \phi_k -  \mu_k \phi_k ) \|_2 
= 
\|      L_{rw} \phi_k -  \mu_k \phi_k  \|_{\frac{D}{N}}
\le  \text{Err}_{pt},
\]
and the last inequality is by \eqref{eq:pontwise-rate-2norm-bound-rw-D/N}.
This gives that
\[
\left(  \sum_{j\neq k, j=1}^N
\frac{  | v_j^T  \frac{D}{N} \phi_k|^2 }{ \| v_j\|_{\frac{D}{N}}^2} \right)^{1/2}
\le \frac{ \text{Err}_{pt} }{\gamma_K}.
\]
Meanwhile,
$P_{S_k^{\perp}} \left( (\frac{D}{N})^{1/2} \phi_k  \right) = 
  \sum_{j\neq k, j=1}^N
\frac{  v_j^T(\frac{D}{N}) \phi_k }{ \| v_j\|_{\frac{D}{N}}^2} (\frac{D}{N})^{1/2} v_j$,
and by $D$-orthogonality of $v_j$ again,
$\sum_{j\neq k, j=1}^N
\frac{  | v_j^T  \frac{D}{N} \phi_k|^2 }{ \| v_j\|_{\frac{D}{N}}^2} 
= \| P_{S_k^{\perp}} \left( (\frac{D}{N})^{1/2} \phi_k  \right) \|_2^2$.
Thus,
\begin{equation}\label{eq:project-D/N-phik-small}
\| P_{S_k^{\perp}} \left( (\frac{D}{N})^{1/2} \phi_k  \right) \|_2
= \left(   \sum_{j\neq k, j=1}^N
\frac{  | v_j^T  \frac{D}{N} \phi_k|^2 }{ \| v_j\|_{\frac{D}{N}}^2}  \right)^{1/2}
\le  \frac{ \text{Err}_{pt} }{\gamma_K}
= O(\epsilon, \sqrt{\frac{\log N}{N \epsilon^{d/2+1}}} ).
\end{equation}
Finally,  define
\[
\beta_k := \frac{   v_k^T(\frac{D}{N})\phi_k }{ \| v_k\|_{\frac{D}{N}}^2},
\quad
\beta_k    (\frac{D}{N})^{1/2} v_k = P_{S_k} (\frac{D}{N})^{1/2}\phi_k,
\]
\[
P_{S_k^{\perp}} \left( (\frac{D}{N})^{1/2} \phi_k  \right)
= (\frac{D}{N})^{1/2}\phi_k - P_{S_k} (\frac{D}{N})^{1/2}\phi_k
= (\frac{D}{N})^{1/2} \left( \phi_k - \beta_k
v_k \right),
\]
and then, together with \eqref{eq:project-D/N-phik-small},
\[
\| \phi_k -  \beta_k  v_k \|_{\frac{D}{N}}
=  \| P_{S_k^{\perp}} \left( (\frac{D}{N})^{1/2} \phi_k  \right) \|_2
= O(\epsilon, \sqrt{\frac{\log N}{N \epsilon^{d/2+1}}} ).
\]
Applying \eqref{eq:concen-uDu} to $u= \phi_k -  \beta_k  v_k $, 
$ \| \phi_k -  \beta_k  v_k  \|_2  = 
(  \frac{1}{p}( 1  + O( \epsilon, \sqrt{\frac{\log N}{ N \epsilon^{d/2}}})) )^{1/2}  \|  \phi_k -  \beta_k  v_k   \|_{\frac{D}{N}}
 = O( \|  \phi_k -  \beta_k  v_k   \|_{\frac{D}{N}})$,
 and we have shown that 
\[
\| \phi_k -  \beta_k  v_k  \|_2 
= O( \|  \phi_k -  \beta_k  v_k   \|_{\frac{D}{N}}) = O(\epsilon, \sqrt{\frac{\log N}{N \epsilon^{d/2+1}}} ).
\]
To finish Step 2, it remains to show that $|\beta_k| = 1 + o(1)$, and then we define $\alpha_k = \frac{1}{\beta_k}$.
By definition of $\beta_k$,
\begin{align*}
\|  \phi_k \|_{\frac{D}{N}}^2 
& =  \|  (\frac{D}{N})^{1/2} \phi_k \|_2^2 =
\| P_{S_k^{\perp}} \left( (\frac{D}{N})^{1/2} \phi_k  \right)  \|_2^2
+ \| \beta_k  (\frac{D}{N})^{1/2} v_k  \|_2^2 
 =  
\| P_{S_k^{\perp}} \left( (\frac{D}{N})^{1/2} \phi_k  \right)  \|_2^2
+ \beta_k^2 \|  v_k \|_{\frac{D}{N}}^2,
\end{align*}
by that $\|  v_k \|_{\frac{D}{N}}^2 = p$, and \eqref{eq:phik-D/N-2norm}, and \eqref{eq:project-D/N-phik-small},
this gives 
$
  p ( 1  + o(1)
  )
  = o(1)   + \beta_k^2 p$,
and thus  $\beta_k^2 = 1+o(1)$.
 \vspace{5pt}

\underline{Step 3. of ${L}_{rw}$}:
For $2 \le k \le k_{max}$,  by the relation \eqref{eq:rw-D-adjoint-relation},
\[
v_k^T D( L_{rw} \phi_k - \mu_k \phi_k) 
= (L_{rw}^T D v_k)^T \phi_k -  \mu_k  v_k^T D\phi_k
= (\lambda_k - \mu_k)  v_k^T D \phi_k,
\]
and we have shown that
\[
v_k = \alpha_k \phi_k + \varepsilon_k, 
\quad
\alpha_k = 1+o(1),
\quad
  \|  \varepsilon_k  \|_{\frac{D}{N}}= O(\epsilon, \sqrt{\frac{\log N}{N \epsilon^{d/2+1}}} ).
\]
Similar as in the proof of Proposition \ref{prop:step3},
\begin{align*}
& | \lambda_k - \mu_k|  |v_k^T \frac{D}{N} \phi_k|
 = 
|v_k^T \frac{D}{N}( L_{rw} \phi_k - \mu_k \phi_k) | 
=| (\alpha_k \phi_k + \varepsilon_k)^T \frac{D}{N}( L_{rw} \phi_k - \mu_k \phi_k) |
\\
& ~~~
\le |\alpha_k| | \phi_k ^T \frac{D}{N} L_{rw} \phi_k - \mu_k \| \phi_k \|^2_{\frac{D}{N}}
| +  | \varepsilon_k^T \frac{D}{N}( L_{rw} \phi_k - \mu_k \phi_k)|
=: \textcircled{1} + \textcircled{2}.
\end{align*}
By \eqref{eq:phik-D/N-2norm},
$\|  \phi_k \|_{\frac{D}{N}}^2 
 = p ( 1  + O( \epsilon, \sqrt{\frac{\log N}{ N \epsilon^{d/2}}}))$,
 and meanwhile, $\phi_k ^T \frac{D}{N} L_{rw} \phi_k = \frac{1}{p} E_N(\rho_X \psi_k)= p \mu_k + O(\epsilon ,  \sqrt{  \frac{  \log N    }{ N \epsilon^{d/2  }}   } )$
by \eqref{eq:form-rate-psi}.
Thus  $\textcircled{1} = O(  | \phi_k ^T \frac{D}{N} L_{rw} \phi_k - \mu_k \| \phi_k \|^2_{\frac{D}{N}}|) = O(\epsilon ,  \sqrt{  \frac{  \log N    }{ N \epsilon^{d/2  }}   } )$.
By \eqref{eq:pontwise-rate-2norm-bound-rw-D/N} and the bound of $\varepsilon_k$,
$|\textcircled{2}| \le \| \varepsilon_k\|_{\frac{D}{N}} \|  L_{rw} \phi_k - \mu_k \phi_k \|_{\frac{D}{N}} =O(\text{Err}_{pt}^2)$ which is $O(\epsilon)$ as shown in the proof of Proposition \ref{prop:step3}.
Finally, by the definition of $\beta_k$, and that $\|  v_k \|_{\frac{D}{N}}^2 = p$,
 \[
| \lambda_k - \mu_k|  |\beta_k|
\le \frac{ | \textcircled{1}| + |\textcircled{2} | }{\| v_k\|_{\frac{D}{N}}^2 }
= \frac{  O(\epsilon ,  \sqrt{  \frac{  \log N    }{ N \epsilon^{d/2  }}   } ) + O(\epsilon)    }{ {p}  } 
= O(\epsilon ,  \sqrt{  \frac{  \log N    }{ N \epsilon^{d/2  }}   } ).
\]
Since $|\beta_k| = 1+o(1)$, this proves the bound of $| \lambda_k - \mu_k|$, and the argument for all $k \le k_{max}$. 
\end{proof}
\section{Proofs about the density-corrected graph Laplacian with $\tilde{W}$}
\label{app:proofs-density-corrected}

\subsection{Proofs of the point-wise convergence of $\tilde{L}_{rw}$}

\begin{proof}[Proof of Lemma \ref{lemma:Di-concen-eps2}]

Part 1): 
By that $\frac{1}{N} D_i  = \frac{1}{N}( Y_i + \sum_{j \neq i}^N Y_j )$,
$Y_j : = K_\epsilon(x_i, x_j)$. 
For $j\neq i$, $Y_j$ has expectation (Lemma 8 in \cite{coifman2006diffusion}, Lemma {A.3} in \cite{cheng2020convergence}) 
\[
\int_{\calM} K_\epsilon(x_i, y) p(y) dV(y) 
= m_0 p(x_i) + \frac{m_2}{2} \epsilon( \omega p (x_i) + \Delta p(x_i)) + O_p(\epsilon^2),
\]
where $\omega \in C^{\infty}(\calM)$ is determined by manifold extrinsic coordinates;
Meanwhile,
$K_\epsilon(x_i,x_i) = \epsilon^{-d/2} h(0) = O(\epsilon^{-d/2})$;
In the independent sum $\frac{1}{N-1} \sum_{j \neq i} Y_j$,
$|Y_j|$ is bounded by $\Theta( \epsilon^{-d/2})$ and has variance bounded by $\Theta(\epsilon^{-d/2})$. 
The rest of the proof is the same as in proving Lemma \ref{lemma:Di-concen} 1).

Part 2):
By part 1), 
under a good event $E_{1}$, which happens w.p. $> 1- 2 N^{-9}$, \eqref{eq:degree-D-concen-eps2}  holds. 
Because $p(x) \ge p_{min} > 0$ for any $x \in \calM$, we then have
\begin{equation}\label{eq:degree-D-concen-rel}
\frac{1}{N} D_i = m_0 p(x_i) ( 1+ \varepsilon^{(D)}_i),
\quad
\sup_{1 \le i \le N} |\varepsilon^{(D)}_i| = O(\epsilon,  \sqrt{ \frac{\log N}{N \epsilon^{d/2}} }).
\end{equation}
Since $O(\epsilon,  \sqrt{ \frac{\log N}{N \epsilon^{d/2}} }) = o(1)$, with large enough $N$ and under $E_1$, $D_i > 0$, 
then $\tilde{W}$ is well-defined.
Furtherly, by \eqref{eq:degree-D-concen-rel},
\begin{align*}
& \frac{1}{N}\sum_{j=1}^N W_{i j} \frac{ 1}{  \frac{1}{N}D_j} 
 = \frac{1}{N}\sum_{j=1}^N  \frac{ W_{i j}}{  m_0 p(x_j) ( 1+ \varepsilon^{(D)}_j)}  \\
& ~~~ 
= \left( \frac{1}{ m_0 } \frac{1}{N}\sum_{j=1}^N  W_{i j} \frac{ 1}{ p(x_j)} \right) 
 \left( 1+  O(\epsilon,  \sqrt{ \frac{\log N}{N \epsilon^{d/2}} }) \right).
 \quad
 \text{(by that $p>0$, $W_{ij} \ge 0$)}
\end{align*}
Consider the r.v. $Y_j = K_\epsilon(x_{i}, x_j) p^{-1}(x_j)$ (condition on $x_{i}$), for $j \neq i$,
\[
\E Y_j = \int_{\calM} K_\epsilon(x_{i}, y) p^{-1}(y) p(y) dV(y)
= \int_{\calM} K_\epsilon(x_{i}, y) dV(y)
= m_0 + O(\epsilon),
\]
$Y_j$ is bounded by $\Theta(\epsilon^{-d/2})$ and so is its variance, where the constants in big-$\Theta$ depend on $p$.
Then, similar as in proving \eqref{eq:degree-D-concen-eps2}, we have a good event $E_2$ which happens w.p. $> 1 - 2 N^{-9}$,
under which

\begin{equation}\label{eq:W-p-inverse-concen}
 \frac{1}{ m_0 } \frac{1}{N}\sum_{j=1}^N  W_{i j} \frac{ 1}{ p(x_j)} 
= 1 +  O(\epsilon,  \sqrt{ \frac{\log N}{N \epsilon^{d/2}} }),
\quad 
1 \le i \le N,
\end{equation}
and the constant in big-$O$ depends on $p$, the function $h$, and is uniform for all $x_{i}$.
Then under $E_1 \cap E_2$, 
\begin{equation*}
\sum_{j=1}^N W_{i j} \frac{ 1}{  D_j} 
=\left( 1 +  O(\epsilon,  \sqrt{ \frac{\log N}{N \epsilon^{d/2}} }) \right)
 \left( 1+  O(\epsilon,  \sqrt{ \frac{\log N}{N \epsilon^{d/2}} }) \right) 
 =  1+  O(\epsilon,  \sqrt{ \frac{\log N}{N \epsilon^{d/2}} }),
\end{equation*}
which proves \eqref{eq:denominator}.
Meanwhile, combining 
\eqref{eq:denominator} and \eqref{eq:degree-D-concen-rel},
\begin{equation}\label{eq:NtildeDi}
N \tilde{D}_i = \frac{N}{D_i }\sum_{j=1}^N \frac{W_{ij}}{D_j} 
=  \frac{1}{ m_0 p(x_i) ( 1+ \varepsilon^{(D)}_i) }( 1 +  O(\epsilon,  \sqrt{ \frac{\log N}{N \epsilon^{d/2}} }))
=  \frac{1}{ m_0 p(x_i)}( 1 +  O(\epsilon,  \sqrt{ \frac{\log N}{N \epsilon^{d/2}} })),
\end{equation}
and thus under $E_1 \cap E_2$,  with large $N$, $\tilde{D}_i  > 0$ and $\tilde{L}_{rw}$ is well-defined. 
\end{proof}

\subsection{Proofs of the Dirichlet form convergence}

\begin{proof}[Proof of Lemma \ref{lemma:tildeENu-V-stat}]
As has been shown in the proof of Lemma \ref{lemma:Di-concen-eps2},
under the good event in  Lemma \ref{lemma:Di-concen-eps2} 1),
\eqref{eq:degree-D-concen-eps2} and then \eqref{eq:degree-D-concen-rel} hold.
Notation of $\varepsilon^{(D)}_i$ as in \eqref{eq:degree-D-concen-rel}, and omitting $h$ in the notations $m_2$, $m_0$,
we have that
\begin{align*}
\tilde{E}_N(  u )
& = \frac{1}{ \frac{m_2}{   m_0^2}\epsilon }  \frac{1}{N^2} \sum_{i,j=1}^N W_{i,j} \frac{ ( u_i - u_j)^2  }{ \frac{D_i}{N} \frac{D_j}{N}} \\
& = \frac{1}{  m_2 \epsilon } 
\frac{1}{N^2} \sum_{i,j=1}^N W_{i,j} \frac{ ( u_i - u_j)^2  }{  p(x_i)p(x_j) ( 1+ \varepsilon^{(D)}_i)   ( 1+ \varepsilon^{(D)}_j) } \\
& = \frac{1}{  m_2 \epsilon } 
\frac{1}{N^2} \sum_{i,j=1}^N W_{i,j} \frac{ ( u_i - u_j )^2  }{  p(x_i)p(x_j) }( 1+ \varepsilon_{ij} ), \quad \varepsilon_{ij} = O( \varepsilon^{(D)}_i, \varepsilon^{(D)}_j ) \\
& = 
\left( \frac{1}{  m_2 \epsilon } 
\frac{1}{N^2} \sum_{i,j=1}^N W_{i,j} \frac{ ( u_i - u_j )^2  }{  p(x_i)p(x_j) } \right)
(1 + O(\epsilon,  \sqrt{ \frac{\log N}{N \epsilon^{d/2}} }) ),
\end{align*}
where the last row uses the non-negativity of $W_{i,j} \frac{ ( u_i - u_j )^2  }{  p(x_i)p(x_j) }$.
\end{proof}

Proof of \eqref{eq:proof-form-density-correct-Vstat} in the proof of  Theorem \ref{thm:form-rate-density-correction}:

\begin{proof}
\underline{Proof of \eqref{eq:proof-form-density-correct-Vstat} }:
By definition, for $i \neq j$,
\begin{align*}
\E V_{i,j}
& = \frac{1}{ m_2 \epsilon } \int_\calM \int_\calM K_\epsilon( x, y) (f(x) - f(y))^2    dV(x) dV(y) \\
& =\frac{2}{ m_2 \epsilon } \int_\calM f(x) \left( \int_\calM  K_\epsilon(x,y) (f(x) - f(y)) dV(y) \right)dV(x) 
\end{align*}
By Lemma {A.3} in \cite{cheng2020convergence},
$
 \int_\calM  K_\epsilon(x,y) (f(x) - f(y)) dV(y)
 = -  \epsilon \frac{m_2}{ 2}  \Delta f(x) + O_f(\epsilon^2)$,
and thus,
\[
\E V_{i,j} =  \langle f, - \Delta f \rangle + O_f(\epsilon).
\]
Meanwhile, by that $p \ge  p_{min} > 0$,
$0 \le V_{ij} \le \Theta_p(1)  \frac{1}{m_2 \epsilon } K_\epsilon( x_i, x_j) (f(x_i) - f(x_j))^2$,  
and then by the boundedness and variance calculation in the proof of Theorem {3.4} of \cite{cheng2020convergence}, 
one can verify that, with constants depending on $(f,p)$,
\[
|V_{ij}| \le L=  \Theta( \epsilon^{-d/2}),
\quad
\E V_{ij}^2 \le \nu = \Theta(\epsilon^{-d/2}).
\]
Then, by the same decoupling argument to derive the concentration of  V-statistics, 
under good event $E_3$ which happens  w.p. $> 1-2N^{-10}$, 
\[
\frac{1}{N(N-1)} \sum_{i\neq j, i,j=1}^N V_{ij} = \E V_{ij} + O_{f,p}( \sqrt{ \frac{\log N }{N \epsilon^{d/2}} }).
\]
As a result,
\[
\text{ \textcircled{3} in \eqref{eq:form-pf-1} }
 = (1-\frac{1}{N})  \frac{1}{N(N-1)} \sum_{i\neq j, i,j=1}^N V_{ij} 
 = (1-\frac{1}{N}) \left(    \langle f, - \Delta f \rangle + O_f(\epsilon) + O_{f,p}( \sqrt{ \frac{\log N }{N \epsilon^{d/2}} }) \right),
 \]
which proves \eqref{eq:proof-form-density-correct-Vstat} because $O(\frac{1}{N})$ is higher order than $O( \sqrt{ \frac{\log N }{N \epsilon^{d/2}} }) $.
\end{proof}

\subsection{Proofs of the eigen-convergence of $\tilde{L}_{rw}$}

\begin{proof}[Proof of Proposition \ref{prop:eigvalue-UB-rw-density-correct}]
The proof is similar to that of Proposition \ref{prop:eigvalue-UB-rw}.
We first restrict to the good events $E_1 \cap E_2$ in Lemma \ref{lemma:Di-concen-eps2},
which happens w.p. $> 1- 4 N^{-9}$,  
under which  $\tilde{W}$ and $\tilde{L}_{rw}$ are well-defined,
and \eqref{eq:degree-D-concen-eps2} and \eqref{eq:denominator} hold.

Let $u_k =  \rho_X \psi_k $. 
The following lemma, proved in below,
shows the near $\tilde{D}$-orthonormal of the vectors $u_k$ and is an analogue of Lemma \ref{lemma:rhoX-isometry-whp}.

\begin{lemma}\label{lemma:uk-isometry-whp}
Under the same assumption of Lemma \ref{lemma:Di-concen-eps2},
when $N$ is sufficiently large, w.p. $> 1- 4 N^{-9} -  2 K^2 N^{-10}$,  
\begin{equation}\label{eq:uk-near-orthonormal-tiildepsik}
\begin{split}
 \| \rho_X {\psi}_k \|_{\tilde{D}}^2 
& = \frac{  1}{  m_0}  ( 1 +  O(\epsilon,  \sqrt{ \frac{\log N}{N \epsilon^{d/2}} })),
\quad 1 \le k \le K; \\
 ( \rho_X {\psi}_k)^T \tilde{D} ( \rho_X {\psi}_l) 
& =   O(\epsilon,  \sqrt{ \frac{\log N}{N \epsilon^{d/2}} }),
\quad  k\neq l, \, 1 \le k,l \le K.
\end{split}
\end{equation}
\end{lemma}

Under the good event of Lemma \ref{lemma:uk-isometry-whp}, called $E_5 \subset E_1 \cap E_2$,
$\tilde{D}_i > 0$ for all $i$,
and with large enough $N$, the set $\{ \tilde{D}^{1/2} u_k \}_{k=1}^K$ is linearly independent,
and then so is the set $\{ u_k \}_{k=1}^K$.
Let $L = \text{Span}\{ u_1, \cdots, u_k\}$, then $dim(L) = k$ for each $k \le K$.
For any $v \in L$,  $ v \neq 0$,
there are $c_j$, $1 \le j \le k $, such that 
$v = \sum_{j=1}^k c_j u_j$.
By \eqref{eq:uk-near-orthonormal-tiildepsik}, we have
\begin{equation}\label{eq:m0-v-tildeD-2} 
m_0 \|v\|_{\tilde{D}}^2 
= \| c \|^2 ( 1+ O( \epsilon,  \sqrt{ \frac{\log N}{ N \epsilon^{d/2}}}) ).
\end{equation}
Meanwhile, 
by defining $\tilde{B}_N(u,v) := \frac{1}{4} ( \tilde{E}_N(u+v) - \tilde{E}_N(u-v))$,
similarly as in Lemma \ref{lemma:form-rate-psi},
applying Theorem \ref{thm:form-rate-density-correction} to the $K^2$ cases where $ f = \psi_k$ and $(\psi_k \pm \psi_l)$ gives that,
under a good event $E_6$ which happens w.p.$>1- 2 K^2 N^{-10}$,
\begin{equation}\label{eq:tildeform-rate-psi}
\begin{split}
 \tilde{E}_N( \rho_X \psi_k) 
 & =  \mu_k + O(\epsilon, \sqrt{  \frac{  \log N    }{ N \epsilon^{d/2  }}   }),
\quad k = 1, \cdots, K, \\
 \tilde{B}_N( \rho_X \psi_k, \rho_X \psi_l) 
 & = O(\epsilon, \sqrt{  \frac{  \log N    }{ N \epsilon^{d/2  }}   }),
 \quad k \neq l, \, 1 \le k,l \le K.
\end{split}
\end{equation}
Then, similar as in \eqref{eq:UB-ENv},
\begin{align}
\tilde{E}_N( v) 
& = \sum_{j,l = 1}^k c_j c_k \tilde{B}_N( u_j, u_k) 
= \sum_{j=1}^k c_j^2 \left(  \mu_j + O(\epsilon ,  \sqrt{  \frac{  \log N    }{ N \epsilon^{d/2  }}   } ) \right)  
+ \sum_{j\neq l, j,l=1}^k |c_j | | c_l |  O(\epsilon ,   \sqrt{  \frac{  \log N    }{ N \epsilon^{d/2  }}   } )  \nonumber \\
& = \sum_{j=1}^k \mu_j c_j^2 + \| c \|^2 K  O(\epsilon ,   \sqrt{  \frac{  \log N    }{ N \epsilon^{d/2  }}   } )
 \le  \| c \|^2  \left(  \mu_k +  O(\epsilon ,   \sqrt{  \frac{  \log N    }{ N \epsilon^{d/2  }}   } ) \right).
\label{eq:UB-tildeENv}
\end{align}
Back to the r.h.s. of \eqref{eq:lambdak-rw-density-correct},
together with \eqref{eq:m0-v-tildeD-2},
we have that
\begin{equation}
 \frac{  \frac{1}{m_0 } \tilde{E}_N(v) }{   v^T \tilde{D} v }
 \le 
 \frac{   \mu_k +  O(\epsilon ,   \sqrt{  \frac{  \log N    }{ N \epsilon^{d/2  }}   } )}{   1+ O( \epsilon,  \sqrt{ \frac{\log N}{ N \epsilon^{d/2}}}) }
 =  \mu_k +  O(\epsilon ,   \sqrt{  \frac{  \log N    }{ N \epsilon^{d/2  }}   } ),
\end{equation}
and thus provides an UB of $\lambda_k$. 
The bound holds for all the $ 1 \le k \le  K$, under good events $E_5 \cap E_6$. 
\end{proof}

\begin{proof}[Proof of Lemma \ref{lemma:uk-isometry-whp}]
Restrict to the good events $E_1 \cap E_2$ in Lemma \ref{lemma:Di-concen-eps2}, which happens w.p. $> 1- 4 N^{-9}$,  
under which  $\tilde{W}$ and $\tilde{L}_{rw}$ are well-defined,
and  \eqref{eq:NtildeDi} holds.  Then,
\[
 \| \rho_X {\psi}_k \|_{\tilde{D}}^2 
 =  \frac{1}{N} \sum_{i=1}^N    \frac{ {\psi}_k (x_i)^2 }{ m_0 p(x_i)}( 1 +  O(\epsilon,  \sqrt{ \frac{\log N}{N \epsilon^{d/2}} })) 
 = \frac{ \| \rho_X ( p^{-1/2} \psi_k) \|^2 }{N  m_0}  ( 1 +  O(\epsilon,  \sqrt{ \frac{\log N}{N \epsilon^{d/2}} })), 
 \quad 1 \le k \le K,
\]
\[
\| \rho_X  ( {\psi}_k  \pm   {\psi}_l )\|_{\tilde{D}}^2  = 
\frac{ \| \rho_X ( p^{-1/2} ({\psi}_k  \pm  {\psi}_l) ) \|^2 }{N  m_0}  ( 1 +  O(\epsilon,  \sqrt{ \frac{\log N}{N \epsilon^{d/2}} })), 
\quad k \neq l, 1 \le k, l \le K.
\]
Apply \eqref{eq:rhoXf-iso} to when $f = p^{-1/2}{\psi}_k$ and $p^{-1/2} ({\psi}_k \pm {\psi}_l)$ for $k \neq l$, 
and recall that $\langle \psi_k, \psi_l \rangle = \delta_{kl}$, we have
\[
\frac{ 1 }{N}\| \rho_X ( p^{-1/2} {\psi}_k ) \|^2 = 1+ O(\sqrt{ \frac{\log N}{ N}}), 
\quad
\frac{ 1 }{N}\| \rho_X ( p^{-1/2} ({\psi}_k \pm \psi_l) ) \|^2 = 2+ O(\sqrt{ \frac{\log N}{ N}}),
\]
under a good event which happens w.p.$>1-2K^2 N^{-10}$ with large enough $N$, and then
\begin{align*}
 \| \rho_X {\psi}_k \|_{\tilde{D}}^2 
 & = \frac{  1}{  m_0}  ( 1 +  O(\epsilon,  \sqrt{ \frac{\log N}{N \epsilon^{d/2}} })), 
 \quad 1 \le k \le K, \\
\| \rho_X  ( {\psi}_k  \pm   {\psi}_l )\|_{\tilde{D}}^2 
&  = 
\frac{ 2 }{  m_0}  ( 1 +  O(\epsilon,  \sqrt{ \frac{\log N}{N \epsilon^{d/2}} })), 
\quad k \neq l, 1 \le k, l \le K,
\end{align*}
which proves \eqref{eq:uk-near-orthonormal-tiildepsik}.
\end{proof}

\begin{proof}[Proof of Proposition \ref{prop:eigvalue-LB-crude-rw-density-correct}]
The proof follows the same strategy of proving Proposition \ref{prop:eigvalue-LB-crude-rw},
where we introduce weights by $p(x_i)$
in the heat kernel interpolation map 
when  constructing candidate 
eigenfunctions from eigenvectors. 

We restrict to the good event $E_{UB}''$ in Proposition \ref{prop:eigvalue-UB-rw-density-correct},
which is contained in $E_1 \cap E_2$ in Lemma \ref{lemma:Di-concen-eps2}.
Under $E_{UB}''$,  $D_i > 0 $, $\tilde{D}_i > 0$, and $\tilde{L}_{rw}$ is well-defined,
and, with sufficiently large $N$, 
$\lambda_k \le  \lambda_K \le 1.1 \mu_K = O(1)$. 
Let $\tilde{L}_{rw} v_k = \lambda_k v_k$,
normalized s.t.
\[
v_k^T \tilde{D} v_l = \delta_{kl},
\quad
1 \le k, l \le N.
\]
Note that always $\lambda_1 = 0$.
Under $E_1 \cap E_2$, \eqref{eq:NtildeDi} holds,
and thus
\begin{equation}\label{eq:m0-u2-tildeD-2}
m_0 \| u \|_{\tilde{D}}^2 
= \frac{m_0 }{N}\sum_{i=1}^N u_i^2 (N \tilde{D}_i)
= \left( \frac{1}{N}\sum_{i=1}^N   \frac{u_i^2}{  p(x_i)} \right)( 1 +  O(\epsilon,  \sqrt{ \frac{\log N}{N \epsilon^{d/2}} })),
\quad 
\forall u \in \R^N,
\end{equation}
and the constant in big-$O$ is determined by $(\calM, p)$ and uniform  for all $u$.
Define the notation
\begin{equation}\label{eq:def-u-pinv-norm}
 \| u \|_{p^{-1}}^2  := \frac{1}{N} \sum_{i=1}^N   \frac{ u_i^2}{p(x_i)},
 \quad \forall u \in \R^N.
\end{equation}
Taking $u $ to be $v_k$ and $(v_k \pm v_l)$ gives that 
\begin{equation}\label{eq:vk-pinv-norm-m0}
\begin{split}
m_0 
& =  \| v_k\|_{p^{-1}}^2
 ( 1 +  O(\epsilon,  \sqrt{ \frac{\log N}{N \epsilon^{d/2}} })),
\quad 1 \le k \le K, \\
2 m_0 
& = \| v_k\pm v_l\|_{p^{-1}}^2
 ( 1 +  O(\epsilon,  \sqrt{ \frac{\log N}{N \epsilon^{d/2}} })),
\quad k \neq l,  1 \le k,l \le K.
\end{split}
\end{equation}
Set $\delta$, $r$, $t$ in the same way as in the proof of Proposition \ref{prop:eigvalue-LB-crude-rw},
and define $\tilde{I}_r[u]$ as in \eqref{eq:def-tilde-Ir}.
We have 
$\langle \tilde{I}_r [u], \tilde{I}_r  [u] \rangle = q_{\delta \epsilon} (\tilde{u})$,
$\langle \tilde{I}_r [u], Q_t \tilde{I}_r  [u] \rangle = q_{ \epsilon} (\tilde{u})$,
and  \eqref{eq:def-tilde-qs} for $s > 0$.
Next, 
similar as in the proof of Lemma \ref{lemma:qs0-concen},
one can show that with large $N$ and w.p.$>1-2N^{-9}$,
\begin{equation}\label{eq:tildeDsi-concen}
\frac{1}{N} \sum_{j=1}^N  \frac{ \calH_s(x_i, x_j) }{p(x_i) p(x_j) } 
= \frac{1}{p(x_i)} (1 + O_{\calM,p}(\sqrt{\frac{\log N}{N s^{d/2}}})),
\quad 
1 \le i \le N,
\end{equation}
where the notation $O_{\calM,p}(\cdot)$ indicates that the constant depends on $(\calM,p)$ and is uniform for all $x_i$.
Applying \eqref{eq:tildeDsi-concen} 
to $s = \delta \epsilon$ gives that, under a good event $E_{(0)}'$, which happens w.p.$>1-2N^{-9}$,
\begin{align}
\tilde{q}^{(0)}_{\delta \epsilon }(u)  
&= \frac{1}{N} \sum_{i=1}^N   \frac{u_i^2}{p(x_i)} (1 + O_{\calM,p}(\delta^{-d/4} \sqrt{\frac{\log N}{N \epsilon^{d/2}}}))  \nonumber \\
& =  \| u \|_{p^{-1}}^2 (1 + O_{\calM,p}(\delta^{-d/4}  \sqrt{\frac{\log N}{N \epsilon^{d/2}}})),
\quad \forall u \in \R^N.
\label{tildeq0-u-E0'}
\end{align}
Applying \eqref{eq:tildeDsi-concen} 
to $s = \epsilon$ gives the good event $E_{(1)}'$, which happens w.p.$>1-2N^{-9}$, under which 
\begin{equation}\label{tildeq0-u-E1'}
\tilde{q}^{(0)}_{\epsilon }(u)  
=  \| u \|_{p^{-1}}^2 (1 + O_{\calM,p}(\sqrt{\frac{\log N}{N \epsilon^{d/2}}})),
\quad \forall u \in \R^N.
\end{equation}
The constants in big-$O$  in 
\eqref{tildeq0-u-E0'} and \eqref{tildeq0-u-E1'}
are determined by $(\calM, p)$ only and uniform for all $u$.

We also need an analogue of Lemma \ref{lemma:qs2-UB} to upper bound $\tilde{q}^{(2)}_s$, proved in below. 
The proof follows same method of Lemma  \ref{lemma:qs2-UB},
and makes use of the uniform boundedness of $p$ from below, and Lemma \ref{lemma:tildeENu-V-stat}.

\begin{lemma}\label{lemma:qs2-UB-density-correct}
Under Assumption \ref{assump:M-p}, $h$ being Gaussian, 
let $0 < \alpha < 1$ be a fixed constant.
Suppose $\epsilon = o(1)$, $\epsilon^{d/2} = \Omega( \frac{\log N}{ N})$,
then with sufficiently large $N$,
and under the good event $E_1$ of Lemma \ref{lemma:Di-concen-eps2} 1),
\begin{equation}\label{eq:q2-eps-UB-density-correct}
0 \le \tilde{q}^{(2)}_{ \epsilon}(u) 
= \left( 1 + O \left(  \epsilon  (\log \frac{1}{ \epsilon})^2,  \sqrt{\frac{\log N }{ N \epsilon^{d/2}}} \right) \right)  
( u^T( \tilde{D}-\tilde{W}) u )
 +  \| u \|_{p^{-1}}^2  O(\epsilon^{3}),
\quad 
\forall u \in \R^N,
\end{equation}
and
\begin{equation}\label{eq:q2-alphaeps-UB-density-correct}
0 \le \tilde{q}^{(2)}_{ \alpha \epsilon}(u) 
\le 1.1 \alpha^{-d/2} ( u^T( \tilde{D}-\tilde{W}) u ) +  \| u \|_{p^{-1}}^2   O(\epsilon^{3}),
\quad 
\forall u \in \R^N.
\end{equation}
The constants in big-$O$ only depend on $(\calM,p)$ and are uniform for all $u$ and $\alpha$.
\end{lemma}

We proceed to define  $f_k = \tilde{I}_r[v_k]$, $f_k \in C^{\infty}(\calM)$.
Next, note that since $ ( I-\tilde{D}^{-1}\tilde{W}) v_k = \epsilon  \lambda_k v_k $, and $v_k$ are $\tilde{D}$-orthonormal, then 
\begin{equation}\label{eq:form-vk-vl-tildeD-tildeD}
\begin{split}
 v_k^T ( \tilde{D}-\tilde{W}) v_k 
 & = \epsilon  \lambda_k  v_k^T \tilde{D} v_k = \epsilon \lambda_k,
 \quad 1  \le k \le K, \\
(v_k \pm v_l)^T ( \tilde{D}-\tilde{W}) (v_k \pm v_l)
 & = \epsilon   ( \lambda_k + \lambda_l), 
 \quad k \neq l, 1 \le k,l \le K.
 \end{split}
\end{equation}
Taking $\alpha = \delta$ in Lemma \ref{lemma:qs2-UB-density-correct}, \eqref{eq:q2-alphaeps-UB-density-correct} then gives
\[
\begin{split}
\tilde{q}^{(2)}_{\delta \epsilon }( v_k) 
& =
O( \delta^{-d/2} ) \epsilon 
\lambda_k 
 +  O(\epsilon^{3}) , 
 \quad 1 \le k \le K, \\
\tilde{q}^{(2)}_{\delta \epsilon }( v_k \pm v_l) 
&  
= O( \delta^{-d/2})   \epsilon 
( \lambda_k + \lambda_l) 
+ 2  O(\epsilon^{3}) ,
 \quad k \neq l, 1 \le k,l \le K,
\end{split}
\]
and both are $O(\epsilon)$.
Meanwhile, 
\eqref{tildeq0-u-E0'}and \eqref{eq:vk-pinv-norm-m0} give that
(with that $\delta >0$ is a fixed constant determined by $K$ and $-\Delta$)
\begin{equation}
\begin{split}
\tilde{q}^{(0)}_{\delta \epsilon }(v_k)  
&=  \| v_k \|_{p^{-1}}^2 (1 + O(\sqrt{\frac{\log N}{N \epsilon^{d/2}}}))
= m_0(1 + O( \epsilon, \sqrt{\frac{\log N}{N \epsilon^{d/2}}})),
\quad 1 \le k \le K, \\
\tilde{q}^{(0)}_{\delta \epsilon }(v_k \pm v_l)  
&=  \| v_k \pm v_l\|_{p^{-1}}^2 (1 + O(\sqrt{\frac{\log N}{N \epsilon^{d/2}}}))
= 2m_0(1 + O( \epsilon, \sqrt{\frac{\log N}{N \epsilon^{d/2}}})),
\quad k \neq l, \, 1\le k, l \le K.
\end{split}
\end{equation}
Putting together with the bounds of $q_{\delta \epsilon}^{(2)}$,
this gives that 
\begin{equation}
\begin{split}
\langle f_k, f_k \rangle 
&  = \tilde{q}_{\delta \epsilon}^{(0)}(v_k) - \tilde{q}_{\delta \epsilon}^{(2)}(v_k) 
= m_0(1 + O( \epsilon, \sqrt{\frac{\log N}{N \epsilon^{d/2}}})) - O(\epsilon),
 \quad 1 \le k \le K, \\
\langle f_k, f_l \rangle 
& = \frac{1}{4}( \tilde{q}_{\delta \epsilon}( v_k + v_l)  -  \tilde{q}_{\delta \epsilon}( v_k - v_l) )
= O( \epsilon, \sqrt{\frac{\log N}{ N \epsilon^{d/2}}}) 
+ O(\epsilon),
\quad k \neq l, \, 1\le k, l \le K.
\end{split}
\end{equation}
Then due to that $O( \epsilon, \sqrt{\frac{\log N}{ N \epsilon^{d/2}}}) = o(1)$, 
we have linear independence of $\{ f_j\}_{j=1}^K$ with large enough $N$.

Same as before, for any $2 \le k \le K$, we let $L_k = \text{Span}\{ f_1, \cdots, f_k\}$,
and have \eqref{eq:eigen-relation-2}.
For any $f \in L_k$, $ f = \sum_{j=1}^k c_j f_j $,
$f = \tilde{I}_r [v]$ , $v := \sum_{j = 1}^k c_j v_j$, and
 \[
v^T \tilde{D} v 
  = \sum_{j=1}^k c_j^2 v_j^T \tilde{D} v_j = \| c \|^2.
\]
Meanwhile, by \eqref{eq:m0-u2-tildeD-2},  $m_0 =1$,
\begin{equation}\label{eq:m0-c2-v}
 \|c\|^2
=  \| v \|_{\tilde{D}}^2 
= \| v \|_{p^{-1}}^2
( 1 +  O(\epsilon,  \sqrt{ \frac{\log N}{N \epsilon^{d/2}} })),
\end{equation}
and by \eqref{eq:form-vk-vl-tildeD-tildeD},
\begin{equation}\label{eq:v-form-density-correct}
v^T (\tilde{D} - \tilde{W}) v = \epsilon \sum_{j=1}^k \lambda_j c_j^2 
\le \epsilon \| c\|^2  \lambda_k.
\end{equation}
Then, as we work under $E^{(0)} \cap E^{(1)}$,  \eqref{tildeq0-u-E0'} and \eqref{tildeq0-u-E1'} hold.
Applying to $u=v$ and subtracting the two,
\begin{align*}
\tilde{q}^{(0)}_{\delta \epsilon }(v)  
- \tilde{q}^{(0)}_{\epsilon }(v)  
& = \| v \|_{p^{-1}}^2  O_{\calM,p}( \delta^{-d/4}\sqrt{\frac{\log N}{N \epsilon^{d/2}}})
= \|c\|^2 ( 1 +  O(\epsilon,  \sqrt{ \frac{\log N}{N \epsilon^{d/2}} })) O_{\calM,p}( \delta^{-d/4} \sqrt{\frac{\log N}{N \epsilon^{d/2}}}) \\
& = \|c\|^2 O_{\calM,p}( \delta^{-d/4} \sqrt{\frac{\log N}{N \epsilon^{d/2}}}),
\end{align*}
where we used \eqref{eq:m0-c2-v} to obtain the 2nd equality. 
To upper bound $\tilde{q}^{(2)}_{\epsilon}(v) $,
  by \eqref{eq:q2-eps-UB-density-correct},
and with the shorthand that $\tilde{O}(\epsilon) $ stands for $ {O}( \epsilon  (\log \frac{1}{ \epsilon})^2) $,
\begin{align*}
\tilde{q}^{(2)}_{\epsilon}(v) 
& = \left( 1 + \tilde{O}(\epsilon )+ O (  \sqrt{\frac{\log N }{ N \epsilon^{d/2}}} ) \right)  
( v^T( \tilde{D}-\tilde{W}) v )
 +  \| v \|_{p^{-1}}^2  O(\epsilon^{3}) \\
& \le
\left( 1 + \tilde{O}(\epsilon )+ O (  \sqrt{\frac{\log N }{ N \epsilon^{d/2}}} ) \right)  
\epsilon \|c\|^2 \lambda_k
 +  \|c\|^2 ( 1 +  O(\epsilon,  \sqrt{ \frac{\log N}{N \epsilon^{d/2}} })) O(\epsilon^{3}) \\
&
\le \epsilon \| c \|^2 \left\{ 
 \lambda_k  \left( 1 + \tilde{O}(\epsilon )+ O (  \sqrt{\frac{\log N }{ N \epsilon^{d/2}}} ) \right)  + O( \epsilon^{2}) 
 \right\}.
\end{align*}
Thus we have that
\begin{align}
& \langle f ,  f  \rangle  -  \langle f ,  Q_t f  \rangle
 \le (  q^{(0)}_{\delta \epsilon}( v) - q^{(0)}_{\epsilon}(v) ) + q^{(2)}_{\epsilon}(v) \nonumber \\
&~~~
 \le  \epsilon \| c \|^2 \left\{
 \lambda_k  \left( 1 + \tilde{O}(\epsilon )+ O (  \sqrt{\frac{\log N }{ N \epsilon^{d/2}}} ) \right)  + O( \epsilon^{2}) 
 +  O_{\calM,p}( \delta^{-d/4}  \frac{1}{\epsilon}\sqrt{ \frac{\log N}{N  \epsilon^{d/2}} })
 \right\}  \nonumber  \\
 & ~~~
 =  \epsilon \| c \|^2 \left\{ 
 \lambda_k 
  +  \tilde{O}(\epsilon) 
 +  O_{\calM,p}( \delta^{-d/4}  \frac{1}{\epsilon}\sqrt{ \frac{\log N}{N  \epsilon^{d/2}} })
 \right\}.
\quad \text{(by $\lambda_k \le 1.1 \mu_K$)}
 \label{eq:numerator-UB-density-correct}
\end{align}

To lower bound $\langle f, f \rangle$,
again  by \eqref{eq:q2-alphaeps-UB-density-correct},
\eqref{eq:m0-c2-v} and \eqref{eq:v-form-density-correct},
\[
0 \le \tilde{q}^{(2)}_{\delta \epsilon}( v)
  \le  \Theta( \delta^{-d/2} )  ( v^T (\tilde{D}-\tilde{W}) v)  + \| v\|_{p^{-1}}^2 O(\epsilon^{3})
\le  \epsilon  \| c \|^2   \left( \lambda_k 
 \Theta(\delta^{-d/2}) + O(\epsilon^{2}) \right)
 =   \| c \|^2   O(\epsilon).
\]
By \eqref{tildeq0-u-E0'} and \eqref{eq:m0-c2-v},
\begin{equation}
\tilde{q}^{(0)}_{\delta \epsilon }(v)  
=  \| v \|_{p^{-1}}^2 (1 + O(\sqrt{\frac{\log N}{N \epsilon^{d/2}}}))
= \|c\|^2 ( 1 +  O(\epsilon,  \sqrt{ \frac{\log N}{N \epsilon^{d/2}} })),
\end{equation}
Thus,
\[
\langle f, f \rangle 
= \tilde{q}^{(0)}_{\delta \epsilon}( v) - \tilde{q}^{(2)}_{\delta \epsilon}( v) 
= \| c \|^2 \left ( 1 + O( \epsilon, \sqrt{ \frac{\log N}{N  \epsilon^{d/2}} })  - O( \epsilon ) \right) 
\ge \| c \|^2 \left ( 1 - O( \epsilon,  \sqrt{ \frac{\log N}{N  \epsilon^{d/2}} })  \right).
\]
the rest of the proof is the same as that in Proposition \ref{prop:eigvalue-LB-crude-rw},
where the constant $C$ is defined as
$C = c_{\calM,p} \delta^{-d/4}$,
$c_{\calM,p}$ being a constant determined by $(\calM,p)$,
and then the constant $c$ in the definition of $c_K$ also depends on $p$.
The needed good events are  $E_{(0)}'$, $E_{(1)}'$, and $E_{UB}''$,
and the LB holds for $k \le K$.
\end{proof}

\begin{proof}[Proof of Lemma \ref{lemma:qs2-UB-density-correct}]
By definition, for any $u \in \R^N$,
\[
\tilde{q}^{(2)}_{ \epsilon}(u) = \frac{1}{2} \frac{1}{N^2} \sum_{i,j=1}^N  \frac{\calH_\epsilon (x_i,x_j)}{ p(x_i) p (x_j)} (u_i - u_j)^2 \ge 0.
\]
Take $t$ in Lemma \ref{lemma:heat} to be $\epsilon$,
since $\epsilon = o(1)$, the three equations hold when $\epsilon < \epsilon_0$.
By \eqref{eq:H-eps-truncate},
truncate at an $\delta_\epsilon = \sqrt{ 6(10 + \frac{d}{2}) \epsilon \log{\frac{1}{\epsilon}}}$ Euclidean ball, 
there is $C_3$, a positive constant determined by $\calM$, s.t.
\[
 \frac{1}{2} \frac{1}{N^2} \sum_{i,j=1}^N \frac{\calH_\epsilon (x_i,x_j)}{ p(x_i) p (x_j)} {\bf 1}_{\{ x_j \notin B_{\delta_\epsilon}(x_i) \}}
 (u_i - u_j)^2
\le C_3 \epsilon^{10} \frac{1}{N^2} \sum_{i,j=1}^N \frac{ (u_i - u_j)^2}{ p(x_i) p (x_j)} {\bf 1}_{\{ x_j \notin B_{\delta_\epsilon}(x_i) \}}.
\]
Note that 
\begin{align}
& 
 \frac{1}{N^2} \sum_{i,j=1}^N \frac{ (u_i - u_j)^2}{ p(x_i) p (x_j)} 
 =
 \frac{2 }{N} \sum_{i=1}^N \frac{ u_i^2 }{ p(x_i) } 
\left(  \frac{1}{N}\sum_{j=1}^N \frac{1}{p (x_j)} \right)
-  2 \left(\frac{1}{N} \sum_{i=1}^N \frac{ u_i }{ p(x_i) } \right)^2  
\nonumber \\
& \le  \frac{2}{N} \sum_{i=1}^N \frac{ u_i^2 }{ p(x_i) }  \left( \frac{1}{N}\sum_{j=1}^N \frac{1}{p (x_j)} \right)
\le  \frac{2}{N} \sum_{i=1}^N \frac{ u_i^2 }{ p(x_i) }  \frac{1}{p_{min}} 
= \frac{2 }{p_{min}}   \| u \|^2_{p^{-1}},
\label{eq:bound-quadratic-sum-with-p}
\end{align}
thus,
\begin{equation}\label{eq:truncate-quad-Heps-with-p}
\tilde{q}^{(2)}_{ \epsilon}(u) 
= \frac{1}{2} \frac{1}{N^2} \sum_{i,j=1}^N \frac{\calH_\epsilon (x_i,x_j)}{ p(x_i) p (x_j)} {\bf 1}_{\{ x_j \in B_{\delta_\epsilon}(x_i) \}}
 (u_i - u_j)^2 
 +  \| u \|^2_{p^{-1}} O( \epsilon^{10}).
\end{equation}
Apply \eqref{eq:H-eps-local} with the short hand that $\tilde{O}(\epsilon) $ stands for $ {O}( \epsilon  (\log \frac{1}{ \epsilon})^2) $, 
\begin{align*}
&  \tilde{q}^{(2)}_{ \epsilon}(u) 
= \frac{1}{2} \frac{1}{N^2} \sum_{i,j=1}^N 
\frac{   K_\epsilon( x_i,x_j) (1 + \tilde{O}( \epsilon)) + O(\epsilon^3) }{p(x_i)p(x_j)} 
{\bf 1}_{\{ x_j \in B_{\delta_\epsilon}(x_i) \}}
 (u_i - u_j)^2 
+  \| u \|^2_{p^{-1}} O( \epsilon^{10}) \\
 & =
(1 + \tilde{O}( \epsilon))
 \frac{1}{2} \frac{1}{N^2} \sum_{i,j=1}^N 
\frac{   K_\epsilon( x_i,x_j)  }{p(x_i)p(x_j)}  {\bf 1}_{\{ x_j \in B_{\delta_\epsilon}(x_i) \}}  (u_i - u_j)^2 
+
 O(\epsilon^3) \frac{1}{N^2} \sum_{i,j=1}^N 
\frac{  (u_i - u_j)^2    }{p(x_i)p(x_j)} 
+  \| u \|^2_{p^{-1}} O( \epsilon^{10})\\
& =
(1 + \tilde{O}( \epsilon))
 \frac{1}{2} \frac{1}{N^2} \sum_{i,j=1}^N 
\frac{   K_\epsilon( x_i,x_j)  }{p(x_i)p(x_j)}  {\bf 1}_{\{ x_j \in B_{\delta_\epsilon}(x_i) \}}  (u_i - u_j)^2 
+ \| u \|^2_{p^{-1}}   O(\epsilon^3)
\quad \text{(by \eqref{eq:bound-quadratic-sum-with-p})}.
\end{align*}
The truncation for $K_\epsilon (x_i,x_j)$ gives that 
$K_\epsilon( x_i,x_j) {\bf 1}_{\{ x_j \notin B_{\delta_\epsilon}(x_i) \}} = O(\epsilon^{10})$,
and then similarly as in \eqref{eq:truncate-quad-Heps-with-p},
\begin{equation}\label{eq:quad-W-u-with-p-trunc}
\frac{1}{2} \frac{1}{N^2} \sum_{i,j=1}^N 
\frac{   K_\epsilon( x_i,x_j)  }{p(x_i)p(x_j)}  {\bf 1}_{\{ x_j \in B_{\delta_\epsilon}(x_i) \}}  (u_i - u_j)^2 
= \frac{1}{2} \frac{1}{N^2} \sum_{i,j=1}^N 
\frac{   K_\epsilon( x_i,x_j)  }{p(x_i)p(x_j)}  (u_i - u_j)^2 
-  \| u \|^2_{p^{-1}} O( \epsilon^{10}).
\end{equation}
By Lemma \ref{lemma:tildeENu-V-stat}, and $m_2=2$ with Gaussian $h$,
we have that under the  good event $E_1$ of Lemma \ref{lemma:Di-concen-eps2} 1), 
\[
\tilde{E}_N( u )
= 
\left( \frac{1}{  2 \epsilon } 
\frac{1}{N^2} \sum_{i,j=1}^N W_{i,j} \frac{ ( u_i - u_j)^2  }{  p(x_i)p(x_j) } \right)
(1 + O(\epsilon,  \sqrt{ \frac{\log N}{N \epsilon^{d/2}} }) ),
\quad 
\forall u \in \R^N,
\] 
and the constant in big-$O$ is determined by $(\calM,p)$ and uniform for all $u$. 
This gives that
\begin{equation}\label{eq:Keps-quad-with-p-approx-form}
\frac{1}{2} \frac{1}{N^2} \sum_{i,j=1}^N 
\frac{   K_\epsilon( x_i,x_j)  }{p(x_i)p(x_j)}  (u_i - u_j)^2 
=  \epsilon \tilde{E}_N( u )
(1 + O(\epsilon,  \sqrt{ \frac{\log N}{N \epsilon^{d/2}} }) ),
\end{equation}
and as a result, together with \eqref{eq:quad-W-u-with-p-trunc},
\begin{align*}
  \tilde{q}^{(2)}_{ \epsilon}(u) 
  & =
(1 + \tilde{O}( \epsilon))
\left(  \epsilon \tilde{E}_N( u )
(1 + O(\epsilon,  \sqrt{ \frac{\log N}{N \epsilon^{d/2}} }) ) 
-  \| u \|^2_{p^{-1}} O( \epsilon^{10})
\right) 
+ \| u \|^2_{p^{-1}}   O(\epsilon^3) \\
& =
  \epsilon \tilde{E}_N( u )(1 + \tilde{O}( \epsilon) + O(  \sqrt{ \frac{\log N}{N \epsilon^{d/2}} }) )
+ \| u \|^2_{p^{-1}}   O(\epsilon^3).
\end{align*}
Recall that $\tilde{E}_N( u ) = \frac{1}{\epsilon} u^T( \tilde{D}- \tilde{W})u$,
this proves \eqref{eq:q2-eps-UB-density-correct}.

To prove \eqref{eq:q2-alphaeps-UB-density-correct},
since $ 0< \alpha \epsilon < \epsilon$,
apply  Lemma \ref{lemma:heat} with $t = \alpha \epsilon$, and similarly as in \eqref{eq:truncate-quad-Heps-with-p},

\begin{align*}
& \tilde{q}^{(2)}_{ \alpha \epsilon}(u) 
 = \frac{1}{2} \frac{1}{N^2} \sum_{i,j=1}^N \frac{\calH_{\alpha \epsilon } (x_i,x_j)}{ p(x_i) p (x_j)} {\bf 1}_{\{ x_j \in B_{\delta_{\alpha \epsilon} }(x_i) \}}
 (u_i - u_j)^2 
 +  \| u \|^2_{p^{-1}} O( \epsilon^{10})\\
 &= \frac{1}{2} \frac{1}{N^2} \sum_{i,j=1}^N \frac{ K_{\alpha\epsilon}( x_i,x_j) (1 + \tilde{O}( \alpha \epsilon)) + O( \alpha^3 \epsilon^3) 
 }{ p(x_i) p (x_j)} {\bf 1}_{\{ x_j \in B_{\delta_{\alpha \epsilon} }(x_i) \}}
 (u_i - u_j)^2 
 +  \| u \|^2_{p^{-1}} O( \epsilon^{10})
 \quad \text{(by \eqref{eq:H-eps-local})} \\
 &=
 (1 + \tilde{O}( \epsilon))  \frac{1}{2} \frac{1}{N^2} \sum_{i,j=1}^N \frac{ K_{\alpha\epsilon}( x_i,x_j) 
 }{ p(x_i) p (x_j)} {\bf 1}_{\{ x_j \in B_{\delta_{\alpha \epsilon} }(x_i) \}}
 (u_i - u_j)^2 
 +  \| u \|^2_{p^{-1}} O( \epsilon^{3}).
 \quad \text{(by \eqref{eq:bound-quadratic-sum-with-p})}
\end{align*}
Then, using \eqref{eq:K-alphaeps-K-eps},
\eqref{eq:quad-W-u-with-p-trunc} and \eqref{eq:Keps-quad-with-p-approx-form},
\begin{align*}
\tilde{q}^{(2)}_{ \alpha \epsilon}(u) 
 & \le   (1 + \tilde{O}( \epsilon)) \alpha^{-d/2}   \frac{1}{2N^2} \sum_{i,j=1}^N \frac{ K_{\epsilon}( x_i,x_j) 
 }{ p(x_i) p (x_j)} {\bf 1}_{\{ x_j \in B_{\delta_{\alpha  \epsilon} }(x_i) \}}
 (u_i - u_j)^2 
 +  \| u \|^2_{p^{-1}} O( \epsilon^{3}) \\
& =  (1 + \tilde{O}( \epsilon))  \alpha^{-d/2}
\left( \epsilon \tilde{E}_N( u )
(1 + O(\epsilon,  \sqrt{ \frac{\log N}{N \epsilon^{d/2}} }) ) 
-  \| u \|^2_{p^{-1}} O( \epsilon^{10}) \right)
 +  \| u \|^2_{p^{-1}} O( \epsilon^{3}) \\
& = (1 + \tilde{O}( \epsilon) + O(\epsilon,  \sqrt{ \frac{\log N}{N \epsilon^{d/2}} }) ) 
\alpha^{-d/2}
 \epsilon \tilde{E}_N( u )
 +  \| u \|^2_{p^{-1}} O( \epsilon^{3}),
\end{align*}
which proves \eqref{eq:q2-alphaeps-UB-density-correct}
because $\tilde{O}( \epsilon) + O(\epsilon,  \sqrt{ \frac{\log N}{N \epsilon^{d/2}} }) = o(1)$
and thus the constant in front of $\alpha^{-d/2}$ is less than 1.1 for sufficiently small $\epsilon$.
\end{proof}

\begin{proof}[Proof of Theorem \ref{thm:refined-rates-rw-density-correct}]
With sufficiently large $N$, 
we restrict to the intersection of 
the good events in Proposition \ref{prop:eigvalue-LB-crude-rw-density-correct} 
and the $K= k_{max}+1$ good events of applying Theorem \ref{thm:pointwise-rate-dencity-correct}
 to $\{ \psi_k \}_{k=1}^K$.
 Because the good event in Proposition \ref{prop:eigvalue-LB-crude-rw-density-correct} 
 is already under $E_{UB}''$ of Proposition \ref{prop:eigvalue-UB-rw-density-correct},
 and under  $E_1\cap E_2$ of Lemma \ref{lemma:Di-concen-eps2},
 the extra good events in addition to what is needed in Proposition \ref{prop:eigvalue-LB-crude-rw-density-correct} 
 are those corresponding to $E_3 \cap E_4$ in the proof of Theorem \ref{thm:pointwise-rate-dencity-correct}
 where  $f = \psi_k$ for each $1 \le k \le K$,
 and, by a union bound, happens w.p.$>1-K \cdot 4N^{-9}$.
This gives to the final high probability indicated in the theorem. 
 In addition,  $D_i > 0$, $\tilde{D}_i > 0$ for all $i$,
and $\tilde{L}_{rw}$ is well-defined.
 
 The rest of the proof follows similar method as that of Theorem \ref{thm:refined-rates-rw},
 but differs in the normalization of the eigenvectors 
 and that of the eigenfunctions.
 With the definition of $\| u \|_{\tilde{D}}$ and $ \| u\|_{p^{-1}}$ in \eqref{eq:def-u-tildeD-norm} and \eqref{eq:def-u-pinv-norm} respectively, 
As has been shown in \eqref{eq:m0-u2-tildeD-2}, under $E_1 \cap E_2$,
\begin{equation}\label{eq:tildeD-norm-and-pinv-norm}
 \| u \|_{\tilde{D}}^2 
=
 \| u \|_{p^{-1}}^2
( 1 +  O(\epsilon,  \sqrt{ \frac{\log N}{N \epsilon^{d/2}} })),
\quad \forall u \in \R^N,
\end{equation}
and the constant in big-O is determined by $(\calM, p)$ and uniform for all $u$.
This also gives that  with sufficiently large $N$, 
\begin{equation}\label{eq:equivalent-2norm-tildeD-norm}
\frac{0.9}{p_{max}} \frac{ \| u \|_2^2}{ N}  \le
0.9 \| u \|_{p^{-1}}^2 \le 
\| u \|_{\tilde{D}}^2 \le 1.1 \| u \|_{p^{-1}}^2 \le \frac{1.1}{p_{min}} \frac{ \| u \|_2^2}{ N},
\quad \forall u \in \R^N,
\end{equation}
because 
$\| u \|_{p^{-1}}^2 = \frac{1}{N} \sum_{i=1}^N \frac{u_i^2}{p(x_i)}$ is upper bounded by  $\frac{1}{ p_{min} N} \|u\|_2^2$
and lower bounded by $\frac{1.1}{p_{max}} \frac{ \| u \|_2^2}{ N}$.
Apply \eqref{eq:equivalent-2norm-tildeD-norm} to $u =v_k$, this gives that 
$\frac{0.9}{p_{max}} \| v_k \|_2^2
 \le
\| v_k \|_{\tilde{D}}^2 N = 1
 \le \frac{1.1}{p_{min}} \| v_k \|_2^2$,
that is 
\[
  \sqrt{\frac{p_{min}}{1.1}}  \le \| v_k \|_2 \le   \sqrt{\frac{p_{max}}{0.9}},
  \quad 1 \le  k \le K,
  \]
and this verifies that $\| v_k \|_2 = \Theta(1)$ under the high probability event.

Meanwhile, 
because the good event $E_{UB}''$ is under the one needed in 
Lemma \ref{lemma:uk-isometry-whp},
as shown in the proof of Lemma \ref{lemma:uk-isometry-whp}, we have that
\[
    \|  \rho_X {\psi}_k \|_{p^{-1}}^2 
= \frac{1}{N} \sum_{i=1}^N  \frac{\psi_k(x_i)^2}{ p(x_i)} 
= 1  + O( \sqrt{ \frac{\log N}{N}}),
\quad 1 \le k \le K,
\]
where the constant in big-$O$ depends on $(\calM, p)$ and is uniform for all $k \le K$.
By definition, $N   \|  \tilde{\phi}_k \|_{p^{-1}}^2 
=    \|  \rho_X {\psi}_k \|_{p^{-1}}^2$,
and then, apply \eqref{eq:tildeD-norm-and-pinv-norm} to $u = \tilde{\phi}_k$, 
\begin{equation}\label{eq:norm-tilde-tildephik}
 \|  \tilde{\phi}_k \|_{\tilde{D}}^2 
=
 \|  \tilde{\phi}_k \|_{p^{-1}}^2 ( 1 +  O(\epsilon,  \sqrt{ \frac{\log N}{N \epsilon^{d/2}} }))
 = 
 \frac{1}{N } 
  ( 1 +  O(\epsilon,  \sqrt{ \frac{\log N}{N \epsilon^{d/2}} })),
  \quad 
1 \le  k \le K.
\end{equation}
\\

\underline{ Step 2. for $\tilde{L}_{rw}$}: 
When $k=1$, $\lambda_1 =0$, and  $v_1$ is always the constant vector, thus the discrepancy is zero.
Consider $2 \le k \le K$, by Theorem \ref{thm:pointwise-rate-dencity-correct}
and that $\| u\|_2 \le \sqrt{N} \| u\|_\infty $,
\begin{equation}\label{eq:pontwise-rate-2norm-bound-rw-density-correct}
\| \tilde{L}_{rw} \tilde{\phi}_k - \mu_k \tilde{\phi}_k \|_2  
=  O(\epsilon, \sqrt{\frac{\log N}{N \epsilon^{d/2+1}}} ),
\quad 2 \le k \le K.
\end{equation}
Then, by \eqref{eq:equivalent-2norm-tildeD-norm},
$\sqrt{N}  \| \tilde{L}_{rw} \tilde{\phi}_k - \mu_k \tilde{\phi}_k \|_{\tilde{D}}
= 
O( \| \tilde{L}_{rw} \tilde{\phi}_k - \mu_k \tilde{\phi}_k \|_2 )
= O(\epsilon, \sqrt{\frac{\log N}{N \epsilon^{d/2+1}}} )$,
that is,
 there is $\text{Err}_{pt} > 0$, s.t.
\begin{equation}\label{eq:pontwise-rate-2norm-bound-rw-tildeD}
\sqrt{N}\| L_{rw} \tilde{\phi}_k - \mu_k \tilde{\phi}_k \|_{\tilde{D}}  \le  \text{Err}_{pt},
\quad 2 \le k \le K,
\quad \text{Err}_{pt} = O(\epsilon, \sqrt{\frac{\log N}{N \epsilon^{d/2+1}}} ).
\end{equation}
Meanwhile, because we are under $E_{UB}''$, 
\eqref{eq:eigen-stay-away} holds for $\lambda_k$.
The proof then proceeds in the same way as the Step 2. in Theorem \ref{thm:refined-rates-rw}, replacing $\frac{D}{N}$ with $\tilde{D}$. 
Specifically, let $S_k= \text{Span}\{ \tilde{D}^{1/2} v_k \}$, 
$S_k^{\perp} = \text{Span}\{  \tilde{D}^{1/2} v_j, \, j\neq k, 1 \le j \le N \}$. We then have
$P_{S_k^{\perp}} \left( \tilde{D}^{1/2}  \mu_k \tilde{\phi}_k  \right)  
= \tilde{D}^{1/2} \sum_{j \neq k, j=1}^N \frac{   v_j^T \tilde{D} \tilde{\phi}_k }{ \|  v_j\|_{ \tilde{D} }^2} \mu_k  v_j$,
and  because 
\begin{equation}\label{eq:rw-tildeD-adjoint-relation}
\tilde{L}_{rw}^T \tilde{D} v_j = \frac{1}{\epsilon}(I - \tilde{W} \tilde{D}^{-1}) \tilde{D} v_j 
= \frac{1}{\epsilon}( \tilde{D} - \tilde{W})  v_j = \tilde{D} \lambda_j v_j,
\end{equation}
we also have
$P_{S_k^{\perp}} \left( \tilde{D}^{1/2}  \tilde{L}_{rw} \tilde{\phi}_k  \right)  
 = 
 \tilde{D}^{1/2}  \sum_{j\neq k, j=1}^N
\frac{    v_j^T   \tilde{D}  \tilde{\phi}_k }{ \| v_j\|_{ \tilde{D}}^2}  \lambda_j  v_j $.
Take subtraction
$P_{S_k^{\perp}} \left( \tilde{D}^{1/2}  (  \tilde{L}_{rw} \tilde{\phi}_k -  \mu_k \tilde{\phi}_k ) \right) $
and do the same calculation as before,
by \eqref{eq:pontwise-rate-2norm-bound-rw-tildeD}, it gives that 
\begin{equation}\label{eq:project-tildeD-phik-small}
\| P_{S_k^{\perp}} \left(  \tilde{D}^{1/2} \tilde{\phi}_k  \right) \|_2
= \left(   \sum_{j\neq k, j=1}^N
\frac{  | v_j^T  \tilde{D} \tilde{\phi}_k|^2 }{ \| v_j\|_{ \tilde{D}}^2}  \right)^{1/2}
\le  \frac{ \text{Err}_{pt} }{\sqrt{N} \gamma_K}
= \frac{1}{\sqrt{N}}O(\epsilon, \sqrt{\frac{\log N}{N \epsilon^{d/2+1}}} ).
\end{equation}
We similarly define
$\beta_k := \frac{   v_k^T \tilde{D} \tilde{\phi}_k }{ \| v_k\|_{\tilde{D} }^2}$, $\beta_k    \tilde{D}^{1/2} v_k = P_{S_k} \tilde{D}^{1/2}\tilde{\phi}_k$, and
$P_{S_k^{\perp}} \left( \tilde{D}^{1/2} \tilde{\phi}_k  \right)
= \tilde{D}^{1/2}\tilde{\phi}_k - P_{S_k} \tilde{D}^{1/2}\tilde{\phi}_k
= \tilde{D}^{1/2} \left( \tilde{\phi}_k -  
\beta_k 
 v_k \right)$.
Then, by \eqref{eq:project-tildeD-phik-small},  we have
$\| \tilde{\phi}_k -  \beta_k  v_k \|_{\tilde{D} }
= \| P_{S_k^{\perp}} \left( \tilde{D}^{1/2} \tilde{\phi}_k  \right)\|_{2 }
= \frac{1}{\sqrt{N}} O(\epsilon, \sqrt{\frac{\log N}{N \epsilon^{d/2+1}}} )$,
 and by \eqref{eq:equivalent-2norm-tildeD-norm},
\[
\| \tilde{\phi}_k -  \beta_k  v_k \|_2 
= O(\epsilon, \sqrt{\frac{\log N}{N \epsilon^{d/2+1}}} ).
\]
To finish Step 2, it remains to show that $|\beta_k| = 1 + o(1)$, and then we define $\alpha_k = \frac{1}{\beta_k}$. Note that 
\begin{equation}\label{eq:betak2-relation}
\|  \tilde{\phi}_k \|_{ \tilde{D}}^2 
=  \|  \tilde{D}^{1/2} \tilde{\phi}_k \|_2^2 =
\| P_{S_k^{\perp}} \left( \tilde{D}^{1/2} \tilde{\phi}_k  \right)  \|_2^2
+ \|  P_{S_k} \left( \tilde{D}^{1/2} \tilde{\phi}_k  \right) \|_2^2 
 =  
\| P_{S_k^{\perp}} \left( \tilde{D}^{1/2} \tilde{\phi}_k  \right)  \|_2^2
+ \beta_k^2 \|  v_k \|_{\tilde{D}}^2.
\end{equation}
By that $\|  v_k \|_{\tilde{D}}^2 = \frac{1}{N}$,
inserting into \eqref{eq:betak2-relation} together with \eqref{eq:project-tildeD-phik-small}, \eqref{eq:norm-tilde-tildephik}, 
\[
\frac{1}{N } 
  ( 1 +  O(\epsilon,  \sqrt{ \frac{\log N}{N \epsilon^{d/2}} }))
= 
(\frac{1}{\sqrt{N}}O(\epsilon, \sqrt{\frac{\log N}{N \epsilon^{d/2+1}}} ))^2
+ \beta_k^2 \frac{1}{N},
\]
which gives that 
$1 +  o(1) = 
o(1) + \beta_k^2$ by multiplying $N$ to both sides. 
\\

\underline{Step 3. of $\tilde{L}_{rw}$}:
The proof is the same as Step 3. in Theorem \ref{thm:refined-rates-rw}, replacing $\frac{D}{N}$ with $\tilde{D}$.
Specifically, using the relation \eqref{eq:rw-tildeD-adjoint-relation},
and the eigenvector consistency in Step 2, we have
\[
 | \lambda_k - \mu_k|  |v_k^T \tilde{D} \tilde{\phi}_k|
\le 
|\alpha_k || \tilde{\phi}_k^T \tilde{D}  \tilde{L}_{rw} \tilde{\phi}_k - \mu_k \| \tilde{\phi}\|_{\tilde{D}}^2 | 
+ |\varepsilon_k^T \tilde{D} ( \tilde{L}_{rw} \tilde{\phi}_k - \mu_k \tilde{\phi}_k)|
=: \textcircled{1}+ \textcircled{2}.
\]
where $\|  \varepsilon_k  \|_{\tilde{D}}= \frac{1}{\sqrt{N}} O(\epsilon, \sqrt{\frac{\log N}{N \epsilon^{d/2+1}}} )$
and $\alpha_k = 1+o(1)$. 
By \eqref{eq:tildeform-rate-psi},
$\tilde{\phi}_k^T \tilde{D}  \tilde{L}_{rw} \tilde{\phi}_k =   \tilde{E}_N(\tilde{\phi}_k) =  \frac{1}{N} (\mu_k + O(\epsilon, \sqrt{  \frac{  \log N    }{ N \epsilon^{d/2  }}   })$.
Together with  \eqref{eq:norm-tilde-tildephik}, one can show that $N  \textcircled{1} = O(\epsilon,  \sqrt{ \frac{\log N}{N \epsilon^{d/2}} })$.
For $ \textcircled{2}$,  with \eqref{eq:pontwise-rate-2norm-bound-rw-tildeD}, one can verify that 
$
 \textcircled{2} \le  \| \varepsilon_k \|_{\tilde{D}}  \| \tilde{L}_{rw} \tilde{\phi}_k - \mu_k \tilde{\phi}_k \|_{\tilde{D}} 
 = \frac{1}{N}O( \text{Err}_{pt}^2) = \frac{O(\epsilon)}{N}$,
 where used that $O( \text{Err}_{pt}^2)  = O(\epsilon)$ same as before. 
 Putting together, and with the definition of $\beta_k$ above, 
\[
| \lambda_k - \mu_k|  |\beta_k|
\le \frac{  \textcircled{1} + \textcircled{2} }{\| v_k\|_{\tilde{D}}^2 }
= \frac{ (O(\epsilon,  \sqrt{ \frac{\log N}{N \epsilon^{d/2}} }) + O(\epsilon)   ) /N   }{  1/N }
= O(\epsilon,  \sqrt{ \frac{\log N}{N \epsilon^{d/2}} }).
\]
We have shown that $|\beta_k| = 1+o(1)$,
thus the bound of $| \lambda_k - \mu_k|$ is proved, and holds for $k \le k_{max}$.
\end{proof}

\end{document}